\documentclass{amsart}

\usepackage{amssymb}
\usepackage{amsmath}
\usepackage{braket}
\usepackage{mathrsfs}
\usepackage[all]{xy}
\usepackage{graphicx}
\usepackage{color}
\usepackage{transparent}

\newcommand{\C}{\Csc}               
\newcommand{\D}{\Dsc}               
\newcommand{\E}{\Ebb}               
\newcommand{\F}{\Fbb}               
\newcommand{\T}{\Tsc}               

\newcommand{\Ebb}{\mathbb{E}}        
\newcommand{\Fbb}{\mathbb{F}}        
\newcommand{\Wbb}{\mathbb{W}}        

\newcommand{\CEs}{(\C,\E,\sfr)}       
\newcommand{\wC}{\wt{\C}}             
\newcommand{\wE}{\wt{\Ebb}}           
\newcommand{\ws}{\wt{\sfr}}           
\newcommand{\wCEs}{(\wC,\wE,\ws)}     

\newcommand{\Mor}{\operatorname{Mor}}      
\newcommand{\Iso}{\operatorname{Iso}}      
\newcommand{\add}{\operatorname{add}}      
\newcommand{\Ob}{\operatorname{Ob}}        
\newcommand{\Ext}{\operatorname{Ext}}       
\newcommand{\Ker}{\operatorname{Ker}}       
\newcommand{\Cone}{\operatorname{Cone}}     
\newcommand{\CoCone}{\operatorname{CoCone}} 

\newcommand{\Ab}{\mathit{Ab}}        
\newcommand{\op}{\mathrm{op}}       

\newcommand{\ush}{^\sharp}           
\newcommand{\ssh}{_\sharp}           
\newcommand{\uas}{^{\ast}}            
\newcommand{\sas}{_{\ast}}            

\newcommand{\id}{\mathrm{id}}         
\newcommand{\Id}{\mathrm{Id}}         

\newcommand{\se}{\subseteq}           
\newcommand{\ppr}{^{\prime}}          
\newcommand{\pprr}{^{\prime\prime}}   
\newcommand{\co}{\colon}              
\newcommand{\ci}{\circ}               
\newcommand{\iv}{^{-1}}               

\newcommand{\lla}{\longleftarrow}     
\newcommand{\lra}{\longrightarrow}    
\newcommand{\tc}{\Rightarrow}         
\newcommand{\ltc}{\Longrightarrow}    
\newcommand{\EQ}{\Leftrightarrow}     
\newcommand{\dra}{\dashrightarrow}  

\newcommand{\wt}{\widetilde}          

\newcommand{\sfr}{\mathfrak{s}}  
\newcommand{\tfr}{\mathfrak{t}}  

\newcommand{\Csc}{\mathscr{C}}   
\newcommand{\Dsc}{\mathscr{D}}   
\newcommand{\Ssc}{\mathscr{S}}   
\newcommand{\Tsc}{\mathscr{T}}   

\newcommand{\Ical}{\mathcal{I}} 
\newcommand{\Kcal}{\mathcal{K}} 
\newcommand{\Lcal}{\mathcal{L}} 
\newcommand{\Mcal}{\mathcal{M}} 
\newcommand{\Ncal}{\mathcal{N}} 
\newcommand{\Rcal}{\mathcal{R}} 
\newcommand{\Scal}{\mathcal{S}} 
\newcommand{\Tcal}{\mathcal{T}} 
\newcommand{\Ucal}{\mathcal{U}} 
\newcommand{\Vcal}{\mathcal{V}} 
\newcommand{\abf}{\mathbf{a}}   
\newcommand{\bbf}{\mathbf{b}}   
\newcommand{\ubf}{\mathbf{u}}   
\newcommand{\vbf}{\mathbf{v}}   
\newcommand{\wbf}{\mathbf{w}}   
\newcommand{\ovl}{\overline}    
\newcommand{\ov}{\overset}      
\newcommand{\un}{\underset}     
\newcommand{\ti}{\times}        

\newcommand{\bsm}{\begin{smallmatrix}}
\newcommand{\esm}{\end{smallmatrix}}

\newcommand{\al}{\alpha}         
\newcommand{\be}{\beta}          
\newcommand{\gam}{\gamma}        
\newcommand{\lam}{\lambda}       
\newcommand{\thh}{\theta}        
\newcommand{\del}{\delta}        
\newcommand{\ep}{\varepsilon}    
\newcommand{\sig}{\sigma}        

\newcommand{\FR}[3]{{#1}\backslash {#2}/{#3}}

\numberwithin{equation}{section}

\newtheorem{thm}{Theorem}[section]
\newtheorem{cor}[thm]{Corollary}
\newtheorem{prop}[thm]{Proposition}
\theoremstyle{definition}
\newtheorem{dfn}[thm]{Definition}
\newtheorem{ex}[thm]{Example}
\newtheorem{claim}[thm]{Claim}
\newtheorem{lem}[thm]{Lemma}

\newtheorem{cond}[thm]{Condition}

\theoremstyle{remark}
\newtheorem{rem}[thm]{Remark}

\newtheorem*{introthm}{Theorem}

\begin{document}

\title{Localization of extriangulated categories}

\author{Hiroyuki Nakaoka}
\email{nakaoka.hiroyuki@math.nagoya-u.ac.jp} %
\address{Graduate School of Mathematics, Nagoya University, Furocho, Chikusaku, Nagoya 464-8602, Japan}

\author{Yasuaki Ogawa}
\email{ogawa.yasuaki.gh@cc.nara-edu.ac.jp} %
\address{Center for Educational Research of Science and Mathematics, Nara University of Education, Takabatake-cho, Nara, 630-8528, Japan}

\author{Arashi Sakai}
\email{m20019b@math.nagoya-u.ac.jp} %
\address{Graduate School of Mathematics, Nagoya University, Furocho, Chikusaku, Nagoya 464-8602, Japan}

\thanks{The authors are grateful for Mikhail Gorsky, Osamu Iyama, Kiriko Kato, Yann Palu and Katsunori Takahashi for their interest and valuable comments. This work is supported by JSPS KAKENHI Grant Number JP20K03532.}

\begin{abstract}
In this article, we show that the localization of an extriangulated category by a multiplicative system satisfying mild assumptions can be equipped with a natural, universal structure of an extriangulated category. This construction unifies the Serre quotient of abelian categories and the Verdier quotient of triangulated categories. Indeed we give such a construction for a bit wider class of morphisms, so that it covers several other localizations appeared in the literature, such as Rump's localization of exact categories by biresolving subcategories, localizations of extriangulated categories by means of Hovey twin cotorsion pairs, and the localization of exact categories by two-sided admissibly percolating subcategories.
\end{abstract}

\maketitle

\tableofcontents

\section{Introduction}
Abelian categories, exact categories and triangulated categories are the main categorical frameworks used in homological algebra.
The notion of an {\it extriangulated category} was recently introduced in \cite{NP} as a unification of such classes of categories. The class of extriangulated categories not only contains them as typical cases, but has an advantage that it is closed by several operations such as taking extension-closed subcategories \cite[Remark~2.18]{NP}, ideal quotients by subcategories consisting of projective-injectives \cite[Proposition 3.30]{NP}, and relative theories \cite[Proposition 3.16]{HLN}. 
If one names another fundamental operation which still lacks in extriangulated categories, it will be the \emph{localization}.
As we know, for abelian/triangulated categories, localization can be performed in a satisfactory generality. As a unifying notion, it is natural to expect that extriangulated categories provide their common generalization. 

In this article, let us discuss about localizations of extriangulated categories. As for the localizations which involve abelian/exact/triangulated categories, the following are known in the literature.
\begin{itemize}
\item[{\rm (i)}] Serre quotient of an abelian category \cite{Ga}.
\item[{\rm (ii)}] Verdier quotient of a triangulated category \cite{V}.
\item[{\rm (iii)}] C\'{a}rdenas-Escudero's localization of an exact category \cite{C-E}. More generally, localization of an exact category by a two-sided admissibly percolating subcategory \cite{HR},\cite{HKR}. 
\item[{\rm (iv)}] Rump's localization of an exact category \cite{R} by a biresolving subcategory.
\item[{\rm (v)}] Localization of an abelian category with respect to an abelian model structure  \cite{Ho1},\cite{Ho2}. More generally, localization of an exact category with respect to an exact model structure \cite{Gi}.
\item[{\rm (vi)}] As a counterpart of {\rm (v)}, localization of a triangulated category with respect to a triangulated model structure \cite{Y}.
\item[{\rm (vii)}] As a unification of {\rm (v)} and {\rm (vi)}, localization of an extriangulated category with respect to an admissible Hovey twin cotorsion pair \cite{NP}.
\end{itemize}

The localization of extriangulated categories which we will introduce in this article, gives a unification of all the above localizations. 
After briefly reviewing the definition and basic properties of extriangulated categories in Section~\ref{Section_ReviewOnExtri}, we introduce our main theorem (Theorem~\ref{ThmMultLoc}) in Section~\ref{Section_Localization}, which is as follows. In the below, $\CEs$ is an extriangulated category, $\Ncal_{\Ssc}\se\C$ is an additive full subcategory associated to $\Ssc$, and $\ovl{\Ssc}$ denotes a set of morphisms in the ideal quotient $\ovl{\C}=\C/[\Ncal_{\Ssc}]$ obtained from $\Ssc$ by taking closure with respect to the composition with isomorphisms in $\ovl{\C}$.

\begin{introthm}
Let $\Ssc$ be a set of morphisms in $\C$ containing all isomorphisms and closed by compositions. 
Suppose that $\ovl{\Ssc}$ satisfies the following conditions in $\ovl{\C}$.
\begin{itemize}
\item[{\rm (MR1)}] $\ovl{\Ssc}$ satisfies $2$-out-of-$3$ with respect to compositions.
\item[{\rm (MR2)}] $\ovl{\Ssc}$ is a multiplicative system.
\item[{\rm (MR3)}] Let $\langle A\ov{x}{\lra}B\ov{y}{\lra}C,\del\rangle$, $\langle A\ppr\ov{x\ppr}{\lra}B\ppr\ov{y\ppr}{\lra}C\ppr,\del\ppr\rangle$ be any pair of $\sfr$-triangles, and suppose that $a\in\C(A,A\ppr),c\in\C(C,C\ppr)$ satisfies $a\sas\del=c\uas\del\ppr$. If $\ovl{a},\ovl{c}\in\ovl{\Ssc}$, then there exists $\bbf\in\ovl{\Ssc}$ such that $\bbf\ci\ovl{x}=\ovl{x}\ppr\ci\ovl{a}$, $\ovl{c}\ci\ovl{y}=\ovl{y}\ppr\ci\bbf$.
\item[{\rm (MR4)}]  $\ovl{\Mcal}_{\mathsf{inf}}=\{ \vbf\ci \ovl{x}\ci \ubf\mid x\ \text{is an}\ \sfr\text{-inflation}, \ubf,\vbf\in\ovl{\Ssc}\}$ is closed by compositions. Dually for $\sfr$-deflations. 
\end{itemize}
Then the localization of $\C$ by $\Ssc$ gives an extriangulated category $(\wC,\wE,\ws)$ equipped with an exact functor $(Q,\mu)\co\CEs\to \wCEs$, which is universal among exact functors inverting $\Ssc$.
\end{introthm}
In fact, we show in Theorem~\ref{ThmMultLoc} that even without condition {\rm (MR4)}, we obtain a universal \emph{weakly} extriangulated category in the sense of \cite{BBGH}.
In the last Section~\ref{Section_Examples}, we will demonstrate how the above-mentioned localizations can be seen as particular cases of the localization given in this article, by showing in Propositions~\ref{PropSatisfy} and \ref{PropPerc} that the assumption of Theorem~\ref{ThmMultLoc} is indeed satisfied. More precisely, we roughly divide the above list into the following two cases. 
\begin{itemize}
\item[{\rm (A)}] Localizations obtained in {\rm (ii),(iv),(vii)} (and hence also {\rm (v),(vi)}). 
\item[{\rm (B)}] Localizations obtained in {\rm (i) (ii),(iii)}. 
\end{itemize}
This division results from particular properties of thick subcategories (Definition~\ref{DefThick}) used to give $\Ssc$. We remark that {\rm (ii)} sits in their intersection.
Case {\rm (A)} is dealt in Subsection~\ref{Subsection_LocTri} by using \emph{biresolving} thick subcategories (Definition~\ref{Def_BiResol}), while case {\rm (B)} is dealt in Subsection~\ref{Subsection_Percolating} by using \emph{percolating} thick subcategories (Definition~\ref{Def_Percolating}). In case {\rm (A)}, the resulting localization always becomes triangulated as will be shown in Corollary~\ref{CorLocTri}. On the other hand in case {\rm (B)}, with some additional condition which fits well with percolating subcategories, the resulting localization becomes exact as in Corollary~\ref{CorLast}.

\bigskip

Throughout this article, we use the following notations and terminology. For a category $\C$, let $\Mcal=\Mor(\C)$ denote the class of all morphisms of $\C$. Also, $\Iso(\C)\se\Mcal$ denotes the class of all isomorphisms. If a class of morphisms $\Ssc\se\Mcal$ is closed by compositions and contains all identities in $\C$, then we may regard $\Ssc\se\C$ as a (not full) subcategory satisfying $\Ob(\Ssc)=\Ob(\C)$. With this view in mind, we write $\Ssc(X,Y)=\{f\in \C(X,Y)\mid f\in\Ssc\}$ for any $X,Y\in\C$.
Throughout this article, let $\C$ denote an additive category. To avoid any set-theoretic problem in considering its localizations, we assume that $\C$ is small.

\section{Review on the definition of extriangulated category}\label{Section_ReviewOnExtri}

In this subsection, we review the definition of extriangulated categories. More precisely, we use instead the notion of $1$-exangulated category equivalent to it (\cite[Definition~2.32]{HLN}).

\begin{dfn}\label{DefExtension}
Suppose $\C$ is equipped with a biadditive functor $\E\co\C^\op\ti\C\to\Ab$. For any pair of objects $A,C\in\C$, an element $\del\in\E(C,A)$ is called an \emph{$\E$-extension}. 
We abbreviately express it as $C\ov{\del}{\dra}A$.

For any $a\in\C(A,A\ppr)$ and $c\in\C(C\ppr,C)$, we abbreviate
$\E(C,a)(\del)\in\E(C,A\ppr)$ and $\E(c,A)(\del)\in\E(C\ppr,A)$ to $a\sas\del$ and $c\uas\del$, respectively.
By Yoneda lemma, natural transformations
\[ \del\ssh\colon\C(-,C)\Rightarrow\E(-,A)\ \ \text{and}\ \ \del\ush\colon\C(A,-)\Rightarrow\E(C,-) \]
are associated to $\del$.
Explicitly, they are given by
\begin{eqnarray*}
&\del\ssh\co\C(X,C)\to\E(X,A)\ ;\ f\mapsto f\uas\del,&\\
&\del\ush\co\C(A,X)\to\E(C,X)\ ;\ g\mapsto g\sas\delta&
\end{eqnarray*}
for each $X\in\C$.
\end{dfn}

\begin{dfn}\label{DefEquiv2Seq}
Let $A,C\in\C$ be any pair of objects. Sequences $A\ov{x}{\lra}B\ov{y}{\lra}C$ and $A\ov{x\ppr}{\lra}B\ppr\ov{y\ppr}{\lra}C$ are said to be equivalent if there exists an isomorphism $b\in\C(B,B\ppr)$ such that $b\ci x=x\ppr$ and $y\ppr\ci b=y$. We denote the equivalence class to which $A\ov{x}{\lra}B\ov{y}{\lra}C$ belongs by $[A\ov{x}{\lra}B\ov{y}{\lra}C]$.
\end{dfn}

\begin{dfn}\label{DefConf}
Let $\C$ and $\Ebb$ be as before. Suppose that we are given a map $\sfr$ which assigns an equivalence class $\sfr(\del)=[A\ov{x}{\lra}B\ov{y}{\lra}C]$ to each $\Ebb$-extension $\del\in\Ebb(C,A)$. We use the following terminology.
\begin{enumerate}
\item A sequence $A\ov{x}{\lra}B\ov{y}{\lra}C$ in $\C$ is called an \emph{$\sfr$-conflation} if it satisfies $\sfr(\del)=[A\ov{x}{\lra}B\ov{y}{\lra}C]$ for some $\del\in\Ebb(C,A)$.

A morphism $x$ in $\C$ is called an \emph{$\sfr$-inflation} if it appears in some $\sfr$-conflation as $A\ov{x}{\lra}B\ov{y}{\lra}C$. 
Dually, a morphism $y$ in $\C$ is called an \emph{$\sfr$-deflation} if it appears in some $\sfr$-conflation as $A\ov{x}{\lra}B\ov{y}{\lra}C$. 
\item An \emph{$\sfr$-triangle} $\langle A\ov{x}{\lra}B\ov{y}{\lra}C,\del\rangle$ is a pair of a sequence $A\ov{x}{\lra}B\ov{y}{\lra}C$ in $\C$ and an $\Ebb$-extension $\del\in\Ebb(C,A)$ satisfying $\sfr(\del)=[A\ov{x}{\lra}B\ov{y}{\lra}C]$. We denote such a pair abbreviately by $A\ov{x}{\lra}B\ov{y}{\lra}C\ov{\del}{\dra}$.
\item A \emph{morphism of $\sfr$-triangles} from $\langle A\ov{x}{\lra}B\ov{y}{\lra}C,\del\rangle$ to $\langle A\ppr\ov{x\ppr}{\lra}B\ppr\ov{y\ppr}{\lra}C\ppr,\del\ppr\rangle$ is a triplet $(a,b,c)$ of morphisms in $\C$ which satisfies $b\ci x=x\ppr\ci a$, $c\ci y=y\ppr\ci b$ and $a\sas\del=c\uas\del\ppr$. We denote it abbreviately by a diagram as follows.
\begin{equation}\label{***abc}
\xy
(-12,6)*+{A}="0";
(0,6)*+{B}="2";
(12,6)*+{C}="4";
(24,6)*+{}="6";
(-12,-6)*+{A\ppr}="10";
(0,-6)*+{B\ppr}="12";
(12,-6)*+{C\ppr}="14";
(24,-6)*+{}="16";
{\ar^{x} "0";"2"};
{\ar^{y} "2";"4"};
{\ar@{-->}^{\del} "4";"6"};
{\ar_{a} "0";"10"};
{\ar_{b} "2";"12"};
{\ar^{c} "4";"14"};
{\ar_{x\ppr} "10";"12"};
{\ar_{y\ppr} "12";"14"};
{\ar@{-->}_{\del\ppr} "14";"16"};
{\ar@{}|\circlearrowright "0";"12"};
{\ar@{}|\circlearrowright "2";"14"};
\endxy
\end{equation}
\end{enumerate}
\end{dfn}

\begin{dfn}\label{DefEACat}
Let $\CEs$ be a triplet of an additive category $\C$, a biadditive functor $\E\co\C^\op\ti\C\to\Ab$, and a map $\sfr$ which assigns an equivalence class $\sfr(\del)=[A\ov{x}{\lra}B\ov{y}{\lra}C]$ to each $\Ebb$-extension $\del\in\Ebb(C,A)$.
\begin{enumerate}
\item $\CEs$ is a \emph{weakly extriangulated category} if it satisfies the following conditions. 
\begin{itemize}
\item[{\rm (C1)}] For any $\sfr$-triangle $A\ov{x}{\lra}B\ov{y}{\lra}C\ov{\del}{\dra}$, the sequence
\[ \C(X,A)\ov{x\ci-}{\lra}\C(X,B)\ov{y\ci-}{\lra}\C(X,C)\ov{\del\ssh}{\lra}\Ebb(X,A) \]
is exact for any $X\in\C$.
\item[{\rm (C1')}] Dually, for any $\sfr$-triangle $A\ov{x}{\lra}B\ov{y}{\lra}C\ov{\del}{\dra}$, the sequence
\[ \C(C,X)\ov{-\ci y}{\lra}\C(B,X)\ov{-\ci x}{\lra}\C(A,X)\ov{\del\ush}{\lra}\Ebb(C,X) \]
is exact for any $X\in\C$.
\item[{\rm (C2)}] For any $A\in\C$, zero element $0\in\E(0,A)$ satisfies
$\sfr(0)=[A\ov{\id_A}{\lra}A\to0]$.
\item[{\rm (C2')}] Dually, for any $A\in\C$, zero element $0\in\E(A,0)$ satisfies $\sfr(0)=[0\to A\ov{\id_A}{\lra}A]$.
\item[{\rm (C3)}] For any $\sfr$-triangle $A\ov{x}{\lra}B\ov{y}{\lra}C\ov{\del}{\dra}$, any $a\in\C(A,A\ppr)$ and any $\sfr$-triangle $A\ppr\ov{x\ppr}{\lra}B\ppr\ov{y\ppr}{\lra}C\ov{a\sas\del}{\dra}$, there exists $b\in\C(B,B\ppr)$ which gives a morphism of $\sfr$-triangles
\[
\xy
(-12,6)*+{A}="0";
(0,6)*+{B}="2";
(12,6)*+{C}="4";
(24,6)*+{}="6";
(-12,-6)*+{A\ppr}="10";
(0,-6)*+{B\ppr}="12";
(12,-6)*+{C}="14";
(24,-6)*+{}="16";
{\ar^{x} "0";"2"};
{\ar^{y} "2";"4"};
{\ar@{-->}^{\del} "4";"6"};
{\ar_{a} "0";"10"};
{\ar_{b} "2";"12"};
{\ar@{=} "4";"14"};
{\ar_{x\ppr} "10";"12"};
{\ar_{y\ppr} "12";"14"};
{\ar@{-->}_{a\sas\del} "14";"16"};
{\ar@{}|\circlearrowright "0";"12"};
{\ar@{}|\circlearrowright "2";"14"};
\endxy
\]
and makes
\[ A\ov{\left[\bsm x\\ a\esm\right]}{\lra}B\oplus A\ppr\ov{[b\ -x\ppr]}{\lra}B\ppr\ov{y^{\prime\ast}\del}{\dra} \]
an $\sfr$-triangle.

\item[{\rm (C3')}] For any $\sfr$-triangle $A\ov{x}{\lra}B\ov{y}{\lra}C\ov{\del}{\dra}$, any $c\in\C(C\ppr,C)$ and any $\sfr$-triangle $A\ov{x\ppr}{\lra}B\ppr\ov{y\ppr}{\lra}C\ppr\ov{c\uas\del}{\dra}$, there exists $b\in\C(B\ppr,B)$ which gives a morphism of $\sfr$-triangles
\[
\xy
(-12,6)*+{A}="0";
(0,6)*+{B\ppr}="2";
(12,6)*+{C\ppr}="4";
(24,6)*+{}="6";
(-12,-6)*+{A}="10";
(0,-6)*+{B}="12";
(12,-6)*+{C}="14";
(24,-6)*+{}="16";
{\ar^{x\ppr} "0";"2"};
{\ar^{y\ppr} "2";"4"};
{\ar@{-->}^{c\uas\del} "4";"6"};
{\ar@{=} "0";"10"};
{\ar_{b} "2";"12"};
{\ar^{c} "4";"14"};
{\ar_{x} "10";"12"};
{\ar_{y} "12";"14"};
{\ar@{-->}_{\del} "14";"16"};
{\ar@{}|\circlearrowright "0";"12"};
{\ar@{}|\circlearrowright "2";"14"};
\endxy
\]
and makes
\[ B\ppr\ov{\left[\bsm -y\ppr\\ b\esm\right]}{\lra}C\ppr\oplus B\ov{[c\ y]}{\lra}C\ov{x\ppr\sas\del}{\dra} \]
an $\sfr$-triangle.
\end{itemize}
\item $\CEs$ is an \emph{extriangulated category} if it is weakly extriangulated and moreover satisfies the following conditions.
\begin{itemize}
\item[{\rm (C4)}] $\sfr$-inflations are closed by compositions. Namely, if $A\ov{f}{\lra}A\ppr$ and $A\ppr\ov{f\ppr}{\lra}A\pprr$ are $\sfr$-inflations, then so is $f\ppr\ci f$.
\item[{\rm (C4')}] Dually, $\sfr$-deflations are closed by compositions.
\end{itemize}
\end{enumerate}
For a (weakly) extriangulated category $\CEs$, we call $\sfr$ a \emph{realization} of $\Ebb$.

For an $\sfr$-conflation $A\ov{x}{\lra}B\ov{y}{\lra}C$ in an extriangulated category, we write $C=\Cone(x)$ and call it a \emph{cone} of $x$. This is uniquely determined by $x$ up to isomorphisms. Dually we write $A=\CoCone(y)$ and call it a \emph{cocone} of $y$, which is uniquely determined by $y$ up to isomorphisms.
\end{dfn}

\begin{rem}\label{RemExTriEquiv}
It has been shown in \cite[Section~4.1]{HLN} that a triplet $\CEs$ satisfies the above conditions 
if and only if it is an extriangulated category defined in \cite[Definition~2.12]{NP}. By this reason, we call it simply an extriangulated category in this article.
In \cite{HLN}, an additional condition {\rm (R0)} is also required, which is negligible as it follows from {\rm (C3)} and {\rm (C3')}.

The notion of a weakly extriangulated category was introduced in \cite[Definition 5.15]{BBGH}, regarding its importance in the classification of additive subfunctors of the functor $\Ebb$ of an extriangulated category $\CEs$.
\end{rem}

\begin{rem}\label{RemWE}
Let $\CEs$ be a weakly extriangulated category, and let $(\ref{***abc})$ be any morphism of $\sfr$-triangles. If $a,c$ are isomorphisms, then so is $b$. In this case we have $\sfr(\del\ppr)=[A\ppr\ov{b\iv\ci x\ppr}{\lra}B\ov{y\ppr\ci b}{\lra}C\ppr]=[A\ppr\ov{x\ci a\iv}{\lra}B\ov{c\ci y}{\lra}C\ppr]$.
\end{rem}

\begin{rem}
As in Remark~\ref{RemExTriEquiv}, 
an extriangulated category $\CEs$ satisfies the following {\rm (ET4)} (and its dual {\rm (ET4)$^\op$}), which is one of the conditions in its original definition \cite[Definition~2.12]{NP}.

\begin{itemize}
\item[{\rm (ET4)}]
Let $A\ov{f}{\lra}B\ov{f\ppr}{\lra}D\ov{\del}{\dra}$ and $B\ov{g}{\lra}C\ov{g\ppr}{\lra}F\ov{\rho}{\dra}$ be any pair of $\sfr$-triangles.
Then there exists a diagram satisfying $d\uas\del\ppr=\del,e\uas\rho=f\sas\del\ppr$
\[
\xy
(-21,7)*+{A}="0";
(-7,7)*+{B}="2";
(7,7)*+{D}="4";
(-21,-7)*+{A}="10";
(-7,-7)*+{C}="12";
(7,-7)*+{E}="14";
(-7,-21)*+{F}="22";
(7,-21)*+{F}="24";
{\ar^{f} "0";"2"};
{\ar^{f\ppr} "2";"4"};
{\ar^{\del}@{-->} "4";(19,7)};
{\ar@{=} "0";"10"};
{\ar_{g} "2";"12"};
{\ar^{d} "4";"14"};
{\ar_{h=g\ci f} "10";"12"};
{\ar_{h\ppr} "12";"14"};
{\ar@{-->}^{\del\ppr} "14";(19,-7)};
{\ar_{g\ppr} "12";"22"};
{\ar^{e} "14";"24"};
{\ar@{=} "22";"24"};
{\ar@{-->}_{\rho} "22";(-7,-34)};
{\ar@{-->}^{f\ppr\sas\rho} "24";(7,-34)};
{\ar@{}|\circlearrowright "0";"12"};
{\ar@{}|\circlearrowright "2";"14"};
{\ar@{}|\circlearrowright "12";"24"};
\endxy
\]
in which $A\ov{h}{\lra}C\ov{h\ppr}{\lra}E\ov{\del\ppr}{\dra}$ and $D\ov{d}{\lra}E\ov{e}{\lra}F\ov{f\ppr\sas\rho}{\dra}$ are $\sfr$-triangles. 
\end{itemize}
\end{rem}

\begin{rem}\label{Ex_ExTri}
The following holds for an extriangulated category $\CEs$. (See \cite[Corollaries~3.18 and 7.6]{NP} for the detail.)
\begin{enumerate}
\item If any $\sfr$-inflation is monomorphic and any $\sfr$-deflation is epimorphic, then $\C$ has a natural structure of an exact category, in which admissible exact sequences are precisely given by $\sfr$-conflations. In this case we simply say that $\CEs$ \emph{corresponds to an exact category}, in this article.
\item If any morphism is both an $\sfr$-inflation and an $\sfr$-deflation, then  $\C$ has a natural structure of a triangulated category, in which distinguished triangles come from $\sfr$-triangles. In this case we simply say that $\CEs$ \emph{corresponds to a triangulated category}, in this article.
\end{enumerate}
\end{rem}

In this article, we sometimes refer to the following condition introduced in \cite[Condition~5.8]{NP}. 
\begin{cond}\label{ConditionWIC}
For a (weakly) extriangulated category $\CEs$, consider the following condition {\rm (WIC)}.
\begin{itemize}
\item[{\rm (WIC)}] For morphisms $h=g\ci f$ in $\C$, if $h$ is an $\sfr$-inflation, then $f$ is also an $\sfr$-inflation. Dually, if $h$ is an $\sfr$-deflation, then so is $g$.
\end{itemize}
\end{cond}

\begin{rem}
If an extriangulated category $\CEs$ corresponds to a triangulated category, then {\rm (WIC)} is always satisfied.
On the other hand, if $\CEs$ corresponds to an exact category, then it satisfies {\rm (WIC)} if and only if $\C$ is weakly idempotent complete (\cite[Proposition~7.6]{B}). (In \cite{R} and the reference therein, it is also called \emph{divisive}.)
\end{rem}

\begin{dfn}\label{DefExFun}
Let $\CEs$, $(\C\ppr,\E\ppr,\sfr\ppr)$ and $(\C\pprr,\E\pprr,\sfr\pprr)$ be weakly extriangulated categories.
\begin{enumerate}
\item  {\rm (\cite[Definition 2.23]{B-TS})}
An \emph{exact functor} $(F,\phi)\co\CEs\to(\C\ppr,\E\ppr,\sfr\ppr)$ is a pair of an additive functor $F\co\C\to\C\ppr$ and a natural transformation $\phi\co\E\Rightarrow\E\ppr\ci(F^\op\ti F)$ which satisfies
\[\sfr\ppr(\phi_{C,A}(\del))=[F(A)\ov{F(x)}{\lra}F(B)\ov{F(y)}{\lra}F(C)]\]
for any $\sfr$-triangle $A\ov{x}{\lra}B\ov{y}{\lra}C\ov{\del}{\dra}$ in $\C$.
\item If $(F,\phi)\co\CEs\to (\C\ppr,\E\ppr,\sfr\ppr)$ and $(F\ppr,\phi\ppr)\co(\C\ppr,\E\ppr,\sfr\ppr)\to (\C\pprr,\E\pprr,\sfr\pprr)$ are exact functors, then their composition $(F\pprr,\phi\pprr)=(F\ppr,\phi\ppr)\ci(F,\phi)$ is defined by
 $F\pprr=F\ppr\ci F$ and $\phi\pprr=(\phi\ppr\ci(F^\op\ti F))\cdot\phi$.

\item Let $(F,\phi),(G,\psi)\co \CEs\to (\C\ppr,\E\ppr,\sfr\ppr)$ be exact functors. A \emph{natural transformation} $\eta\co (F,\phi)\tc (G,\psi)$ \emph{of exact functors} is a natural transformation $\eta\co F\tc G$ of additive functors, which satisfies
\begin{equation}\label{P7}
(\eta_A)\sas\phi_{C,A}(\del)=(\eta_C)\uas\psi_{C,A}(\del)
\end{equation}
for any $\del\in\Ebb(C,A)$. Horizontal compositions and vertical compositions are defined by those for natural transformations of additive functors.
\end{enumerate}
\end{dfn}

\begin{rem}\label{RemExFun}
The above definition of an exact functor in {\rm (1)} is nothing but that of an \emph{extriangulated functor} introduced in \cite{B-TS}, applied to weakly extriangulated categories. 
The notion of an exact functor coincides with usual ones when both of $\CEs$ and $(\C\ppr,\E\ppr,\sfr\ppr)$ correspond to exact categories or to triangulated categories. Here we briefly recall how an exact functor ($=$ triangulated functor) between triangulated categories can be regarded as an exact functor of extriangulated categories in the above sense. See \cite[Theorem 2.33, 2.34]{B-TS} for the detail.

By definition, an exact functor between triangulated categories $(F,\xi)\co\T\to\T\ppr$ consists of an additive functor $F$ and a natural isomorphism $\xi\co F\ci [1]\ov{\cong}{\ltc}[1]\ci F$ which respects distinguished triangles. As shown in \cite[Proposition~3.22]{NP}, we may regard $\T$ as an extriangulated category $(\T,\E,\sfr)$, for which $\E$ is given by $\E=\Ext^1_{\T}=\T(-,-[1])$ and a sequence $A\ov{x}{\lra}B\ov{y}{\lra}C\ov{\del}{\dra}$ is an $\sfr$-triangle if and only if $A\ov{x}{\lra}B\ov{y}{\lra}C\ov{\del}{\lra}A[1]$ is a distinguished triangle in $\T$. Similarly for $\T\ppr$.
We see that $\xi$ induces a natural transformation $\phi\co\Ext^1_{\T}\tc \Ext_{\T\ppr}^1\ci(F^\op\ti F)$ defined by
\[ \phi_{C,A}\co \Ext^1_{\T}(C,A)\to \Ext_{\T\ppr}^1(FC,FA)\ ;\ \del\mapsto\xi_A\ci F(\del) \]
for each $A,C\in\C$. This gives an exact functor $(F,\phi)$ between extriangulated categories in the sense of Definition~\ref{DefExFun}.

The above definition of a natural transformation of exact functors in {\rm (3)} is automatically satisfied when $(\C\ppr,\E\ppr,\sfr\ppr)$ corresponds to an exact categories, and is equivalent to the definition of a morphism of triangulated functors in \cite[Definition 10.1.9]{KS} when both $\CEs$ and $(\C\ppr,\E\ppr,\sfr\ppr)$ correspond to triangulated categories.
\end{rem}

\begin{prop}\label{PropExEq}
Let $(F,\phi)\co\CEs\to(\D,\Fbb,\tfr)$ be an exact functor between weakly extriangulated categories. The following are equivalent.
\begin{enumerate}
\item $F$ is an equivalence of categories and $\phi$ is a natural isomorphism. 
\item $(F,\phi)$ is an equivalence of extriangulated categories, in the sense that there exist an exact functor $(G,\psi)\co(\D,\Fbb,\tfr)\to \CEs$, natural transformations of exact functors $(\Id_{\C},\id_{\E})\tc (G,\psi)\ci(F,\phi)$ and $(F,\phi)\ci (G,\psi)\tc(\Id_{\D},\id_{\F})$ which have inverses.
\end{enumerate}
\end{prop}
\begin{proof}
This is analogous to the usual argument for exact functors between triangulated categories.
As the converse is trivial, let us only show that {\rm (1)} implies {\rm (2)}. Suppose that $F$ is an equivalence of categories and let $G\co\D\to\C$ be a quasi-inverse of $F$, equipped with natural isomorphisms $\eta\co\Id_{\C}\ov{\cong}{\ltc} G\ci F$ and $\ep\co F\ci G\ov{\cong}{\ltc}\Id_{\D}$. We may choose $\eta$ and $\ep$ to be the unit and the counit of an adjoint pair $F\dashv G$.

Define $\psi$ to be the composition of
\[ \Fbb\ov{\Fbb\ci(\ep^\op\ti\ep\iv)}{\ltc}\Fbb\ci(F^\op\ti F)\ci(G^\op\ti G)
\ov{\phi\iv\ci(G^\op\ti G)}{\ltc}\Ebb\ci(G^\op\ti G). \]
By definition, for any $X,Z\in\D$, a homomorphism
\[ \psi_{Z,X}\co\Fbb(Z,X)\to\Ebb(GZ,GX) \]
sends each $\rho\in\Fbb(Z,X)$ to the unique element $\psi_{Z,X}(\rho)\in\Ebb(GZ,GX)$ which satisfies
\begin{equation}\label{Eq_phi_psi}
\phi_{GZ,GX}(\psi_{Z,X}(\rho))=(\ep_Z)\uas(\ep_X\iv)\sas\rho
\end{equation}
in $\Fbb(FGZ,FGX)$.

Let us show that $(G,\psi)\co(\D,\Fbb,\tfr)\to\CEs$ is an exact functor. Let $X\ov{x}{\lra}Y\ov{y}{\lra}Z\ov{\rho}{\dra}$ be any $\tfr$-triangle. 
By Remark~\ref{RemWE},
\begin{equation}\label{stri_compare1}
FGX\ov{FGx}{\lra}FGY\ov{FGy}{\lra}FGZ\ov{(\ep_Z)\uas(\ep_X\iv)\sas\rho}{\dra}
\end{equation}
is also a $\tfr$-triangle.
Put $\del=\psi_{Z,X}(\rho)\in\Ebb(GZ,GX)$, and realize it by an $\sfr$-triangle $GX\ov{f}{\lra}B\ov{g}{\lra}GZ\ov{\del}{\dra}$. Since $(F,\phi)$ is exact,
\begin{equation}\label{stri_compare2}
FGX\ov{Ff}{\lra}FB\ov{Fg}{\lra}FGZ\ov{\phi_{GZ,GX}(\del)}{\dra}
\end{equation}
becomes a $\tfr$-triangle. 
As $(\ref{Eq_phi_psi})$ suggests, $\tfr$-triangles $(\ref{stri_compare1})$ and $(\ref{stri_compare2})$ realize the same $\Fbb$-extension, and thus
\[ [FGX\ov{FGx}{\lra}FGY\ov{FGy}{\lra}FGZ]=[FGX\ov{Ff}{\lra}FB\ov{Fg}{\lra}FGZ] \]
holds as sequences in $\D$. Since $F$ is fully faithful, it is easy to deduce that
\[ [GX\ov{Gx}{\lra}GY\ov{Gy}{\lra}GZ]=[GX\ov{f}{\lra}B\ov{g}{\lra}GZ] \]
holds as sequences in $\C$. This means that $\sfr(\psi_{Z,X}(\rho))=[GX\ov{Gx}{\lra}GY\ov{Gy}{\lra}GZ]$ holds, and thus $(G,\psi)$ is exact.

It is immediate from the construction that $\ep\co(F,\phi)\ci (G,\psi)\tc(\Id_{\D},\id_{\F})$ is a natural transformation of exact functors. Indeed $(\ref{Eq_phi_psi})$ is nothing but the equality to be satisfied by $\ep$ in Definition~\ref{DefExFun} {\rm (3)}. Let us confirm that $\eta\co(\Id_{\C},\id_{\E})\tc (G,\psi)\ci(F,\phi)$ is a natural transformation of exact functors. Take any $\del\in\Ebb(C,A)$, and put $\sig=\phi_{C,A}(\del)\in\Fbb(FC,FA)$. 
It is enough to confirm that
\begin{equation}\label{ToShowpsi}
\psi_{FC,FA}(\sig)=(\eta_C\iv)\uas(\eta_A)\sas\del
\end{equation}
holds in $\Ebb(GFC,GFA)$.  
By definition $\psi_{FC,FA}(\sig)\in\Ebb(GFC,GFA)$ is the unique element which satisfies
\[ \phi_{GFC,GFA}(\psi_{FC,FA}(\sig))=(\ep_{FC})\uas(\ep_{FA}\iv)\sas(\sig) \]
in $\Ebb(FGFC,FGFA)$. Since $F\ci\eta=(\ep\ci F)\iv$ holds by the adjointness, we have
\[ (\ep_{FC})\uas(\ep_{FA}\iv)\sas(\sig)=(F(\eta_C)\iv)\uas(F(\eta_A))\sas(\sig)=(F(\eta_C)\iv)\uas(F(\eta_A))\sas(\phi_{C,A}(\del)). \]
By the naturality of $\phi$, we also have
\[ (F(\eta_C)\iv)\uas(F(\eta_A))\sas(\phi_{C,A}(\del))=\phi_{GFC,GFA}((\eta_C\iv)\uas(\eta_A)\sas\del). \]
Combining the above equalities, we obtain
\[ \phi_{GFC,GFA}(\psi_{FC,FA}(\sig))=\phi_{GFC,GFA}((\eta_C\iv)\uas(\eta_A)\sas\del). \]
Since $\phi_{GFC,GFA}$ is an isomorphism, this implies $(\ref{ToShowpsi})$.

It is also obvious from definition  that the inverses $\eta\iv\co(G,\psi)\ci(F,\phi)\tc(\Id_{\C},\id_{\E})$ and $\ep\iv\co (\Id_{\D},\id_{\F})\tc(F,\phi)\ci (G,\psi)$ are natural transformations of exact functors.
\end{proof}

\section{Localization}\label{Section_Localization}

In the rest of this article, let $\CEs$ be an extriangulated category.
In this section, let $\Ssc\se\Mcal$ be a set of morphisms which satisfies the following condition. 
\begin{itemize}
\item[{\rm (M0)}] $\Ssc$ contains all isomorphisms in $\C$, and is closed by compositions. Also, $\Ssc$ is closed by taking finite direct sums. Namely, if $f_i\in\Ssc(X_i,Y_i)$ for $i=1,2$, then $f_1\oplus f_2\in\Ssc(X_1\oplus X_1,Y_1\oplus Y_2)$.
\end{itemize}

Let $\wC$ denote the localization of $\C$ by $\Ssc$.
The aim of this section is to equip $\wC$ with a natural structure of an extriangulated category, under some assumptions.
First we associate a full subcategory $\Ncal_{\Ssc}\se\C$ in the following way.
\begin{dfn}\label{DefNS}
Let $\Ssc\se\Mcal$ be as above. Define $\Ncal_{\Ssc}\se\C$ to be the full subcategory consisting of objects $N\in\C$ such that both $N\to0$ and $0\to N$ belong to $\Ssc$. 

It is obvious that $\Ncal_{\Ssc}\se\C$ is an additive subcategory. In the rest, we will denote the ideal quotient by $p\co\C\to \ovl{\C}=\C/[\Ncal_{\Ssc}]$, and $\ovl{f}$ will denote a morphism in $\ovl{\C}$ represented by $f\in\C(X,Y)$. Also, let $\ovl{\Ssc}\se\ovl{\Mcal}=\Mor(\ovl{\C})$ be the closure of $p(\Ssc)$ with respect to compositions with isomorphisms in $\ovl{\C}$.
\end{dfn}

\begin{lem}\label{LemSplitN}
Let $\Ssc$ and $\ovl{\C}$ be as above. The following holds.
\begin{enumerate}
\item For any $f\in\Ssc(A,B)$ and any $i\in\C(A,N)$ with $N\in\Ncal_{\Ssc}$, we have $\left[\bsm f\\ i\esm\right]\in\Ssc(A,B\oplus N)$.
\item Suppose that $f,g\in\C(A,B)$ satisfy $\ovl{f}=\ovl{g}$ in $\ovl{\C}$. Then $f\in\Ssc$ holds if and only if $g\in\Ssc$.
\item The following conditions are equivalent.
\begin{itemize}
\item[{\rm (i)}] $p(\Ssc)=\ovl{\Ssc}$.
\item[{\rm (ii)}] $p\iv(\ovl{\Ssc})=\Ssc$.
\item[{\rm (iii)}] $p\iv(\Iso(\ovl{\C}))\se\Ssc$.
\item[{\rm (iv)}]  $f\in\Ssc$ holds for any split monomorphism $f\in\C(A,B)$ such that $\ovl{f}$ is an isomorphism in $\ovl{\C}$.
\item[{\rm (v)}] $f\in\Ssc$ holds for any split epimorphism $f\in\C(A,B)$ such that $\ovl{f}$ is an isomorphism in $\ovl{\C}$.
\end{itemize}
\end{enumerate}
\end{lem}
\begin{proof}
{\rm (1)} As $\Ssc$ is closed by finite direct sums, we have $\left[\bsm1\\0\esm\right]\in\Ssc(A,A\oplus N)$. This also implies $\left[\bsm1\\i\esm\right]\in\Ssc(A,A\oplus N)$, since
\[ \left[\begin{array}{cc}1&0\\i&1\end{array}\right]
\left[\begin{array}{c}1\\0\end{array}\right]=
\left[\begin{array}{c}1\\i\end{array}\right] \]
holds. Thus $\left[\bsm f\\i\esm\right]\in\Ssc(A,B\oplus N)$ follows from $\left[\bsm f\\i\esm\right]=(f\oplus \id_N)\ci \left[\bsm1\\i\esm\right]$.

{\rm (2)} Suppose that $f$ belongs to $\Ssc$. If $\ovl{f}=\ovl{g}$, there exist $N\in\Ncal_{\Ssc}$ and $i\in\C(A,N),j\in\C(N,B)$ such that $g=f+j\ci i$. By {\rm (1)} and its dual, we have $\left[\bsm f\\ i\esm\right]\in\Ssc(A,B\oplus N)$ and $[1\ j]\in\Ssc(B\oplus N,B)$. This implies $g=[1\ j]\left[\bsm f\\ i\esm\right]\in\Ssc$.

{\rm (3)} Remark that we always have $p(\Ssc)\se\ovl{\Ssc}$ and $p\iv(\ovl{\Ssc})\supseteq\Ssc$ by definition. {\rm (i)}\,$\EQ$\,{\rm (ii)}\,$\EQ$\,{\rm (iii)}
 follows from {\rm (2)} since $p$ is full.

Since {\rm (iii)}\,$\EQ$\,{\rm (v)} can be shown dually, it is enough to show {\rm (iii)}\,$\EQ$\,{\rm (iv)}.
As {\rm (iii)}\,$\tc$\,{\rm (iv)} is trivial, it remains to show {\rm (iv)}\,$\tc$\,{\rm (iii)}. Suppose that $\ovl{f}$ is an isomorphism in $\ovl{\C}$ for a morphism $f\in\C(A,B)$. Then there exist $N\in\Ncal_{\Ssc}$ and $i\in\C(A,N)$ such that $\left[\bsm f\\ i\esm\right]\in\C(A,B\oplus N)$ is a split monomorphism. Then {\rm (iv)} forces $\left[\bsm f\\ i\esm\right]\in\Ssc$, which implies $f=[1\ 0]\left[\bsm f\\ i\esm\right]\in\Ssc$.
\end{proof}

\begin{rem}\label{RemSplitN}
If $\C$ is idempotent complete, or more generally if any split monomorphism has a cokernel in $\C$ (namely, $\C$ is weakly idempotent complete), then any $\Ssc$ with {\rm (M0)} satisfies {\rm (iv)} and hence all the other equivalent conditions. In particular if $\CEs$ satisfies condition {\rm (WIC)}, then any $\Ssc$ with {\rm (M0)} should satisfy $p(\Ssc)=\ovl{\Ssc}$.
\end{rem}

\subsection{Statement of the main theorem}

One of our aim in Section~\ref{Section_Localization} is to show the following.
\begin{cor}\label{CorMultLoc}
Assume that $\Ssc\se\Mcal$ satisfies {\rm (M0)} as before, and moreover $p(\Ssc)=\ovl{\Ssc}$. Then the following holds.
\begin{enumerate}
\item Suppose that $\Ssc$ satisfies the following conditions {\rm (M1),(M2),(M3)}. Then we obtain a weakly extriangulated category $(\wC,\wE,\ws)$ together with an exact functor $(Q,\mu)\co\CEs\to\wCEs$.
\begin{itemize}
\item[{\rm (M1)}] $\Ssc\se\Mcal$ satisfies $2$-out-of-$3$ with respect to compositions in $\C$.
\item[{\rm (M2)}] $\Ssc$ is a multiplicative system in $\C$.
\item[{\rm (M3)}] Let $\langle A\ov{x}{\lra}B\ov{y}{\lra}C,\del\rangle$, $\langle A\ppr\ov{x\ppr}{\lra}B\ppr\ov{y\ppr}{\lra}C\ppr,\del\ppr\rangle$ be any pair of $\sfr$-triangles, and let $a\in\C(A,A\ppr),c\in\C(C,C\ppr)$ be any pair of morphisms satisfying $a\sas\del=c\uas\del\ppr$. If $a,c\in\Ssc$, then there exists $b\in\Ssc$ which gives the following morphism of $\sfr$-triangles.
\[
\xy
(-12,6)*+{A}="0";
(0,6)*+{B}="2";
(12,6)*+{C}="4";
(24,6)*+{}="6";
(-12,-6)*+{A\ppr}="10";
(0,-6)*+{B\ppr}="12";
(12,-6)*+{C\ppr}="14";
(24,-6)*+{}="16";
{\ar^{x} "0";"2"};
{\ar^{y} "2";"4"};
{\ar@{-->}^{\del} "4";"6"};
{\ar_{a} "0";"10"};
{\ar_{b} "2";"12"};
{\ar^{c} "4";"14"};
{\ar_{x\ppr} "10";"12"};
{\ar_{y\ppr} "12";"14"};
{\ar@{-->}_{\del\ppr} "14";"16"};
{\ar@{}|\circlearrowright "0";"12"};
{\ar@{}|\circlearrowright "2";"14"};
\endxy
\]
\end{itemize}

\item The exact functor $(Q,\mu)\co\CEs\to\wCEs$ obtained in {\rm (1)} is characterized  by the following universality.
\begin{itemize}
\item[{\rm (i)}]
For any exact functor $(F,\phi)\co(\C,\Ebb,\sfr)\to (\D,\Fbb,\tfr)$ such that $F(s)$ is an isomorphism for any $s\in\Ssc$, there exists a unique exact functor $(\wt{F},\wt{\phi})\co\wCEs\to (\D,\Fbb,\tfr)$ with $(F,\phi)=(\wt{F},\wt{\phi})\ci (Q,\mu)$.
\item[{\rm (ii)}] For any pair of exact functors $(F,\phi),(G,\psi)\co(\C,\Ebb,\sfr)\to (\D,\Fbb,\tfr)$ which send  any $s\in\Ssc$ to isomorphisms, let $(\wt{F},\wt{\phi}),(\wt{G},\wt{\psi})\co\wCEs\to (\D,\Fbb,\tfr)$ be the exact functors obtained in {\rm (i)}. Then for any natural transformation $\eta\co (F,\phi)\tc(G,\psi)$ of exact functors, there uniquely exists a natural transformation $\wt{\eta}\co (\wt{F},\wt{\phi})\tc(\wt{G},\wt{\psi})$ of exact functors satisfying $\eta=\wt{\eta}\ci (Q,\mu)$.
\end{itemize}

\item If $\Ssc$ moreover satisfies the following {\rm (M4)}, then $(\wC,\wE,\ws)$ is extriangulated.
\begin{itemize}
\item[{\rm (M4)}] $\Mcal_{\mathsf{inf}}:=\{ t\ci x\ci s\mid x\ \text{is an}\ \sfr\text{-inflation}, s,t\in\Ssc\}$ is closed by composition in $\Mcal$. 
Dually, $\Mcal_{\mathsf{def}}:=\{ t\ci y\ci s\mid y\ \text{is an}\ \sfr\text{-deflation}, s,t\in\Ssc\}\se\Mcal$ is closed by compositions.
\end{itemize}
\end{enumerate}
\end{cor}

More generally, in order to include several typical cases in mind (see Section~\ref{Section_Examples}), we will show the following, as our main theorem.
\begin{thm}\label{ThmMultLoc}
Assume that $\Ssc\se\Mcal$ satisfies {\rm (M0)}, as before.
\begin{enumerate}
\item Suppose that $\ovl{\Ssc}$ satisfies the following conditions {\rm (MR1),(MR2),(MR3)}.
Then we obtain a weakly extriangulated category $(\wC,\wE,\ws)$ together with an exact functor $(Q,\mu)\co\CEs\to\wCEs$.
\begin{itemize}
\item[{\rm (MR1)}] $\ovl{\Ssc}\se\ovl{\Mcal}$ satisfies $2$-out-of-$3$ with respect to compositions in $\ovl{\C}$.
\item[{\rm (MR2)}] $\ovl{\Ssc}$ is a multiplicative system in $\ovl{\C}$.
\item[{\rm (MR3)}] Let $\langle A\ov{x}{\lra}B\ov{y}{\lra}C,\del\rangle$, $\langle A\ppr\ov{x\ppr}{\lra}B\ppr\ov{y\ppr}{\lra}C\ppr,\del\ppr\rangle$ be any pair of $\sfr$-triangles, and let $a\in\C(A,A\ppr),c\in\C(C,C\ppr)$ be any pair of morphisms satisfying $a\sas\del=c\uas\del\ppr$. If $\ovl{a}$ and $\ovl{c}$ belong to $\ovl{\Ssc}$, then there exists $\bbf\in\ovl{\Ssc}(B,B\ppr)$ which satisfies $\bbf\ci\ovl{x}=\ovl{x}\ppr\ci\ovl{a}$ and $\ovl{c}\ci\ovl{y}=\ovl{y}\ppr\ci\bbf$.
\end{itemize}
\item The exact functor $(Q,\mu)\co\CEs\to\wCEs$ obtained in {\rm (1)} is characterized by the same universality as stated in Corollary~\ref{CorMultLoc} {\rm (2)}.

\item If $\ovl{\Ssc}$ moreover satisfies the following {\rm (MR4)}, then $(\wC,\wE,\ws)$ is extriangulated. 
\begin{itemize}
\item[{\rm (MR4)}] $\ovl{\Mcal}_{\mathsf{inf}}:=\{ \vbf\ci \ovl{x}\ci \ubf\mid x\ \text{is an}\ \sfr\text{-inflation}, \ubf,\vbf\in\ovl{\Ssc}\}$ is closed by composition in $\ovl{\Mcal}$. 
Dually, $\ovl{\Mcal}_{\mathsf{def}}:=\{ \vbf\ci \ovl{y}\ci \ubf\mid y\ \text{is an}\ \sfr\text{-deflation}, \ubf,\vbf\in\ovl{\Ssc}\}\se\ovl{\Mcal}$ is closed by compositions.
\end{itemize}
\end{enumerate}
\end{thm}

Proof of Theorem~\ref{ThmMultLoc} will be given in the end of this section.
In advance, let us confirm that Corollary~\ref{CorMultLoc} is indeed a corollary of Theorem~\ref{ThmMultLoc}. It suffices to show the following.

\begin{claim}\label{ClaimMultLoc}
Assume that $\Ssc$ satisfies {\rm (M0)} as before, and also $p(\Ssc)=\ovl{\Ssc}$. Then the following holds.
\begin{enumerate}
\item {\rm (M1)} is equivalent to {\rm (MR1)}.
\item {\rm (M2)} implies {\rm (MR2)}.
\item {\rm (M3)} implies {\rm (MR3)}.
\item {\rm (M4)} implies {\rm (MR4)}.
\end{enumerate}
\end{claim}
\begin{proof}
{\rm (1)} is immediate from $p(\Ssc)=\ovl{\Ssc}$ and Lemma~\ref{LemSplitN}.
Also {\rm (3),(4)} are obvious.

Let us show {\rm (2)}. Suppose that $\Ssc$ satisfies {\rm (M2)}.
Remark that $\ovl{\Ssc}$ is closed by compositions by definition.
To show {\rm (MR2)}, by duality it suffices to show the following.
\begin{itemize}
\item[{\rm (i)}] For any $\ovl{f}\in\ovl{\C}(A,B)$, if there exists $\ovl{s}\in\ovl{\Ssc}(A\ppr,A)$ satisfying $\ovl{f}\ci\ovl{s}=0$, then there exists $\ovl{t}\in\ovl{\Ssc}(B,B\ppr)$ such that $\ovl{t}\ci\ovl{f}=0$.
\item[{\rm (ii)}] For any $A\ppr\ov{\ovl{s}}{\lla}A\ov{\ovl{f}}{\lra}B$ with $\ovl{s}\in\ovl{\Ssc}$, there exists a commutative square
\[
\xy
(-6,6)*+{A}="0";
(6,6)*+{B}="2";
(-6,-6)*+{A\ppr}="4";
(6,-6)*+{B\ppr}="6";
{\ar^{\ovl{f}} "0";"2"};
{\ar_{\ovl{s}} "0";"4"};
{\ar^{\ovl{s}\ppr} "2";"6"};
{\ar_{\ovl{f}\ppr} "4";"6"};
{\ar@{}|\circlearrowright "0";"6"};
\endxy
\]
in $\ovl{\C}$ such that $\ovl{s}\ppr\in\ovl{\Ssc}$.
\end{itemize}
Since {\rm (ii)} is immediate from {\rm (M2)}, let us show {\rm (i)}. Suppose that $\ovl{f}\in\ovl{\C}(A,B)$ and $\ovl{s}\in\ovl{\Ssc}(A\ppr,A)$ satisfy $\ovl{f}\ci\ovl{s}=0$. Then there is $N\in\Ncal_{\Ssc}$ and $i\in\C(A\ppr,N),j\in\C(N,B)$ such that $f\ci s=j\ci i$. This means $[f\ -j]\ci \left[\bsm s\\ i\esm\right]=0$. Since $\left[\bsm s\\ i\esm\right]\in\Ssc$ by Lemma~\ref{LemSplitN}, there exists $t\in\Ssc(B,B\ppr)$ such that $t\ci [f\ -j]=0$ by {\rm (M2)}. In particular we have $t\ci f=0$. 
\end{proof}

The following is also an immediate corollary.
\begin{cor}\label{CorOfThm}
Let $(F,\phi)\co\CEs\to(\D,\Fbb,\tfr)$ be an exact functor to a weakly extriangulated category $(\D,\Fbb,\tfr)$. If $\Ssc=F\iv(\Iso(\D))$ satisfies {\rm (M2)}, or more generally if $\ovl{\Ssc}$ satisfies {\rm (MR2)},
then there exists a unique exact functor $(\wt{F},\wt{\phi})\co\wCEs\to(\D,\Fbb,\tfr)$ such that $(F,\phi)=(\wt{F},\wt{\phi})\ci (Q,\mu)$.
\end{cor}
\begin{proof}
Since $\Ssc$ satisfies {\rm (M0),(M1),(M3)} and $p\iv(\ovl{\Ssc})=\Ssc$ by construction, this will follow from Theorem~\ref{ThmMultLoc} {\rm (1),(2)} and Claim~\ref{ClaimMultLoc}.
\end{proof}

\begin{cor}\label{CorAdjoint}
Let $(F,\phi)\co\CEs\to(\D,\Fbb,\tfr)$ be an exact functor to a weakly extriangulated category $(\D,\Fbb,\tfr)$, and put $\Ssc=F\iv(\Iso(\D))$. 
If $F$ admits a fully faithful right adjoint and a fully faithful left adjoint,
then there exists a unique exact functor $(\wt{F},\wt{\phi})\co\wCEs\to(\D,\Fbb,\tfr)$ such that $(F,\phi)=(\wt{F},\wt{\phi})\ci (Q,\mu)$, and $\wt{F}$ is an equivalence.
\end{cor}
\begin{proof}
It is well-known for experts that {\rm (M2)} holds for such adjoint triple, e.g. \cite[Ch.I, Section 2]{GZ}.
Thus, the existence of a unique exact functor $(\wt{F},\wt{\phi})\co\wCEs\to(\D,\Fbb,\tfr)$ directly follows from Corollary \ref{CorOfThm}.

Obviously $\wt{F}$ is essentially surjective.
To show the faithfulness of $\wt{F}$, let $\al$ be a morphism in $\wC$ which is expressed as $\al=Q(s)\iv\ci Q(f)$ by a digram $X\ov{f}{\lra}Y\ov{s}{\lla} Y\ppr$ in $\C$ with $s\in\Ssc$.
If $\wt{F}(\alpha)=0$, then we get $F(f)=0$. Let $R$ be a right adjoint functor of $F$, and let $\eta$ be its unit. 
Note that $\eta_Y\co Y\to RF(Y)$ belongs to $\Ssc$, and thus $RF(f)=0$ implies $\eta_Y\ci f=0$.
Hence we have $Q(f)=0$.

It remains to show that $\wt{F}$ is full. Let $X,Y\in\wC$ be any pair of objects, and let $g\in\D(\wt{F}(X),\wt{F}(Y))$ be any morphism. 
If we put $\alpha:=Q(\eta_Y)^{-1}\ci QR(g)\ci Q(\eta_X)$, then we can easily check $\wt{F}(\alpha)=g$.
\end{proof}

\begin{rem}
In the situation of Corollary~\ref{CorAdjoint},
unlikely to the case of triangulated categories, a quasi-inverse $\wt{F}\iv\co\wt{\D}\to\wt{\C}$ of $\wt{F}$ can not in general be equipped with a natural transformation $\psi$ which makes $(\wt{F}\iv,\psi)\co (\D,\Fbb,\tfr)\to\wCEs$ an exact functor, unless $\wt{\phi}$ is a natural isomorphism.

A typical example for which $\wt{\phi}$ is not an isomorphism is given by a relative theory. For an extriangulated category $(\D,\Fbb,\tfr)$ and a closed subfunctor $\Ebb\se\Fbb$, we have an extriangulated category $\CEs$ with $\C=\D$ and $\sfr=\tfr|_{\Ebb}$, and an exact functor $(\Id,\phi)\co\CEs\to(\D,\Fbb,\tfr)$ is induced by the inclusion $\phi\co\Ebb\hookrightarrow\Fbb$. The localization in Corollary~\ref{CorAdjoint} does not essentially change $\CEs$, hence $\wt{\phi}$ is not an isomorphism unless $\Ebb=\Fbb$.
\end{rem}

\begin{rem}
A \emph{recollement} of extriangulated categories is defined in \cite{WWZ} as a diagram of functors between extriangulated categories of the following shape
\[
\xymatrix@C=1.2cm{\Ncal\ar[r]|-{{i_\ast}}
&\C\ar[r]|-{j^\ast}\ar@/^1.2pc/[l]^-{i^!}\ar_-{i^\ast}@/_1.2pc/[l]
&\D \ar@/^1.2pc/[l]^{j_\ast}\ar@/_1.2pc/[l]_{j_!}}
\]
satisfying some assumptions (see \cite[Definition~3.1]{WWZ} for the detail) which allow us to apply Corollary~\ref{CorAdjoint} to the right half of the above diagram.
\end{rem}

In the rest of this section, we proceed to show Theorem~\ref{ThmMultLoc}.
\begin{rem}\label{RemSaturation}
Remark that replacing $\Ssc$ by $p\iv(\ovl{\Ssc})$ does not change $\ovl{\Ssc}$, neither affect the resulting localization. As conditions {\rm (MR1)},\,$\ldots\,$,\,{\rm (MR4)} is only about $\ovl{\Ssc}$, and since condition {\rm (M0)} is stable under this replacement, we see that $p\iv(\ovl{\Ssc})$ satisfies the assumption of Theorem~\ref{ThmMultLoc} whenever $\Ssc$ does. Thus in proving Theorem~\ref{ThmMultLoc}, we may replace $\Ssc$ by $p\iv(\ovl{\Ssc})$ so that $\Ssc=p\iv(\ovl{\Ssc})$ is satisfied from the beginning.
\end{rem}

\subsection{Construction of $\wE$}

Suppose that $\Ssc$ satisfies {\rm (M0),(MR1),(MR2),(MR3)} and $\Ssc=p\iv(\ovl{\Ssc})$.
Remark that a localization $\C\to\wC$ factors uniquely through $p\co\C\to\ovl{\C}$, and we may regard $\wC$ as a localization of $\ovl{\C}$ by $\ovl{\Ssc}$. In other words, we may use the following functor $Q$ as a localization functor.
\begin{dfn}\label{DefQ}
Take a localization $\ovl{Q}\co\ovl{\C}\to\wC$ of $\ovl{\C}$ by the multiplicative system $\ovl{\Ssc}$, and put $Q=\ovl{Q}\ci p\co\C\to\wC$. This gives a localization of $\C$ by $\Ssc$.
\end{dfn}

In this subsection, we will construct a biadditive functor $\wE\co\wC^\op\ti\wC\to\Ab$ and a natural transformation $\mu\co\Ebb\tc\wE\ci(Q^\op\ti Q)$.
More in detail, we are going to define a biadditive functor $\ovl{\Ebb}$ and natural transformations $\ovl{\mu},\wp$ which fit in the following diagram,
\[
\xy
(-12,14)*+{\C^\op\ti\C}="0";
(-12,0)*+{\ovl{\C}^\op\ti\ovl{\C}}="2";
(-12,-14)*+{\wC^\op\ti\wC}="4";
(14,0)*+{\Ab}="6";
{\ar_{p^\op\ti p} "0";"2"};
{\ar_(0.46){\ovl{Q}^\op\ti \ovl{Q}} "2";"4"};
{\ar@/_1.5pc/_{Q^\op\ti Q} (-19,11.5);(-19,-11.5)};
{\ar^{\Ebb} "0";"6"};
{\ar^{\ovl{\Ebb}} "2";"6"};
{\ar_{\wE} "4";"6"};
{\ar@{=>}_{\wp} (-4,7);(-4,3)};
{\ar@{=>}_{\ovl{\mu}} (-4,-3);(-4,-7)};
{\ar@{}|\circlearrowright (-17,0);(-26,0)};
\endxy
\]
and will define $\mu$ by $\mu=\big(\ovl{\mu}\ci(p^\op\ti p)\big)\cdot\wp$.

\begin{lem}\label{LemExtVanish}
For any extension $\del\in\Ebb(C,A)$, the following are equivalent.
\begin{enumerate}
\item There exists $s\in\Ssc(A,A\ppr)$ such that $s\sas\del=0$.
\item There exists $t\in\Ssc(C\ppr,C)$ such that $t\uas\del=0$.
\end{enumerate}
\end{lem}
\begin{proof}
By duality, it is enough to show that {\rm (1)} implies {\rm (2)}. Remark that the equivalent conditions in Lemma~\ref{LemSplitN} {\rm (3)} are satisfied by our assumption. Let $A\ov{x}{\lra}B\ov{y}{\lra}C\ov{\del}{\dra}$ be an $\sfr$-triangle, and suppose that $s\in\Ssc(A,A\ppr)$ satisfies $s\sas\del=0$. By {\rm (MR3)}, there exists $u\in\Ssc(B,A\ppr\oplus C)$ which makes
\[
\xy
(-16,6)*+{A}="0";
(0,6)*+{B}="2";
(16,6)*+{C}="4";
(-16,-6)*+{A\ppr}="10";
(0,-6)*+{A\ppr\oplus C}="12";
(16,-6)*+{C}="14";
{\ar^{\ovl{x}} "0";"2"};
{\ar^{\ovl{y}} "2";"4"};
{\ar_{\ovl{s}} "0";"10"};
{\ar_{\ovl{u}} "2";"12"};
{\ar@{=} "4";"14"};
{\ar_(0.4){\left[\bsm1\\0\esm\right]} "10";"12"};
{\ar_(0.6){[0\ 1]} "12";"14"};
{\ar@{}|\circlearrowright "0";"12"};
{\ar@{}|\circlearrowright "2";"14"};
\endxy
\]
commutative in $\ovl{\C}$. By {\rm (MR2)}, we obtain a commutative square
\[
\xy
(-7,6)*+{B\ppr}="0";
(7,6)*+{C}="2";
(-7,-6)*+{B}="4";
(7,-6)*+{A\ppr\oplus C}="6";
{\ar^{\ovl{v}} "0";"2"};
{\ar_{\ovl{b}} "0";"4"};
{\ar^{\left[\bsm0\\1\esm\right]} "2";"6"};
{\ar_(0.4){\ovl{u}} "4";"6"};
{\ar@{}|\circlearrowright "0";"6"};
\endxy
\]
in $\ovl{\C}$ such that $v\in\Ssc$. Then we have $\ovl{v}=\ovl{y\ci b}$, which forces $y\ci b\in\Ssc$ by Lemma~\ref{LemSplitN} {\rm (2)}. Thus $t=y\ci b$ satisfies the required properties.
\end{proof}

\begin{dfn}\label{DefK}
Define a subset $\Kcal(C,A)\se\Ebb(C,A)$ by
\begin{eqnarray*}
\Kcal(C,A)&=&\{\del\in\Ebb(C,A)\mid s\sas\del=0\ \text{for some}\ s\in\Ssc(A,A\ppr)\}\\
&=&\{\del\in\Ebb(C,A)\mid t\uas\del=0\ \text{for some}\ t\in\Ssc(C\ppr,C)\}
\end{eqnarray*}
for each $A,C\in\C$.
By Lemma~\ref{LemExtVanish}, these form a subfunctor $\Kcal\se\Ebb$.
\end{dfn}

\begin{prop}\label{PropK}
The following holds.
\begin{enumerate}
\item $\Kcal\se\Ebb$ is an additive subfunctor.
\item $\Ebb/\Kcal\co\C^\op\ti\C\to\Ab$ is a biadditive functor which satisfies $(\Ebb/\Kcal)(N,-)=0$ and $(\Ebb/\Kcal)(-,N)=0$ for any $N\in\Ncal_{\Ssc}$.
\end{enumerate}
\end{prop}
\begin{proof}
{\rm (1)} For any $\del\in\Kcal(C,A)$ and $\del\ppr\in\Kcal(C\ppr,A\ppr)$, obviously we have $\del\oplus\del\ppr\in\Kcal(C\oplus C\ppr,A\oplus A\ppr)$. Here $\del\oplus\del\ppr\in\Ebb(C\oplus C\ppr,A\oplus A\ppr)$ is the element corresponding to $(\del,0,0,\del\ppr)$ through the isomorphism $\Ebb(C\oplus C\ppr,A\oplus A\ppr)\cong\Ebb(C,A)\oplus\Ebb(C,A\ppr)\oplus\Ebb(C\ppr,A)\oplus\Ebb(C\ppr,A\ppr)$ induced by the biadditivity of $\Ebb$. This shows the additivity of $\Kcal\se\Ebb$.
{\rm (2)} is immediate from {\rm (1)} and the definition.
\end{proof}

In the rest, for any $\del\in\Ebb(X,Y)$, let $\ovl{\del}\in(\Ebb/\Kcal)(X,Y)$ denote an element represented by $\del$. 

\begin{dfn}\label{DefEbar}
By Proposition~\ref{PropK}, we obtain a well-defined biadditive functor $\ovl{\Ebb}\co\ovl{\C}^\op\ti\ovl{\C}\to\Ab$ given by the following.
\begin{itemize}
\item $\ovl{\Ebb}(C,A)=(\Ebb/\Kcal)(C,A)$ for any $A,C\in\C$.
\item $\ovl{a}\sas\ovl{\del}=\ovl{a\sas\del}$ for any $\ovl{\del}\in\ovl{\Ebb}(C,A)$ and $\ovl{a}\in\ovl{\C}(A,A\ppr)$.
\item $\ovl{c}\uas\ovl{\del}=\ovl{c\uas\del}$ for any $\ovl{\del}\in\ovl{\Ebb}(C,A)$ and $\ovl{c}\in\ovl{\C}(C\ppr,C)$.
\end{itemize}
Also, we have a natural transformation $\wp\co\Ebb\tc\ovl{\Ebb}\ci(p^\op\ti p)$ given by
\[ \wp_{C,A}\co\Ebb(C,A)\to\ovl{\Ebb}(C,A)\ ;\ \del\mapsto\ovl{\del} \]
for each $A,C\in\C$. We will often abbreviate an $\ovl{\Ebb}$-extension $\ovl{\del}\in\ovl{\Ebb}(C,A)$ to $C\ov{\ovl{\del}}{\dra}A$, as in Definition~\ref{DefExtension}.

\end{dfn}

\begin{rem}\label{RemTrivK}
By definition, the following are equivalent for any $\del\in\Ebb(C,A)$.
\begin{enumerate}
\item $\ovl{\del}=0$ holds in $\ovl{\Ebb}(C,A)$.
\item $\ovl{s}\sas \ovl{\del}=0$ holds in $\ovl{\Ebb}(C,A\ppr)$ for some/any $\ovl{s}\in\ovl{\Ssc}(A,A\ppr)$.
\item $\ovl{t}\uas\ovl{\del}=0$ holds in $\ovl{\Ebb}(C\ppr,A)$ for some/any $\ovl{t}\in\ovl{\Ssc}(C\ppr,C)$.
\end{enumerate}
\end{rem}

We are going to construct a biadditive functor $\wE\co\wC^\op\ti\wC\to\Ab$ and a natural transformation $\ovl{\mu}\co\ovl{\Ebb}\tc\wE\ci(\ovl{Q}^\op\ti\ovl{Q})$. Though this can be substantialized by using coends as $\wC\un{\ovl{\C}}{\otimes}\ovl{\Ebb}\un{\ovl{\C}}{\otimes}\wC$, we give an explicit description of $\wE$ in the below, which enables us to construct a realization $\ws$ of $\wE$ in the next subsection.
In fact, we will define $\wE(C,A)$ to be a set of equivalence classes of triplets $(C\ov{\ovl{t}}{\lla}Z\ov{\ovl{\del}}{\dra}X\ov{\ovl{s}}{\lla}A)$ with respect to some equivalence relation. In its definition (Definition~\ref{DefwE}), we involve the following Proposition~\ref{PropCommonDenom}.

\begin{lem}\label{LemCommonDenom}
For any finite family of morphisms $\{\ovl{s}_i\in\ovl{\Ssc}(A,X_i)\}_{1\le i\le n}$, there exist $X\in\C$, $\ovl{s}\in\ovl{\Ssc}(A,X)$ and $\{\ovl{u}_i\in\ovl{\Ssc}(X_i,X)\}_{1\le i\le n}$ such that $\ovl{s}=\ovl{u}_i\ci\ovl{s}_i$ for any $1\le i\le n$. Dually for a family of morphisms $\{\ovl{t}_i\in\ovl{\Ssc}(Z_i,C)\}_{1\le i\le n}$.
\end{lem}
\begin{proof}
This follows from {\rm (MR1)} and {\rm (MR2)}, by a usual argument on multiplicative systems.
\end{proof}

\begin{prop}\label{PropCommonDenom}
Let $n\ge2$ be an integer. Let $A,C\in\C$ be any pair of objects, and let
\[ (\FR{\ovl{t}_i}{\ovl{\del}_i}{\ovl{s}_i})=(C\ov{\ovl{t}_i}{\lla}Z_i\ov{\ovl{\del}_i}{\dra}X_i\ov{\ovl{s}_i}{\lla}A)\qquad(i=1,2,\ldots,n) \]
be triplets with 
$\ovl{s}_i\in\ovl{\Ssc}(A,X_i), \ovl{t}_i\in\ovl{\Ssc}(Z_i,C), \ovl{\del}_i\in\ovl{\Ebb}(Z_i,X_i)$. The following holds.
\begin{enumerate}
\item We can take a \emph{common denominator} $\ovl{s},\ovl{t}\in\ovl{\Ssc}$. 
More precisely, there exist 
common $X,Z\in\C$, $\ovl{s}\in\ovl{\Ssc}(A,X)$ and $\ovl{t}\in\ovl{\Ssc}(Z,C)$, equipped with
$\ovl{u}_i\in\ovl{\Ssc}(X_i,X)$, $\ovl{v}_i\in\ovl{\Ssc}(Z,Z_i)$ 
which satisfy $\ovl{s}=\ovl{u}_i\ci\ovl{s}_i$ and $\ovl{t}=\ovl{t}_i\ci\ovl{v}_i$ as depicted in the following diagram for $1\le i\le n$,
\begin{equation}\label{DiagCommonDenom}
\xy
(-23,-6)*+{C}="C";
(-7,6)*+{Z_i}="1";
(-7,-6)*+{Z}="3";
(7,6)*+{X_i}="11";
(7,-6)*+{X}="13";
(23,-6)*+{A}="A";
{\ar_{\ovl{t}_i} "1";"C"};
{\ar^{\ovl{t}} "3";"C"};
{\ar@{-->}^{\ovl{\del}_i} "1";"11"};
{\ar@{-->}_{\ovl{\rho}_i} "3";"13"};
{\ar_{\ovl{s}_i} "A";"11"};
{\ar^{\ovl{s}} "A";"13"};
{\ar_{\ovl{v}_i} "3";"1"};
{\ar_{\ovl{u}_i} "11";"13"};
{\ar@{}|\circlearrowright (-12,0);(-10,-4)};
{\ar@{}|\circlearrowright (12,0);(10,-4)};
\endxy
\end{equation}
where we put $\ovl{\rho}_i=\ovl{v}_i\uas\ovl{u}_{i\ast}\ovl{\del}_i$.
When $n=2$, we write $(\FR{\ovl{t}_1}{\ovl{\del}_1}{\ovl{s}_1})\sim(\FR{\ovl{t}_2}{\ovl{\del}_2}{\ovl{s}_2})$ if they satisfy $\ovl{\rho}_1=\ovl{\rho}_2$. 
\item The relation $(\FR{\ovl{t}_1}{\ovl{\del}_1}{\ovl{s}_1})\sim(\FR{\ovl{t}_2}{\ovl{\del}_2}{\ovl{s}_2})$ is independent of the choice of common denominators. Namely, for any other choice of 
\[ (\FR{\ovl{t}\ppr}{\ovl{\rho}\ppr_i}{\ovl{s}\ppr})=(C\ov{\ovl{t}\ppr}{\lla}Z\ppr\ov{\ovl{\rho}\ppr_i}{\dra}X\ppr\ov{\ovl{s}\ppr}{\lla}A) \]
and $\ovl{u}\ppr_i\in\ovl{\Ssc}(X_i,X\ppr)$, $\ovl{v}\ppr_i\in\ovl{\Ssc}(Z\ppr,Z_i)$, $\ovl{\rho}\ppr_i\in\ovl{\Ebb}(Z\ppr,X\ppr)$ satisfying $\ovl{s}\ppr=\ovl{u}\ppr_i\ci\ovl{s}_i$, $\ovl{t}\ppr=\ovl{t}_i\ci\ovl{v}\ppr_i$ and $\ovl{\rho}\ppr_i=\ovl{v}_i^{\prime\ast}\ovl{u}\ppr_{i\ast}\ovl{\del}_i$ for $i=1,2$, we have $\ovl{\rho}_1=\ovl{\rho}_2$ if and only if $\ovl{\rho}_1\ppr=\ovl{\rho}_2\ppr$.
\end{enumerate}
\end{prop}

\begin{proof}
{\rm (1)} This is immediate from Lemma~\ref{LemCommonDenom}.

{\rm (2)}
Assume that there exists a diagram (\ref{DiagCommonDenom}) with $\ovl{\rho}_1=\ovl{\rho}_2$.
For another common denominator $\ovl{s}\ppr, \ovl{t}\ppr\in\ovl{\Ssc}$ given in the assertion (2), we shall show $\ovl{\rho}\ppr_1=\ovl{\rho}\ppr_2$.
For triplets 
$(\FR{\ovl{t}}{\ovl{\rho}_i}{\ovl{s}})$ and $(\FR{\ovl{t}\ppr}{\ovl{\rho}\ppr_i}{\ovl{s}\ppr})$, by Lemma~\ref{LemCommonDenom},
there exists a common denominator $\ovl{s}\pprr, \ovl{t}\pprr\in\ovl{\Ssc}$ together with the following commutative diagrams
\[
\xymatrix{
Z_i\ar[rrd]_{\ovl{t}_i}&Z\ar@{}[d]|( .3)\circlearrowright\ar[l]_{\ovl{v}_i}\ar[dr]|{\ovl{t}}&Z\pprr\ar@{}[dr]|( .3)\circlearrowright\ar@{}[dl]|( .3)\circlearrowright\ar[l]_{\ovl{v}}\ar[r]^{\ovl{v}\ppr}\ar[d]^{\ovl{t}\pprr}&Z\ppr\ar@{}[d]|( .3)\circlearrowright\ar[dl]|{\ovl{t}\ppr}\ar[r]^{\ovl{v}\ppr_i}&Z_i\ar[lld]^{\ovl{t}_i}&&A\ar[ld]|{\ovl{s}}\ar[d]^{\ovl{s}\pprr}\ar[rd]|{\ovl{s}\ppr}\ar[rrd]^{\ovl{s}_i}\ar[lld]_{\ovl{s}_i}&&\\
&&C&&X_i\ar[r]_{\ovl{u}_i}&X\ar@{}[u]|( .3)\circlearrowright\ar[r]_{\ovl{u}}&X\pprr\ar@{}[ur]|( .3)\circlearrowright\ar@{}[ul]|( .3)\circlearrowright&X\ppr\ar@{}[u]|( .3)\circlearrowright\ar[l]^{\ovl{u}\ppr}&X_i\ar[l]^{\ovl{u}\ppr_i}
}
\]
in which all morphisms belong to $\ovl{\Ssc}$.
Composing some morphisms $Z^{\prime\prime\prime}\to Z\pprr$ and $X\pprr\to X^{\prime\prime\prime}$ in $\ovl{\Ssc}$ if necessary, we may assume that $\ovl{v_iv}=\ovl{v\ppr_iv\ppr}$ and $\ovl{uu_i}=\ovl{u\ppr u\ppr_i}$ hold for $i=1,2$ from the first.
Then we have the following equations
\begin{eqnarray*}
\ovl{v}\uas\ovl{u}_{\ast}\ovl{\rho}_i\ =\ \ovl{v}\uas\ovl{u}_{\ast}(\ovl{v}_i\uas\ovl{u}_{i\ast}\ovl{\del}_i)&=&(\ovl{v_iv})\uas(\ovl{uu_i})_\ast\ovl{\del}_i\\
&=&(\ovl{v\ppr_iv\ppr})\uas(\ovl{u\ppr u\ppr_i})_\ast\ovl{\del}_i\ =\ 
{\ovl{v}\ppr}\uas\ovl{u}\ppr_{\ast}({\ovl{v}\ppr_i}\uas{\ovl{u}\ppr_i}_{\ast}\ovl{\del}_i)\ =\ {\ovl{v}\ppr}\uas\ovl{u}\ppr_{\ast}\ovl{\rho}\ppr_i
\end{eqnarray*}
for $i=1,2$.
By the assumption, we have ${\ovl{v}\ppr}\uas\ovl{u}\ppr_{\ast}\ovl{\rho}\ppr_1={\ovl{v}\ppr}\uas\ovl{u}\ppr_{\ast}\ovl{\rho}\ppr_2$ and particularly ${\ovl{v}\ppr}\uas\ovl{u}\ppr_{\ast}(\ovl{\rho}\ppr_1-\ovl{\rho}\ppr_2)=0$.
Remark~\ref{RemTrivK} shows $\ovl{\rho}\ppr_1=\ovl{\rho}\ppr_2$.
\end{proof}

\begin{cor}\label{CorwETrans}
The above $\sim$ gives an equivalence relation on the set of triplets
\begin{equation}\label{wE}
\{ (\FR{\ovl{t}}{\ovl{\del}}{\ovl{s}})=(C\ov{\ovl{t}}{\lla}Z\ov{\ovl{\del}}{\dra}X\ov{\ovl{s}}{\lla}A) \mid X,Z\in\C,\, \ovl{s},\ovl{t}\in\ovl{\Ssc},\,\ovl{\del}\in\ovl{\Ebb}(Z,X)\}
\end{equation}
for each $A,C\in\C$. 
\end{cor}
\begin{proof}
Transitivity can be shown by using Proposition~\ref{PropCommonDenom} {\rm (2)}.
In fact, if there exist triplets with $(\FR{\ovl{t}_1}{\ovl{\del}_1}{\ovl{s}_1})\sim(\FR{\ovl{t}_2}{\ovl{\del}_2}{\ovl{s}_2})\sim(\FR{\ovl{t}_3}{\ovl{\del}_3}{\ovl{s}_3})$,
by Proposition~\ref{PropCommonDenom} {\rm (1)}, we have a diagram (\ref{DiagCommonDenom}) for $i=1,2,3$.
Proposition~\ref{PropCommonDenom} {\rm (2)} forces $\ovl{\rho_1}=\ovl{\rho_2}=\ovl{\rho_3}$.
The other conditions can be checked easily.
\end{proof}

\begin{dfn}\label{DefwE}
Let $A,C\in\C$ be any pair of objects. Define $\wE(C,A)$ to be the quotient set of $(\ref{wE})$ by the equivalence relation $\sim$ obtained above.
We denote the equivalence class of $(\FR{\ovl{t}}{\ovl{\del}}{\ovl{s}})=(C\ov{\ovl{t}}{\lla}Z\ov{\ovl{\del}}{\dra}X\ov{\ovl{s}}{\lla}A)$ by $[\FR{\ovl{t}}{\ovl{\del}}{\ovl{s}}]=[C\ov{\ovl{t}}{\lla}Z\ov{\ovl{\del}}{\dra}X\ov{\ovl{s}}{\lla}A]\in\wE(C,A)$. By definition we have
\[ \wE(C,A)=\{ [\FR{\ovl{t}}{\ovl{\del}}{\ovl{s}}] \mid X,Z\in\C,\, \ovl{s},\ovl{t}\in\ovl{\Ssc},\,\ovl{\del}\in\ovl{\Ebb}(Z,X)\}. \]
\end{dfn}

In the following arguments, we frequently use the following.
\begin{rem}
By {\rm (MR2)}, any morphism $\al\in\wC(A,B)$ can be expressed as
\[ \al=\ovl{Q}(\ovl{s})\iv\ci\ovl{Q}(\ovl{f})=\ovl{Q}(g)\ci\ovl{Q}(\ovl{t})\iv \]
by some pairs of morphisms $A\ov{\ovl{f}}{\lra}B\ppr\ov{\ovl{s}}{\lla}B$ and $A\ov{\ovl{t}}{\lla}A\ppr\ov{\ovl{g}}{\lra}B$ satisfying $\ovl{s},\ovl{t}\in\ovl{\Ssc}$.
\end{rem}

\begin{dfn}\label{DefwEFtr}
Let $[\FR{\ovl{t}}{\ovl{\del}}{\ovl{s}}]=[C\ov{\ovl{t}}{\lla}Z\ov{\ovl{\del}}{\dra}X\ov{\ovl{s}}{\lla}A]\in\wE(C,A)$ be any element.
\begin{enumerate}
\item For any morphism $\al\in\wC(A,A\ppr)$, express it as $\al=\ovl{Q}(\ovl{u})\iv\ci\ovl{Q}(\ovl{a})$ with some $\ovl{a}\in\ovl{\C}(A,D)$ and $\ovl{u}\in\ovl{\Ssc}(A\ppr,D)$. Then take a commutative square
\begin{equation}\label{DD}
\xy
(-6,6)*+{X}="0";
(6,6)*+{X\ppr}="2";
(-6,-6)*+{A}="4";
(6,-6)*+{D}="6";
{\ar^{\ovl{a}\ppr} "0";"2"};
{\ar^{\ovl{s}} "4";"0"};
{\ar_{\ovl{s}\ppr} "6";"2"};
{\ar_{\ovl{a}} "4";"6"};
{\ar@{}|\circlearrowright "0";"6"};
\endxy
\end{equation}
in $\ovl{\C}$ with $\ovl{s}\ppr\in\ovl{\Ssc}$, and put
$\al\sas [\FR{\ovl{t}}{\ovl{\del}}{\ovl{s}}]=[\FR{\ovl{t}}{\ovl{a\ppr\sas\del}}{\ovl{s\ppr\ci u}}]$. 
\item For any morphism $\gam\in\wC(C\ppr,C)$, express it as $\gam=\ovl{Q}(\ovl{c})\ci\ovl{Q}(\ovl{v})\iv$ with some $\ovl{c}\in\ovl{\C}(E,C)$ and $\ovl{v}\in\ovl{\Ssc}(E,C\ppr)$. Then take a commutative square
\[
\xy
(-6,6)*+{E\ppr}="0";
(6,6)*+{Z}="2";
(-6,-6)*+{E}="4";
(6,-6)*+{C}="6";
{\ar^{\ovl{c}\ppr} "0";"2"};
{\ar_{\ovl{t}\ppr} "0";"4"};
{\ar^{\ovl{t}} "2";"6"};
{\ar_{\ovl{c}} "4";"6"};
{\ar@{}|\circlearrowright "0";"6"};
\endxy
\]
in $\ovl{\C}$ with $\ovl{t}\ppr\in\ovl{\Ssc}$, and put
$ \gam\uas [\FR{\ovl{t}}{\ovl{\del}}{\ovl{s}}]=[\FR{\ovl{v\ci t\ppr}}{\ovl{c^{\prime\ast}\del}}{\ovl{s}}]$. 
\end{enumerate}
\end{dfn}

\begin{lem}
The above is well-defined.
\end{lem}
\begin{proof}
{\rm (1)}
Let $(\FR{\ovl{t}}{\ovl{\del}}{\ovl{s}})=(C\ov{\ovl{t}}{\lla}Z\ov{\ovl{\del}}{\dra}X\ov{\ovl{s}}{\lla}A)$ be any triplet with $\ovl{s},\ovl{t}\in\ovl{\Ssc}$ and $\ovl{\del}\in\ovl{\Ebb}(Z,X)$. 
For each expression $\al=\ovl{Q}(\ovl{u})\iv\ci\ovl{Q}(\ovl{a})$ with $\ovl{a}\in\ovl{\C}(A,D)$ and $\ovl{u}\in\ovl{\Ssc}(A\ppr,D)$, the equivalence class  $[\FR{\ovl{t}}{\ovl{a\ppr\sas\del}}{\ovl{s\ppr\ci u}}]$ does not depend on the choice of $(\ref{DD})$. Indeed, for any other commutative square
\[
\xy
(-6,6)*+{X}="0";
(6,6)*+{X\pprr}="2";
(-6,-6)*+{A}="4";
(6,-6)*+{D}="6";
{\ar^{\ovl{a}\pprr} "0";"2"};
{\ar^{\ovl{s}} "4";"0"};
{\ar_{\ovl{s}\pprr} "6";"2"};
{\ar_{\ovl{a}} "4";"6"};
{\ar@{}|\circlearrowright "0";"6"};
\endxy
\]
with $\ovl{s}\pprr\in\ovl{\Ssc}$, by {\rm (MR2)} we may find a pair of morphisms $\ovl{s}_1\in\ovl{\Ssc}(X\ppr,X^{\prime\prime\prime})$ and $\ovl{s}_2\in\ovl{\Ssc}(X\pprr,X^{\prime\prime\prime})$ such that $\ovl{s}_1\ci\ovl{a}\ppr=\ovl{s}_2\ci\ovl{a}\pprr$ and $\ovl{s}_1\ci\ovl{s}\ppr=\ovl{s}_2\ci\ovl{s}\pprr$. Then we obviously have $[\FR{\ovl{t}}{\ovl{a\ppr\sas\del}}{\ovl{s\ppr\ci u}}]=[\FR{\ovl{t}}{\ovl{a\pprr\sas\del}}{\ovl{s\pprr\ci u}}]$. 

Next we check that the class $[\FR{\ovl{t}}{\ovl{a\ppr\sas\del}}{\ovl{s\ppr\ci u}}]$ does not depend on a choice of expressions of $\alpha\in\wC(A,A\ppr)$.
By using {\rm (MR2)}, we may tacitly reduce it to the case where $\alpha=\ovl{Q}(\ovl{u}_1)^{-1}\ci\ovl{Q}(\ovl{a}_1)=\ovl{Q}(\ovl{u})^{-1}\ci\ovl{Q}(\ovl{a})$ admits a commutative diagram below
\[
\xymatrix{
&D_1\ar[d]|{\ovl{u}_2}&\\
A\ar[r]^{\ovl{a}}\ar[ru]^{\ovl{a}_1}&D\ar@{}[ul]|( .3)\circlearrowright\ar@{}[ur]|( .3)\circlearrowright&A\ppr\ar[l]_{\ovl{u}}\ar[lu]_{\ovl{u}_1}
}
\]
with $\ovl{u}_1, \ovl{u}_2, \ovl{u}\in\ovl{\Ssc}$.
By (MR2), the morphisms $X\ov{\ovl{s}}{\lla}A\ov{\ovl{a}_1}{\lra}D_1\ov{\ovl{u}_2}{\lra}D$ yield the following commutative diagram
\[
\xymatrix{
A\ar@<0.5ex>@{}[rr]^\circlearrowright\ar@{}[dr]|\circlearrowright\ar@/^1.2pc/[rr]|{\ovl{a}}\ar[d]_{\ovl{s}}\ar[r]_{\ovl{a}_1}&D_1\ar@{}[dr]|\circlearrowright\ar[d]^{\ovl{s}\pprr}\ar[r]_{\ovl{u}_2}&D\ar[d]^{\ovl{s}\ppr}\\
X\ar@<-0.5ex>@{}[rr]_\circlearrowright\ar@/_1.2pc/[rr]|{\ovl{a}\ppr}\ar[r]^{\ovl{a}\ppr_1}&X\pprr\ar[r]^{\ovl{u}\ppr_2}&X\ppr
}
\]
where the vertical arrows belong to $\ovl{\Ssc}$.
By {\rm (MR1)}, we also have $\ovl{u}_2\ppr\in\ovl{\Ssc}$.
This shows that both $(\FR{\ovl{t}}{\ovl{a\ppr\sas\del}}{\ovl{s\ppr\ci u}})$ and $(\FR{\ovl{t}}{\ovl{a\ppr_{1\ast}\del}}{\ovl{s\pprr\ci u_1}})$ define the same equivalence class in $\wE(C,A)$.

Let us show the independence from the choice of representatives for $[\FR{\ovl{t}}{\ovl{\del}}{\ovl{s}}]$.
Let 
$(\FR{\ovl{t}\pprr}{\ovl{\rho}}{\ovl{s}\pprr})$ be another choice of its representative. It suffices to show in the case where there is a diagram as below,
\begin{equation}\label{Diagwelldefpull}
\xy
(-23,-6)*+{C}="C";
(-7,6)*+{Z}="1";
(-7,-6)*+{Z\pprr}="3";
(7,6)*+{X}="11";
(7,-6)*+{X\pprr}="13";
(23,-6)*+{A}="A";
{\ar_{\ovl{t}} "1";"C"};
{\ar^{\ovl{t}\pprr} "3";"C"};
{\ar@{-->}^{\ovl{\del}} "1";"11"};
{\ar@{-->}_{\ovl{\rho}} "3";"13"};
{\ar_{\ovl{s}} "A";"11"};
{\ar^{\ovl{s}\pprr} "A";"13"};
{\ar_{\ovl{v}\pprr} "3";"1"};
{\ar_{\ovl{u}\pprr} "11";"13"};
{\ar@{}|\circlearrowright (-12,0);(-10,-4)};
{\ar@{}|\circlearrowright (12,0);(10,-4)};
\endxy
\end{equation}
with $\ovl{v}\pprr,\ovl{u}\pprr\in\ovl{\Ssc}$ satisfying $\ovl{v^{\prime\prime\ast}u\pprr\sas\del}=\ovl{\rho}$. Take a commutative square as in $(\ref{DD})$.
For morphisms $X\pprr\ov{\ovl{u}\pprr}{\lla}X\ov{\ovl{a}\ppr}{\lra}X\ppr$, we may take the following commutative square by {\rm (MR2)}
\[
\xy
(-6,6)*+{X}="0";
(6,6)*+{X\ppr}="2";
(-6,-6)*+{X\pprr}="4";
(6,-6)*+{W}="6";
%
%
{\ar^{\ovl{a}\ppr} "0";"2"};
{\ar_{\ovl{u}\pprr} "0";"4"};
{\ar^{\ovl{w}} "2";"6"};
{\ar_{\ovl{a}\pprr} "4";"6"};
%
%
{\ar@{}|\circlearrowright "0";"6"};
\endxy
\]
with $\ovl{w}\in\ovl{\Ssc}$.
This gives rise to the following diagram
\[
\xy
(-23,-6)*+{C}="C";
(-7,6)*+{Z}="1";
(-7,-6)*+{Z\pprr}="3";
(7,6)*+{X\ppr}="11";
(7,-6)*+{W}="13";
(23,-6)*+{A\ppr}="A'";
{\ar_{\ovl{t}} "1";"C"};
{\ar^{\ovl{t}\pprr} "3";"C"};
{\ar@{-->}^{\ovl{a}\ppr_\ast\ovl{\del}} "1";"11"};
{\ar@{-->}_{\ovl{a}\pprr_\ast\ovl{\rho}} "3";"13"};
{\ar_{\ovl{s\ppr u}} "A'";"11"};
{\ar^{\ovl{ws\ppr u}} "A'";"13"};
{\ar_{\ovl{v}\pprr} "3";"1"};
{\ar_{\ovl{w}} "11";"13"};
{\ar@{}|\circlearrowright (-12,0);(-10,-4)};
{\ar@{}|\circlearrowright (12,0);(10,-4)};
\endxy
\]
where all solid arrows belong to $\ovl{\Ssc}$.
This shows $[\FR{\ovl{t}}{\ovl{a}\ppr\sas\ovl{\del}}{\ovl{s\ppr u}}]=[\FR{\ovl{t}\pprr}{\ovl{a}\pprr\sas\ovl{\rho}}{\ovl{ws\ppr u}}]$.

{\rm (2)} This is dual to {\rm (1)}.
\end{proof}

\begin{rem}\label{RemPP}
For any $[\FR{\ovl{t}}{\ovl{\del}}{\ovl{s}}]\in\wE(C,A)$, $\al=\ovl{Q}(\ovl{u})\iv\ci\ovl{Q}(\ovl{a})\in\wC(A,A\ppr)$ and  $\gam=\ovl{Q}(\ovl{c})\ci\ovl{Q}(\ovl{v})\iv\in\wC(C\ppr,C)$ as in Definition~\ref{DefwEFtr}, the element
\[ \gam\uas\al\sas [\FR{\ovl{t}}{\ovl{\del}}{\ovl{s}}]=[\FR{\ovl{v\ci t\ppr}}{\ovl{c^{\prime\ast}a\ppr\sas\del}}{\ovl{s\ppr\ci u}}] \]
is given by the following diagram.
\[
\xy
(-31,-8)*+{C\ppr}="0";
(-15,8)*+{E\ppr}="2";
(-23,0)*+{E}="4";
(-15,-8)*+{C}="6";
(-7,0)*+{Z}="8";
(7,0)*+{X}="10";
(15,8)*+{X\ppr}="12";
(15,-8)*+{A}="14";
(23,0)*+{D}="16";
(31,-8)*+{A\ppr}="18";
{\ar_{\ovl{t}\ppr} "2";"4"};
{\ar_{\ovl{v}} "4";"0"};
{\ar^{\ovl{c}\ppr} "2";"8"};
{\ar^{\ovl{t}} "8";"6"};
{\ar_{\ovl{c}} "4";"6"};
{\ar@{-->}_{\ovl{\del}} "8";"10"};
{\ar@/^0.60pc/@{-->}^{\ovl{c^{\prime\ast}a\ppr\sas\del}} "2";"12"};
{\ar^{\ovl{a}\ppr} "10";"12"};
{\ar_{\ovl{s}\ppr} "16";"12"};
{\ar^{\ovl{s}} "14";"10"};
{\ar_{\ovl{a}} "14";"16"};
{\ar_{\ovl{u}} "18";"16"};
{\ar@{}|\circlearrowright "2";"6"};
{\ar@{}|\circlearrowright "10";"16"};
\endxy
\]
\end{rem}

\begin{lem}\label{LemEmu}
For any $A,C\in\C$, the following holds.
\begin{enumerate}
\item For any pair of elements $[\FR{\ovl{t}_i}{\ovl{\del}_i}{\ovl{s}_i}]\in\wE(C,A)$ $(i=1,2)$, using Proposition~\ref{PropCommonDenom} {\rm (1)}, we define their sum as
\[ [\FR{\ovl{t}_1}{\ovl{\del}_1}{\ovl{s}_1}]+[\FR{\ovl{t}_2}{\ovl{\del}_2}{\ovl{s}_2}]=[\FR{\ovl{t}}{\ovl{\rho}_1+\ovl{\rho}_2}{\ovl{s}}] \]
by taking common denominators so that $[\FR{\ovl{t}_i}{\ovl{\del}_i}{\ovl{s}_i}]=[\FR{\ovl{t}}{\ovl{\rho}_i}{\ovl{s}}]$ hold for $i=1,2$.
Then this sum is well-defined, and makes $\wE(C,A)$ an abelian group with zero element $[\FR{\ovl{\id}}{0}{\ovl{\id}}]$.
\item If we define a map $\ovl{\mu}_{C,A}\co\ovl{\Ebb}(C,A)\to\wE(C,A)$ by $\ovl{\mu}_{C,A}(\ovl{\del})=[\FR{\ovl{\id}}{\ovl{\del}}{\ovl{\id}}]$ for any $\ovl{\del}\in\ovl{\Ebb}(C,A)$, then it is a monomorphism.
\end{enumerate}
\end{lem}
\begin{proof}
{\rm (1)}
Well-definedness of the sum is obvious from Proposition~\ref{PropCommonDenom} {\rm (2)}.
We will show that it is associative. 
For any $[\FR{\ovl{t}_i}{\ovl{\del}_i}{\ovl{s}_i}]\in\wE(C,A)$ $(i=1,2,3)$, we take common denominators so that $[\FR{\ovl{t}_i}{\ovl{\del}_i}{\ovl{s}_i}]=[\FR{\ovl{t}}{\ovl{\rho}_i}{\ovl{s}}]$ hold for $i=1,2,3$.
Then we have 
\begin{eqnarray*}
([\FR{\ovl{t}_1}{\ovl{\del}_1}{\ovl{s}_1}]+[\FR{\ovl{t}_2}{\ovl{\del}_2}{\ovl{s}_2}])+[\FR{\ovl{t}_3}{\ovl{\del}_3}{\ovl{s}_3}]&=&[\FR{\ovl{t}}{\ovl{\rho}_1+\ovl{\rho}_2+\ovl{\rho}_3}{\ovl{s}}] \\
&=&[\FR{\ovl{t}_1}{\ovl{\del}_1}{\ovl{s}_1}]+([\FR{\ovl{t}_2}{\ovl{\del}_2}{\ovl{s}_2}]+[\FR{\ovl{t}_3}{\ovl{\del}_3}{\ovl{s}_3}]).
\end{eqnarray*}
Commutativity is obvious.
For any $[\FR{\ovl{t}}{\ovl{\del}}{\ovl{s}}]\in\wE(C,A)$, we have
\[ [\FR{\ovl{t}}{\ovl{\del}}{\ovl{s}}]+[\FR{\ovl{\id}}{0}{\ovl{\id}}]=[\FR{\ovl{t}}{\ovl{\del}}{\ovl{s}}]+[\FR{\ovl{t}}{0}{\ovl{s}}]=[\FR{\ovl{t}}{\ovl{\del}}{\ovl{s}}].
\]
Thus $[\FR{\ovl{\id}}{0}{\ovl{\id}}]$ is the zero element. 

For any $[\FR{\ovl{t}}{\ovl{\del}}{\ovl{s}}]\in\wE(C,A)$, we have the inverse element $[\FR{\ovl{t}}{\ovl{-\del}}{\ovl{s}}]$ of $[\FR{\ovl{t}}{\ovl{\del}}{\ovl{s}}]$. Thus $\wE(C,A)$ is an abelian group with respect to the above sum.

{\rm (2)} Additivity is immediate from the definition. Monomorphicity follows from Remark~\ref{RemTrivK}.
\end{proof}

\begin{prop}\label{PropEmu}
The above definitions give a biadditive functor $\wE\co\wC^\op\ti\wC\to\Ab$ and a natural transformation $\ovl{\mu}\co\ovl{\Ebb}\tc\wE\ci(\ovl{Q}^\op\ti\ovl{Q})$.
\end{prop}
\begin{proof}
First we show that $\wE$ is a biadditive functor. By Remark~\ref{RemPP}, it is easy to check $\gam\uas\al\sas=\al\sas\gam\uas$ for any morphisms $\al,\gam$ in $\wC$. 

We show that $\al\sas$ is a group homomorphism for any $\al\in\wC(A,A\ppr)$. Take morphisms $\ovl{a}\in\ovl{\C}(A,D)$ and $\ovl{u}\in\ovl{\Ssc}(A\ppr,D)$ such that $\al=\ovl{Q}(\ovl{u})\iv\ci\ovl{Q}(\ovl{a})$. For any $[\FR{\ovl{t}}{\ovl{\del_{1}}}{\ovl{s}}], [\FR{\ovl{t\ppr}}{\ovl{\del_{2}}}{\ovl{s\ppr}}]\in\wE(C,A)$, we need to show $\al\sas([\FR{\ovl{t}}{\ovl{\del_{1}}}{\ovl{s}}]+[\FR{\ovl{t\ppr}}{\ovl{\del_{2}}}{\ovl{s\ppr}}])=\al\sas([\FR{\ovl{t}}{\ovl{\del_{1}}}{\ovl{s}}])+\al\sas([\FR{\ovl{t\ppr}}{\ovl{\del_{2}}}{\ovl{s\ppr}}])$. Taking common denominators, we may assume that $t=t\ppr,s=s\ppr$ and $[\FR{\ovl{t}}{\ovl{\del_i}}{\ovl{s}}]=[C\ov{\ovl{t}}{\lla}Z\ov{\ovl{\del_i}}{\dra}X\ov{\ovl{s}}{\lla}A]\in\wE(C,A)$ hold for $i=1,2$ from the beginning. By {\rm (MR2)}, there exist morphisms $\ovl{a\ppr}\co X\to X\ppr$ and $\ovl{s\ppr}\co D\to X\ppr$ satisfying $\ovl{a\ppr\ci s}=\ovl{s\ppr\ci a}$. Then we obtain
\begin{eqnarray*}
\al\sas([\FR{\ovl{t}}{\ovl{\del_{1}}}{\ovl{s}}]+[\FR{\ovl{t}}{\ovl{\del_{2}}}{\ovl{s}}])
&=&\al\sas([\FR{\ovl{t}}{\ovl{\del_{1}+\del_{2}}}{\ovl{s}}]) \\
&=&[\FR{\ovl{t}}{\ovl{a\ppr\sas(\del_{1}+\del_{2})}}{\ovl{s\ppr\ci u}}] \\
&=&[\FR{\ovl{t}}{\ovl{a\ppr\sas\del_{1}}}{\ovl{s\ppr\ci u}}]+[\FR{\ovl{t}}{\ovl{a\ppr\sas\del_{2}}}{\ovl{s\ppr\ci u}}] \\
&=&\al\sas([\FR{\ovl{t}}{\ovl{\del_{1}}}{\ovl{s}}])+\al\sas([\FR{\ovl{t}}{\ovl{\del_{2}}}{\ovl{s}}]).
\end{eqnarray*}
Dually we can show that $\gam\uas$ is a group homomorphism. 

Let us show $(\be\ci\al)\sas=\be\sas\al\sas$ for any $\al\in\wC(A,A_1)$ and $\be\in\wC(A_1,A_2)$. Put $\al=\ovl{Q}(\ovl{u_{1}})^{-1}\ci\ovl{Q}(\ovl{a_{1}})$ and $\be=\ovl{Q}(\ovl{u_{2}})^{-1}\ci\ovl{Q}(\ovl{a_{2}})$ as the following diagram. For any $[\FR{\ovl{t}}{\ovl{\del}}{\ovl{s}}]=[C\ov{\ovl{t}}{\lla}Z\ov{\ovl{\del}}{\dra}X\ov{\ovl{s}}{\lla}A]\in\wE(C,A)$, we obtain the following diagram by {\rm (MR2)}
\[
\xy
(-15,-8)*+{C}="6";
(-7,0)*+{Z}="8";
(7,0)*+{X}="10";
(15,8)*+{X_{1}}="12";
(15,-8)*+{A}="14";
(23,0)*+{D_{1}}="16";
(31,-8)*+{A_{1}}="18";
(31,8)*+{D_{3}}="20";
(39,0)*+{D_{2}}="22";
(47,-8)*+{A_{2}}="24";
(23,16)*+{X_{2}}="26";
{\ar^{\ovl{t}} "8";"6"};
{\ar@{-->}_{\ovl{\del}} "8";"10"};
{\ar^{\ovl{a_{1}\ppr}} "10";"12"};
{\ar^{\ovl{s_{1}}} "16";"12"};
{\ar^{\ovl{s}} "14";"10"};
{\ar_{\ovl{a_{1}}} "14";"16"};
{\ar_{\ovl{u_{1}}} "18";"16"};
{\ar_{\ovl{a_{2}}} "18";"22"};
{\ar_{\ovl{u_{2}}} "24";"22"};
{\ar_{\ovl{u_{3}}} "22";"20"};
{\ar^{\ovl{a_{3}}} "16";"20"};
{\ar^{\ovl{a_{3}\ppr}} "12";"26"};
{\ar_{\ovl{s_{2}}} "20";"26"};
{\ar@{}|\circlearrowright "10";"16"};
{\ar@{}|\circlearrowright "16";"22"};
{\ar@{}|\circlearrowright "12";"20"};
\endxy
\]
in which each $\ovl{u_{i}},\ovl{s_{i}}$ belongs to $\ovl{\Ssc}$.
Then we obtain
\begin{eqnarray*}
(\be\ci\al)\sas([\FR{\ovl{t}}{\ovl{\del}}{\ovl{s}}])
&=&(\ovl{Q}(\ovl{u_{3}\ci u_{2}})^{-1}\ci\ovl{Q}(\ovl{a_{3}\ci a_{1}}))\sas([\FR{\ovl{t}}{\ovl{\del}}{\ovl{s}}]) \\
&=&[\FR{\ovl{t}}{\ovl{(a_{3}\ppr)\sas(a_{1}\ppr)\sas\del}}{\ovl{s_{2}\ci u_{3}\ci u_{2}}}] \\
&=&\be\sas([\FR{\ovl{t}}{\ovl{(a_{1}\ppr)\sas\del}}{\ovl{s_{1}\ci u_{1}}}])\\
&=&\be\sas\al\sas([\FR{\ovl{t}}{\ovl{\del}}{\ovl{s}}]).
\end{eqnarray*}
Dually we can show $(\be\ci\al)\uas=\al\uas\be\uas$. 

For any $\al_{i}=\ovl{Q}(\ovl{u_{i}})^{-1}\ci\ovl{Q}(\ovl{a_{i}})\in\wC(A, A\ppr)$ for $i=1,2$, we show that $(\al_{1}+\al_{2})\sas=(\al_{1})\sas+(\al_{2})\sas$. Taking common denominators, we may assume $\ovl{u_{1}}=\ovl{u_{2}}$ with $u_{1}\co A\ppr\to D$. Take any element $[\FR{\ovl{t}}{\ovl{\del}}{\ovl{s}}]=[C\ov{\ovl{t}}{\lla}Z\ov{\ovl{\del}}{\dra}X\ov{\ovl{s}}{\lla}A]\in\wE(C,A)$. By {\rm (MR2)}, there are morphisms $\ovl{a_{i}\ppr}\co X\to X\ppr, \ovl{s_{i}}\co D\to X\ppr$ such that $\ovl{a_{i}\ppr\ci s}=\ovl{s_{i}\ci a_{i}}$ for $i=1,2$. Taking common denominators, we may assume that $s_{1}=s_{2}$. Then we have
\begin{eqnarray*}
((\al_{1})\sas+(\al_{2})\sas)([\FR{\ovl{t}}{\ovl{\del}}{\ovl{s}}])
&=&(\al_{1})\sas([\FR{\ovl{t}}{\ovl{\del}}{\ovl{s}}])+(\al_{2})\sas([\FR{\ovl{t}}{\ovl{\del}}{\ovl{s}}]) \\
&=&[\FR{\ovl{t}}{\ovl{(a_{1}\ppr)\sas\del}}{\ovl{s_{1}\ci u_{1}}}]+[\FR{\ovl{t}}{\ovl{(a_{2}\ppr)\sas\del}}{\ovl{s_{1}\ci u_{1}}}] \\
&=&[\FR{\ovl{t}}{\ovl{(a_{1}\ppr+a_{2}\ppr)\sas\del}}{\ovl{s_{1}\ci u_{1}}}] \\
&=&(\ovl{Q}(\ovl{u_{1}})^{-1}\ci\ovl{Q}(\ovl{a_{1}+a_{2}}))\sas([\FR{\ovl{t}}{\ovl{\del}}{\ovl{s}}]) \\
&=&(\al_{1}+\al_{2})\sas([\FR{\ovl{t}}{\ovl{\del}}{\ovl{s}}])
\end{eqnarray*}
as desired. Dually, we obtain $(\al_{1}+\al_{2})\uas=(\al_{1})\uas+(\al_{2})\uas$. Thus $\wE$ is a biadditive functor. 

We show that $\ovl{\mu}$ is a natural transformation. By Lemma~\ref{LemEmu} {\rm (2)}, it suffices to confirm the naturality of $\ovl{\mu}$. Let $\ovl{a}\co A\to A\ppr$ be an arbitrary morphism in $\ovl{\C}$. For any object $C\in\C$ and any element $\ovl{\del}\in\ovl{\E}(C,A)$, we have 
\begin{eqnarray*}
(\wE(C,\ovl{Q}(\ovl{a}))\ci\ovl{\mu}_{C,A})(\ovl{\del})
&=&\wE(C,\ovl{Q}(\ovl{a}))([\FR{\ovl{\id}}{\ovl{\del}}{\ovl{\id}}]) \\
&=&[\FR{\ovl{\id}}{\ovl{a\sas\del}}{\ovl{\id}}] \\
&=&\ovl{\mu}_{C,A\ppr}(\ovl{a\sas\del}) \\
&=&(\ovl{\mu}_{C,A\ppr}\ci\ovl{\E}(C,\ovl{a}))(\ovl{\del}).
\end{eqnarray*}
Dually, $\wE(\ovl{Q}(\ovl{c}),A)\ci\ovl{\mu}_{C,A}=\ovl{\mu}_{C\ppr,A}\ci\ovl{\E}(\ovl{c},A)$ holds for any morphism $\ovl{c}\co C\ppr\to C$ in $\ovl{\C}$. Thus $\ovl{\mu}$ is a natural transformation. 
\end{proof}

\subsection{Construction of $\ws$}

Similarly as in the previous subsection, suppose that $\Ssc$ satisfies {\rm (M0),(MR1),(MR2),(MR3)} and $\Ssc=p\iv(\ovl{\Ssc})$.
In this subsection, we construct a realization $\ws$ of $\wE$.

\begin{prop}\label{PropWellDef}
Let $A,C\in\C$ be any pair of objects.
For each $i=1,2$, let $(\FR{\ovl{t}_i}{\ovl{\del}_i}{\ovl{s}_i})=(C\ov{\ovl{t}_i}{\lla}Z_i\ov{\ovl{\del}_i}{\dra}X_i\ov{\ovl{s}_i}{\lla}A)$ be a triplet satisfying $\ovl{s}_i\in\ovl{\Ssc}(A,X_i), \ovl{t}_i\in\ovl{\Ssc}(Z_i,C), \ovl{\del}_i\in\ovl{\Ebb}(Z_i,X_i)$,  with $\sfr(\del_i)=[X_i\ov{x_i}{\lra}Y_i\ov{y_i}{\lra}Z_i]$.

If $[\FR{\ovl{t}_1}{\ovl{\del}_1}{\ovl{s}_1}]=[\FR{\ovl{t}_2}{\ovl{\del}_2}{\ovl{s}_2}]$ holds in $\wE(C,A)$, then
\begin{equation}\label{Equiv_Seq_in_Ctilde}
[A\ov{Q(x_1\ci s_1)}{\lra}Y_1\ov{Q(t_1\ci y_1)}{\lra}C]=[A\ov{Q(x_2\ci s_2)}{\lra}Y_2\ov{Q(t_2\ci y_2)}{\lra}C]
\end{equation}
holds as sequences in $\wC$.
\end{prop}
\begin{proof}
By the definition of the equivalence relation, it suffices to show $(\ref{Equiv_Seq_in_Ctilde})$ in the case where there exist $\ovl{u}\in\ovl{\Ssc}(X_1,X_2)$ and $\ovl{v}\in\ovl{\Ssc}(Z_2,Z_1)$ such that 
$\ovl{s}_2=\ovl{u}\ci\ovl{s}_1$, $\ovl{t}_2=\ovl{t}_1\ci\ovl{v}$, and $\ovl{u\sas v\uas\del_1}=\ovl{\del}_2$. 
By the definition of $\ovl{\Ebb}$, there exists $u\ppr\in\Ssc(X_2,X_2\ppr)$ such that $u\ppr\sas(u\sas v\uas\del_1)=u\ppr\sas\del_2$.

Hence replacing $u,s_2,\del_2$ by $u\ppr u,u\ppr s_2,u\ppr\sas\del_2$ respectively, we may assume that $u\sas v\uas\del_1=\del_2$ holds from the beginning. In this case, for $\sfr(u\sas\del_1)=[X_2\ov{x_3}{\lra}Y_3\ov{y_3}{\lra}Z_1]$, there exist $\wbf_1,\wbf_2\in\ovl{\Ssc}$ which make
\[
\xy
(-14,12)*+{X_1}="0";
(0,12)*+{Y_1}="2";
(14,12)*+{Z_1}="4";
(-14,0)*+{X_2}="10";
(0,0)*+{Y_3}="12";
(14,0)*+{Z_1}="14";
(-14,-12)*+{X_2}="20";
(0,-12)*+{Y_2}="22";
(14,-12)*+{Z_2}="24";
{\ar^{\ovl{x}_1} "0";"2"};
{\ar^{\ovl{y}_1} "2";"4"};
{\ar_{\ovl{u}} "0";"10"};
{\ar_{\wbf_1} "2";"12"};
{\ar@{=} "4";"14"};
{\ar^{\ovl{x}_3} "10";"12"};
{\ar^{\ovl{y}_3} "12";"14"};
{\ar@{=} "20";"10"};
{\ar^{\wbf_2} "22";"12"};
{\ar_{\ovl{v}} "24";"14"};
{\ar_{\ovl{x}_2} "20";"22"};
{\ar_{\ovl{y}_2} "22";"24"};
{\ar@{}|\circlearrowright "0";"12"};
{\ar@{}|\circlearrowright "2";"14"};
{\ar@{}|\circlearrowright "10";"22"};
{\ar@{}|\circlearrowright "12";"24"};
\endxy
\]
commutative in $\ovl{\C}$ by {\rm (MR3)}.
Then $\al=\ovl{Q}(\wbf_2)\iv\ci\ovl{Q}(\wbf_1)\in\Iso(\wC)$ makes
\[
\xy
(-16,0)*+{A}="0";
(4,0)*+{}="1";
(0,8)*+{Y_1}="2";
(0,-8)*+{Y_2}="4";
(-4,0)*+{}="5";
(16,0)*+{C}="6";
{\ar^{Q(x_1\ci s_1)} "0";"2"};
{\ar_{Q(x_2\ci s_2)} "0";"4"};
{\ar^{Q(t_1\ci y_1)} "2";"6"};
{\ar_{Q(t_2\ci y_2)} "4";"6"};
{\ar^{\cong}_{\al} "2";"4"};
{\ar@{}|\circlearrowright "0";"1"};
{\ar@{}|\circlearrowright "5";"6"};
\endxy
\]
commutative in $\wC$, which shows $(\ref{Equiv_Seq_in_Ctilde})$.
\end{proof}

\begin{dfn}\label{Def_sN}
Let $[\FR{\ovl{t}}{\ovl{\del}}{\ovl{s}}]=[C\ov{\ovl{t}}{\lla}Z\ov{\ovl{\del}}{\dra}X\ov{\ovl{s}}{\lla}A]\in\wE(C,A)$ be any $\wE$-extension. Take $\sfr(\del)=[X\ov{x}{\lra}Y\ov{y}{\lra}Z]$, and put
\begin{equation}\label{Real_sN}
\ws([\FR{\ovl{t}}{\ovl{\del}}{\ovl{s}}])=[A\ov{Q(x\ci s)}{\lra}Y\ov{Q(t\ci y)}{\lra}C]
\end{equation}
in $\wC$. Well-definedness is ensured by Proposition~\ref{PropWellDef}.
\end{dfn}

The following lemma will be used to show Proposition~\ref{PropExact}.
\begin{lem}\label{LemNtoS}
Let $f\in\C(A,B)$ be any morphism. The following holds.
\begin{enumerate}
\item If $f$ is an $\sfr$-inflation with $\Cone(f)\in\Ncal_{\Ssc}$, then $f\in\Ssc$.
\item Dually, if $f$ is an $\sfr$-deflation with $\CoCone(f)\in\Ncal_{\Ssc}$, then $f\in\Ssc$.
\end{enumerate}
\end{lem}
\begin{proof}
By duality, it suffices to show {\rm (1)}. Let $A\ov{f}{\lra}B\ov{g}{\lra}N\ov{\del}{\dra}$ be an $\sfr$-triangle with $N\in\Ncal_{\Ssc}$.
Remark that $A\ov{\id_A}{\lra}A\to 0\ov{0}{\dra}$ is also an $\sfr$-triangle.
By {\rm (MR3)}, there exists $\ovl{u}\in\ovl{\Ssc}(A,B)$ which makes
\[
\xy
(-14,6)*+{A}="0";
(0,6)*+{A}="2";
(14,6)*+{0}="4";
(-14,-6)*+{A}="10";
(0,-6)*+{B}="12";
(14,-6)*+{N}="14";
{\ar^{\ovl{\id_A}} "0";"2"};
{\ar^{} "2";"4"};
{\ar_{\ovl{\id_A}} "0";"10"};
{\ar_{\ovl{u}} "2";"12"};
{\ar "4";"14"};
{\ar_{\ovl{f}} "10";"12"};
{\ar_{\ovl{g}} "12";"14"};
{\ar@{}|\circlearrowright "0";"12"};
{\ar@{}|\circlearrowright "2";"14"};
\endxy
\]
commutative in $\ovl{\C}$. This shows $\ovl{f}=\ovl{u}\in\ovl{\Ssc}$, hence $f\in\Ssc$ follows.
\end{proof}

\begin{prop}\label{PropExact}
Let $[\FR{\ovl{t}}{\ovl{\del}}{\ovl{s}}]=[C\ov{\ovl{t}}{\lla}Z\ov{\ovl{\del}}{\dra}X\ov{\ovl{s}}{\lla}A]\in\wE(C,A)$ be any $\wE$-extension, and take its representative arbitrarily. Then for $\sfr(\del)=[X\ov{x}{\lra}Y\ov{y}{\lra}Z]$, the following holds.
\begin{enumerate}
\item The sequence
\[ \wC(-,A)\ov{Q(x\ci s)\ci-}{\lra}\wC(-,Y)\ov{Q(t\ci y)\ci-}{\lra}\wC(-,C)\ov{[\FR{\ovl{t}}{\ovl{\del}}{\ovl{s}}]\ssh}{\lra}\wE(-,A) \]
is exact.
\item The sequence
\[ \wC(C,-)\ov{-\ci Q(t\ci y)}{\lra}\wC(Y,-)\ov{-\ci Q(x\ci s)}{\lra}\wC(A,-)\ov{[\FR{\ovl{t}}{\ovl{\del}}{\ovl{s}}]\ush}{\lra}\wE(C,-) \]
is exact.
\end{enumerate}
\end{prop}
\begin{proof}
{\rm (1)} By the commutativity of the following diagram,
\[
\xy
(-42,6)*+{\wC(-,A)}="0";
(-14,6)*+{\wC(-,Y)}="2";
(14,6)*+{\wC(-,C)}="4";
(42,6)*+{\wE(-,A)}="6";
(-42,-6)*+{\wC(-,X)}="10";
(-14,-6)*+{\wC(-,Y)}="12";
(14,-6)*+{\wC(-,Z)}="14";
(42,-6)*+{\wE(-,X)}="16";
{\ar^{Q(x\ci s)\ci-} "0";"2"};
{\ar^{Q(t\ci y)\ci-} "2";"4"};
{\ar^{[\FR{\ovl{t}}{\ovl{\del}}{\ovl{s}}]\ssh} "4";"6"};
{\ar_{Q(s)\ci-}^{\cong} "0";"10"};
{\ar@{=} "2";"12"};
{\ar^{Q(t)\ci-}_{\cong} "14";"4"};
{\ar^{Q(s)\sas}_{\cong} "6";"16"};
{\ar_{Q(x)\ci-} "10";"12"};
{\ar_{Q(y)\ci-} "12";"14"};
{\ar_{[\FR{\ovl{\id}}{\ovl{\del}}{\ovl{\id}}]\ssh} "14";"16"};
{\ar@{}|\circlearrowright "0";"12"};
{\ar@{}|\circlearrowright "2";"14"};
{\ar@{}|\circlearrowright "4";"16"};
\endxy
\]
it suffices to show the exactness of the bottom row. 

Let us show the exactness of 
$\wC(P,X)\ov{Q(x)\ci-}{\lra}\wC(P,Y)\ov{Q(y)\ci-}{\lra}\wC(P,Z)$
for any $P\in\C$. Let $\be\in\wC(P,Y)$ be any element. Express $\be$ as $\be=\ovl{Q}(\ovl{f})\ci\ovl{Q}(\ovl{s})\iv$, using $\ovl{f}\in\ovl{\C}(P\ppr,Y)$ and $\ovl{s}\in\ovl{\Ssc}(P\ppr,P)$. If $Q(y)\ci\be=0$ then $\ovl{Q}(\ovl{y}\ci\ovl{f})=0$, hence there exists some $\ovl{s}\ppr\in\ovl{\Ssc}(P\pprr,P\ppr)$ such that $\ovl{y}\ci\ovl{f}\ci\ovl{s}\ppr=0$ in $\ovl{\C}$. This means that there exist $N\in\Ncal_{\Ssc}$ and $i\in\C(P\pprr,N),j\in\C(N,Z)$ such that $y\ci f\ci s\ppr=j\ci i$. For $\sfr(j\uas\del)=[X\ov{x\ppr}{\lra}Y\ppr\ov{y\ppr}{\lra}N]$, we have a morphism of $\sfr$-triangles
\[
\xy
(-12,6)*+{X}="0";
(0,6)*+{Y\ppr}="2";
(12,6)*+{N}="4";
(24,6)*+{}="6";
(-12,-6)*+{X}="10";
(0,-6)*+{Y}="12";
(12,-6)*+{Z}="14";
(24,-6)*+{}="16";
{\ar^{x\ppr} "0";"2"};
{\ar^{y\ppr} "2";"4"};
{\ar@{-->}^{j\uas\del} "4";"6"};
{\ar@{=} "0";"10"};
{\ar^{k} "2";"12"};
{\ar^{j} "4";"14"};
{\ar_{x} "10";"12"};
{\ar_{y} "12";"14"};
{\ar@{-->}_{\del} "14";"16"};
{\ar@{}|\circlearrowright "0";"12"};
{\ar@{}|\circlearrowright "2";"14"};
\endxy
\]
such that $Y\ppr\ov{\left[\bsm -y\ppr\\ k\esm\right]}{\lra}N\oplus Y\ov{[j\ y]}{\lra}Z\ov{x\ppr\sas\del}{\dra}$ is an $\sfr$-triangle.
By the exactness of $\C(P\pprr,Y\ppr)\to\C(P\pprr,N\oplus Y)\to\C(P\pprr,Z)$, there is $p\in\C(P\pprr,Y\ppr)$ such that
\[ \left[\bsm -y\ppr\\ k\esm\right]\ci p=\left[\bsm -i\\ f\ci s\ppr\esm\right]. \]
By Lemma~\ref{LemNtoS}, we have $\ovl{x}\ppr\in\ovl{\Ssc}$. If we put $\al=\ovl{Q}(\ovl{x}\ppr)\iv\ci Q(p)\ci Q(s\ci s\ppr)\iv$, then it satisfies $\be=Q(x)\ci\al$ by construction.

Let us show the exactness of
$\wC(P,Y)\ov{Q(y)\ci-}{\lra}\wC(P,Z)\ov{[\FR{\ovl{\id}}{\ovl{\del}}{\ovl{\id}}]\ssh}{\lra}\wE(P,X)$
for any $P\in\C$. Suppose that a morphism $\be\in\wC(P,Z)$ 
satisfies $[\FR{\ovl{\id}}{\ovl{\del}}{\ovl{\id}}]\ssh(\be)=\be\uas[\FR{\ovl{\id}}{\ovl{\del}}{\ovl{\id}}]=0$. Take $\ovl{f}\in\ovl{\C}(P\ppr,Z)$ and $\ovl{s}\in\ovl{\Ssc}(P\ppr,P)$ such that $\be=\ovl{Q}(\ovl{f})\ci\ovl{Q}(\ovl{s})\iv$. 
Then $\be\uas[\FR{\ovl{\id}}{\ovl{\del}}{\ovl{\id}}]=[\FR{\ovl{s}}{\ovl{f\uas\del}}{\ovl{\id}}]$ by definition. By Remark~\ref{RemTrivK} it follows $\ovl{f\uas\del}=0$, hence $t\uas (f\uas\del)=0$ holds for some $t\in\Ssc(P\pprr,P\ppr)$. 
By the exactness of
$\C(P\pprr,Y)\ov{y\ci-}{\lra}\C(P\pprr,Z)\ov{\del\ssh}{\lra}\Ebb(P\pprr,X)$, there exists $g\in\C(P\pprr,Y)$ such that $y\ci g=f\ci t$. If we put $\al=\ovl{Q}(\ovl{g})\ci\ovl{Q}(\ovl{s\ci t})\iv\in\wC(P,Y)$, this satisfies $Q(y)\ci\al=\be$ in $\wC(P,Z)$ by construction.

{\rm (2)} This can be shown dually to {\rm (1)}. 

\end{proof}

\subsection{Proof of the main theorem}

Similarly as in the previous subsection, suppose that $\Ssc$ satisfies {\rm (M0),(MR1),(MR2),(MR3)} and $\Ssc=p\iv(\ovl{\Ssc})$.
We will prove Theorem~\ref{ThmMultLoc} item by item.

\begin{proof}[Proof of Theorem~\ref{ThmMultLoc} {\rm (1)}]
Since conditions {\rm (C1'),(C2'),(C3')} can be checked in a dual manner, it suffices to confirm conditions {\rm (C1),(C2),(C3)} in Definition~\ref{DefEACat}.
As {\rm (C1)} follows from Proposition~\ref{PropExact} and {\rm (C2)} is obvious from the definition, it is enough to show {\rm (C3)}. 

Let $[\FR{\ovl{t}}{\ovl{\del}}{\ovl{s}}]=[C\ov{\ovl{t}}{\lla}Z\ov{\ovl{\del}}{\dra}X\ov{\ovl{s}}{\lla}A]\in\wE(C,A)$ be any $\wE$-extension, and let $\al\in\wC(A,A\ppr)$ be any morphism. Take any $\ovl{a}\in\ovl{\C}(A,D)$ and $\ovl{u}\in\ovl{\Ssc}(A\ppr,D)$ so that $\al=\ovl{Q}(\ovl{u})\iv\ci\ovl{Q}(\ovl{a})$ holds. Then by definition
$\al\sas [\FR{\ovl{t}}{\ovl{\del}}{\ovl{s}}]=[\FR{\ovl{t}}{\ovl{a\ppr\sas\del}}{\ovl{s\ppr\ci u}}]$
is given by using a commutative square
\begin{equation}\label{CSby(MS2)}
\xy
(-6,6)*+{A}="0";
(6,6)*+{D}="2";
(-6,-6)*+{X}="4";
(6,-6)*+{X\ppr}="6";
{\ar^{\ovl{a}} "0";"2"};
{\ar_{\ovl{s}} "0";"4"};
{\ar^{\ovl{s}\ppr} "2";"6"};
{\ar_{\ovl{a}\ppr} "4";"6"};
{\ar@{}|\circlearrowright "0";"6"};
\endxy
\end{equation}
in $\ovl{\C}$ satisfying $\ovl{s}\ppr\in\ovl{\Ssc}(D,X\ppr)$. By {\rm (C3)} for $\CEs$, 
we obtain a morphism of $\sfr$-triangles
\begin{equation}\label{Morph_dels}
\xy
(-12,6)*+{X}="0";
(0,6)*+{Y}="2";
(12,6)*+{Z}="4";
(24,6)*+{}="6";
(-12,-6)*+{X\ppr}="10";
(0,-6)*+{Y\ppr}="12";
(12,-6)*+{Z}="14";
(24,-6)*+{}="16";
{\ar^{x} "0";"2"};
{\ar^{y} "2";"4"};
{\ar@{-->}^{\del} "4";"6"};
{\ar_{a\ppr} "0";"10"};
{\ar_{b} "2";"12"};
{\ar@{=} "4";"14"};
{\ar_{x\ppr} "10";"12"};
{\ar_{y\ppr} "12";"14"};
{\ar@{-->}_{a\ppr\sas\del} "14";"16"};
{\ar@{}|\circlearrowright "0";"12"};
{\ar@{}|\circlearrowright "2";"14"};
\endxy
\end{equation}
which makes
\begin{equation}\label{stri_mappingcone}
X\ov{\left[\bsm x\\ a\ppr\esm\right]}{\lra}Y\oplus X\ppr\ov{[b\ -x\ppr]}{\lra}Y\ppr\ov{y^{\prime\ast}\del}{\dra}
\end{equation} an $\sfr$-triangle.
By definition we have $\ws([\FR{\ovl{t}}{\ovl{\del}}{\ovl{s}}])=[A\ov{Q(x\ci s)}{\lra}Y\ov{Q(t\ci y)}{\lra}C]$ and $\ws(\al\sas[\FR{\ovl{t}}{\ovl{\del}}{\ovl{s}}])=[A\ppr\ov{Q(x\ppr\ci s\ppr\ci u)}{\lra}Y\ppr\ov{Q(t\ci y\ppr)}{\lra}C]$.
Put $x\pprr=x\ppr\ci s\ppr\ci u$. Obviously, $(\ref{Morph_dels})$ induces the following morphism of $\ws$-triangles. 
\[
\xy
(-18,6)*+{A}="0";
(0,6)*+{Y}="2";
(18,6)*+{C}="4";
(32,6)*+{}="6";
(-18,-6)*+{A\ppr}="10";
(0,-6)*+{Y\ppr}="12";
(18,-6)*+{C}="14";
(32,-6)*+{}="16";
{\ar^{Q(x\ci s)} "0";"2"};
{\ar^{Q(t\ci y)} "2";"4"};
{\ar@{-->}^{[\FR{\ovl{t}}{\ovl{\del}}{\ovl{s}}]} "4";"6"};
{\ar_{\al} "0";"10"};
{\ar^{Q(b)} "2";"12"};
{\ar@{=} "4";"14"};
{\ar_{Q(x\pprr)} "10";"12"};
{\ar_{Q(t\ci y\ppr)} "12";"14"};
{\ar@{-->}_{\al\sas[\FR{\ovl{t}}{\ovl{\del}}{\ovl{s}}]} "14";"16"};
{\ar@{}|\circlearrowright "0";"12"};
{\ar@{}|\circlearrowright "2";"14"};
\endxy
\]
As we have $Q(t\ci y\ppr)\uas[\FR{\ovl{t}}{\ovl{\del}}{\ovl{s}}]=[\FR{\ovl{\id}_{Y\ppr}}{\ovl{y^{\prime\ast}\del}}{\ovl{s}}]$, it remains to show that
\begin{equation}\label{tobe_ws_tri}
A\ov{\left[\bsm Q(x\ci s)\\ \al\esm\right]}{\lra}Y\oplus A\ppr\ov{[Q(b)\ -Q(x\pprr)]}{\lra}Y\ppr\ov{[\FR{\ovl{\id}_{Y\ppr}}{\ovl{y^{\prime\ast}\del}}{\ovl{s}}]}{\dra}
\end{equation}
is an $\ws$-triangle.

Remark that since $(\ref{stri_mappingcone})$ is an $\sfr$-triangle, we have
\[ \ws([\FR{\ovl{\id}_Y}{\ovl{y^{\prime\ast}\del}}{\ovl{s}}])%
=[A\ov{Q(\left[\bsm x\\ a\ppr\esm\right]\ci s)}{\lra}Y\oplus X\ppr\ov{Q([b\ -x\ppr])}{\lra}Y\ppr] \]
by definition. Thus $(\ref{tobe_ws_tri})$ indeed becomes an $\ws$-triangle, since an isomorphism $\be=\id_Y\oplus Q(s\ppr\ci u)\in\wC(Y\oplus A\ppr,Y\oplus X\ppr)$ makes the diagram
\[
\xy
(-18,0)*+{A}="0";
(6,0)*+{}="1";
(0,10)*+{Y\oplus A\ppr}="2";
(0,-10)*+{Y\oplus X\ppr}="4";
(-6,0)*+{}="5";
(18,0)*+{Y\ppr}="6";
{\ar^{\left[\bsm Q(x\ci s)\\\al\esm\right]} "0";"2"};
{\ar_{\left[\bsm Q(x\ci s)\\ Q(a\ppr\ci s)\esm\right]} "0";"4"};
{\ar^{[ Q(b)\ -Q(x\pprr)]} "2";"6"};
{\ar_{[ Q(b)\ -Q(x\ppr)]} "4";"6"};
{\ar^{\cong}_{\be} "2";"4"};
{\ar@{}|\circlearrowright "0";"1"};
{\ar@{}|\circlearrowright "5";"6"};
\endxy
\]
commutative in $\wC$.
\end{proof}

\begin{proof}[Proof of Theorem~\ref{ThmMultLoc} {\rm (2)}]
We will show the universality of the functor $(Q,\mu)\co\CEs\to\wCEs$.
Let $(\D,\Fbb,\tfr)$ be any  weakly extriangulated category.

{\rm (i)} Let $(F,\phi)\co(\C,\Ebb,\sfr)\to (\D,\Fbb,\tfr)$ be an exact functor sending any $s\in\Ssc$ to an isomorphism.
It is obvious that $F$ uniquely factors through $Q$ via an additive functor $\wt{F}\co\wC\to\D$, namely, $F=\wt{F}\ci Q$.
We define a natural transformation $\wt{\phi}\co\wE\to\Fbb\ci(\wt{F}^\op\ti\wt{F})$ so that the pair $(\wt{F},\wt{\phi})$ becomes exact.
For each $A,C\in\C$, let $[\FR{\ovl{t}}{\ovl{\del}}{\ovl{s}}]=[C\ov{\ovl{t}}{\lla}Z\ov{\ovl{\del}}{\dra}X\ov{\ovl{s}}{\lla}A]\in\wE(C,A)$ be any element with $\sfr(\del)=[X\ov{x}{\lra}Y\ov{y}{\lra}Z]$.
Keeping in mind that $F(s)$ and $F(t)$ are isomorphisms,
we put $\wt{\phi}([\FR{\ovl{t}}{\ovl{\del}}{\ovl{s}}]):=(F(t)^{-1})^\ast(F(s)^{-1})_\ast(\phi(\delta))$.
Well-definedness and the additivity of $\wt{\phi}_{C,A}\co\wE(C,A)\to\Fbb(F(C),F(A))$ can be deduced easily from the definition of the equivalence relation given in Corollary~\ref{CorwETrans}.
Moreover, we get
\begin{equation}\label{UnivRealization}
\begin{split}
\tfr(\wt{\phi}([\FR{\ovl{t}}{\ovl{\del}}{\ovl{s}}]))&=[FA\ov{F(x\ci s)}{\lra}FY\ov{F(t\ci y)}{\lra}FC]\\
&=[\wt{F}QA\ov{\wt{F}Q(x\ci s)}{\lra}\wt{F}QY\ov{\wt{F}Q(t\ci y)}{\lra}\wt{F}QC]
\end{split}
\end{equation}
by Remark~\ref{RemWE}.
Naturality of $\wt{\phi}$ can be checked as follows.
Consider a morphism $\alpha=\ovl{Q}(\ovl{u})^{-1}\ci\ovl{Q}(\ovl{a})\co A\to A\ppr$ and the commutative square same as (\ref{CSby(MS2)}).
Then we have the following equalities.
\begin{eqnarray*}
\wt{\phi}(\alpha_\ast[\FR{\ovl{t}}{\ovl{\del}}{\ovl{s}}]) &=& \wt{\phi}([\FR{\ovl{t}}{\ovl{a\ppr_\ast\del}}{\ovl{s\ppr\ci u}}])\\
&=& (F(t)^{-1})^\ast(F(u)^{-1}\ci F(s\ppr)^{-1})_\ast(\phi(a\ppr_\ast\delta))\\
&=& (F(t)^{-1})^\ast(F(u)^{-1}\ci F(s\ppr)^{-1}\ci F(a\ppr))_\ast(\phi(\delta))\\
&=& (F(t)^{-1})^\ast(F(u)^{-1}\ci F(a)\ci F(s)^{-1})_\ast(\phi(\delta))\\
&=& \wt{F}(\alpha)_\ast\wt{\phi}([\FR{\ovl{t}}{\ovl{\del}}{\ovl{s}}])
\end{eqnarray*}
Similarly, we have $\wt{\phi}(\beta^\ast[\FR{\ovl{t}}{\ovl{\del}}{\ovl{s}}])=\wt{F}(\beta)^\ast\wt{\phi}([\FR{\ovl{t}}{\ovl{\del}}{\ovl{s}}])$ for any $\beta\in\wC(C\ppr, C)$.

By construction, for any $\del\in\Ebb(C,A)$ we have
\begin{equation}\label{P22}
\wt{\phi}_{C,A}(\mu_{C,A}(\del))=\wt{\phi}_{C,A}([\FR{\ovl{\id}}{\ovl{\del}}{\ovl{\id}}])=\phi(\del),
\end{equation}
which shows $(\wt{\phi}\ci(Q^\op\ti Q))\cdot\mu=\phi$, and thus $(\wt{F},\wt{\phi})\ci(Q,\mu)=(F,\phi)$ holds. It is obvious that $\wt{\phi}$ is uniquely determined by $(\ref{P22})$ (and the requirement of $(\ref{P7})$ for natural transformations in Definition~\ref{DefExFun} {\rm (3)}), as in the above construction.

{\rm (ii)} Let $\eta\co(F,\phi)\tc(G,\psi)$ be any natural transformation of exact functors between $(F,\phi),(G,\psi)\co\CEs\to (\D,\Fbb,\tfr)$ which send any $s\in\Ssc$ to isomorphisms.
By the universality of localization for additive categories, there exists a unique natural transformation $\wt{\eta}\co\wt{F}\tc \wt{G}$ of additive functors such that $\wt{\eta}\ci Q=\eta$. Indeed, this is given by $\wt{\eta}_C=\eta_C$ for any $C\in\C$. It is obvious from the construction and the naturality of $\eta$, that such $\wt{\eta}$ satisfies
\[ (\wt{\eta}_A)\sas\wt{\phi}_{C,A}([\FR{\ovl{t}}{\ovl{\del}}{\ovl{s}}])=%
(\wt{\eta}_C)\uas\wt{\psi}_{C,A}([\FR{\ovl{t}}{\ovl{\del}}{\ovl{s}}]) \]
for any $[\FR{\ovl{t}}{\ovl{\del}}{\ovl{s}}]\in\wE(C,A)$.
\end{proof}

To show {\rm (3)}, we need the following lemma.
\begin{lem}\label{LemComposeInf}
The following holds for any morphism $\al$ in $\wC$.
\begin{enumerate}
\item $\al$ is an $\ws$-inflation in $\wC$ if and only if $\al=\be\ci Q(f)\ci \gam$ holds for some $\sfr$-inflation $f$ in $\C$ and isomorphisms $\be,\gam$ in $\wC$.
\item $\al$ is an $\ws$-deflation in $\wC$ if and only if $\al=\be\ci Q(f)\ci\gam$ holds for some $\sfr$-deflation $f$ in $\C$ and isomorphisms $\be,\gam$ in $\wC$.
\end{enumerate}
\end{lem}
\begin{proof}
By duality, it is enough to show {\rm (1)}.
By the definition of $\ws$ it is obvious that any $\ws$-inflation can be written as $\be\ci Q(f)\ci \gam$ for some $\sfr$-inflation $f$ and isomorphisms $\be,\gam$ in $\wC$.

Let us show the converse. Take any $\sfr$-triangle $X\ov{f}{\lra}Y\ov{g}{\lra}Z\ov{\del}{\dra}$ and isomorphisms $\gam\in\wC(A,X),\be\in\wC(Y,B)$ in $\wC$. Express $\gam\iv$ as $\gam\iv=\ovl{Q}(\ovl{t})\iv\ci\ovl{Q}(\ovl{e})$, with some $\ovl{e}\in\ovl{\C}(X,X\ppr)$ and $\ovl{t}\in\ovl{\Ssc}(A,X\ppr)$. Remark that $\ovl{Q}(\ovl{e})\in\Iso(\wC)$ and $\gam=\ovl{Q}(\ovl{e})\iv\ci\ovl{Q}(\ovl{t})$ hold. For $\sfr(e\sas\del)=[X\ppr\ov{f\ppr}{\lra}Y\ppr\ov{g\ppr}{\lra}Z]$, we obtain a morphism of $\sfr$-triangles
\[
\xy
(-14,6)*+{X}="0";
(0,6)*+{Y}="2";
(14,6)*+{Z}="4";
(27,6)*+{}="6";
(-14,-6)*+{X\ppr}="10";
(0,-6)*+{Y\ppr}="12";
(14,-6)*+{Z}="14";
(27,-6)*+{}="16";
{\ar^{f} "0";"2"};
{\ar^{g} "2";"4"};
{\ar@{-->}^{\del} "4";"6"};
{\ar_{e} "0";"10"};
{\ar_{d} "2";"12"};
{\ar@{=} "4";"14"};
{\ar_{f\ppr} "10";"12"};
{\ar_{g\ppr} "12";"14"};
{\ar@{-->}_{e\sas\del} "14";"16"};
{\ar@{}|\circlearrowright "0";"12"};
{\ar@{}|\circlearrowright "2";"14"};
\endxy
\]
for some $d\in\C(Y,Y\ppr)$, by {\rm (C3)} for $\CEs$. Since $(Q,\mu)$ is an exact functor, 
\[
\xy
(-16,6)*+{X}="0";
(0,6)*+{Y}="2";
(16,6)*+{Z}="4";
(30,6)*+{}="6";
(-16,-6)*+{X\ppr}="10";
(0,-6)*+{Y\ppr}="12";
(16,-6)*+{Z}="14";
(30,-6)*+{}="16";
{\ar^{Q(f)} "0";"2"};
{\ar^{Q(g)} "2";"4"};
{\ar@{-->}_(0.56){\mu_{Z,X}(\del)} "4";"6"};
{\ar_{Q(e)} "0";"10"};
{\ar_{Q(d)} "2";"12"};
{\ar@{=} "4";"14"};
{\ar_{Q(f\ppr)} "10";"12"};
{\ar_{Q(g\ppr)} "12";"14"};
{\ar@{-->}_(0.56){\mu_{Z,X\ppr}(e\sas\del)} "14";"16"};
{\ar@{}|\circlearrowright "0";"12"};
{\ar@{}|\circlearrowright "2";"14"};
\endxy
\]
becomes a morphism of $\ws$-triangles. In particular we have $Q(d)\in\Iso(\wC)$ by Remark~\ref{RemWE}.
By the definition of $\ws$, we have
$\ws([\FR{\ovl{\id}}{\ovl{e\sas\del}}{\ovl{t}}]=[A\ov{Q(f\ppr\ci t)}{\lra}Y\ppr\ov{Q(g\ppr)}{\lra}Z]$.
Since $\lam=\be\ci Q(d)\iv$ makes
\[
\xy
(-16,0)*+{A}="0";
(4,0)*+{}="1";
(0,8)*+{Y\ppr}="2";
(0,-8)*+{B}="4";
(-4,0)*+{}="5";
(16,0)*+{Z}="6";
{\ar^{Q(f\ppr\ci t)} "0";"2"};
{\ar_{\be\ci Q(f)\ci\gam} "0";"4"};
{\ar^{Q(g\ppr)} "2";"6"};
{\ar_{Q(g\ppr)\ci\lam\iv} "4";"6"};
{\ar^{\cong}_{\lam} "2";"4"};
{\ar@{}|\circlearrowright "0";"1"};
{\ar@{}|\circlearrowright "5";"6"};
\endxy
\]
commutative in $\wC$, it follows that $\be\ci Q(f)\ci\gam$ is an $\ws$-inflation.
\end{proof}

\begin{proof}[Proof of Theorem~\ref{ThmMultLoc} {\rm (3)}]
Suppose that $\ovl{\Ssc}$ moreover satisfies {\rm (MR4)}. Let $X_1\ov{\al_1}{\lra}X_2$ and $X_2\ov{\al_2}{\lra}X_3$ be $\ws$-inflations. By Lemma~\ref{LemComposeInf}, for $i=1,2$ we may write $\al_i$ as $\al_i=\be_i\ci Q(f_i)\ci\gam_i$ with some $\be_i,\gam_i\in\Iso(\wC)$ and an $\sfr$-inflation $f_i$ in $\C$. It suffices to show that their composition $\al_2\ci\al_1$ becomes again of this form.
Express $(\gam_2\ci\be_1)\iv$ as $(\gam_2\ci\be_1)\iv=\ovl{Q}(\ovl{s})\iv\ci\ovl{Q}(\ovl{e})$ for some morphisms $\ovl{e},\ovl{s}$ in $\ovl{\C}$ with $\ovl{s}\in\ovl{\Ssc}$. We have $\ovl{Q}(\ovl{e})\in\Iso(\wC)$ and $\gam_2\ci\be_1=\ovl{Q}(\ovl{e})\iv\ci\ovl{Q}(\ovl{s})$. Denote the domains and codomains of $f_2$ and $e$ by $A\ov{f_2}{\lra}B$ and $A\ov{e}{\lra}A\ppr$, respectively. Since $f_2$ is an $\sfr$-inflation, there exists an $\sfr$-triangle $A\ov{f_2}{\lra}B\ov{g}{\lra}C\ov{\del}{\dra}$. Also to $e\sas\del$, another $\sfr$-triangle $A\ppr\ov{f_2\ppr}{\lra}B\ppr\ov{g\ppr}{\lra}C\ov{e\sas\del}{\dra}$ is associated. By {\rm (C3)} for $\CEs$, we obtain a morphism of $\sfr$-triangles
\[
\xy
(-14,6)*+{A}="0";
(0,6)*+{B}="2";
(14,6)*+{C}="4";
(27,6)*+{}="6";
(-14,-6)*+{A\ppr}="10";
(0,-6)*+{B\ppr}="12";
(14,-6)*+{C}="14";
(27,-6)*+{}="16";
{\ar^{f_2} "0";"2"};
{\ar^{g} "2";"4"};
{\ar@{-->}^{\del} "4";"6"};
{\ar_{e} "0";"10"};
{\ar_{d} "2";"12"};
{\ar@{=} "4";"14"};
{\ar_{f_2\ppr} "10";"12"};
{\ar_{g\ppr} "12";"14"};
{\ar@{-->}_{e\sas\del} "14";"16"};
{\ar@{}|\circlearrowright "0";"12"};
{\ar@{}|\circlearrowright "2";"14"};
\endxy
\]
for some $d\in\C(B,B\ppr)$.
Similarly as in the proof of Lemma~\ref{LemComposeInf}, we have $Q(d)\in\Iso(\wC)$. Applying {\rm (MR4)} to $f_2\ppr\ci s\ci f_1$, we obtain $\ovl{f}_2\ppr\ci \ovl{s}\ci \ovl{f}_1=\vbf\ci\ovl{x}\ci\ubf$ for some $\ubf,\vbf\in\ovl{\Ssc}$ and an $\sfr$-inflation $x$. Then we have
\begin{eqnarray*}
\al_2\ci\al_1
&=&
\big(\be_2\ci\ovl{Q}(\ovl{f}_2)\ci\gam_2\big)\ci\big(\be_1\ci\ovl{Q}(\ovl{f}_1)\ci\gam_1\big)\\
&=&
\big(\be_2\ci Q(d)\iv\ci\ovl{Q}(\vbf)\big)\ci\ovl{Q}(\ovl{x})\ci\big(\ovl{Q}(\ubf)\ci\gam_1\big).
\end{eqnarray*}
Since $\be_2\ci Q(d)\iv\ci\ovl{Q}(\vbf)$ and $\ovl{Q}(\ubf)\ci\gam_1$ belong to $\Iso(\wC)$, this indeed shows that $\al_2\ci\al_1$ is an $\ws$-inflation by Lemma~\ref{LemComposeInf}. Dually for $\ws$-deflations.
\end{proof}

Thus our main theorem is proved. We conclude this subsection with some remarks and consequences of Theorem~\ref{ThmMultLoc}.

\begin{rem}
Suppose that $\ovl{\Ssc}$ satisfies {\rm (M0),(MR1),(MR2),(MR3)} and $p(\Ssc)=\ovl{\Ssc}$, as before. Lemma~\ref{LemComposeInf} shows that $\wCEs$ becomes extriangulated if and only if $\ovl{\Ssc}$ satisfies the following {\rm (MR4$^{-}$)}, which is a weaker variant of {\rm (MR4)}.
\begin{itemize}
\item[{\rm (MR4$^{-}$)}] For any sequence of morphisms $X\ov{\mathbf{f}}{\lra}Y\ov{\mathbf{g}}{\lra}Z$ with $\mathbf{f},\mathbf{g}\in\ovl{\Mcal}_{\mathsf{inf}}$, there exist morphisms $\abf,\bbf$ in $\ovl{\C}$ such that $\abf\ci\mathbf{g}\ci\mathbf{f}\ci\bbf\in\ovl{\Mcal}_{\mathsf{inf}}$ and $\ovl{Q}(\abf),\ovl{Q}(\bbf)\in\Iso(\wC)$. 
Dually for $\ovl{\Mcal}_{\mathsf{def}}\se\ovl{\Mcal}$. 
\end{itemize}
\end{rem}

\begin{rem}\label{RemAdm}
Remark that if $\CEs$ corresponds to a triangulated category, or to an exact category which is moreover abelian, then any morphism $f$ in $\C$ is $\sfr$-\emph{admissible} in the sense that $f=m\ci e$ holds for an $\sfr$-deflation $e$ and an $\sfr$-inflation $m$.
Conversely, if $\CEs$ corresponds to an exact category in which any morphism is $\sfr$-admissible, then $\C$ is abelian.

By Lemma~\ref{LemComposeInf}, if any morphism in $\C$ is $\sfr$-admissible, then any morphism in $\wC$ is $\ws$-admissible. In particular if $\wCEs$ corresponds to an exact category and any morphism in $\C$ is $\sfr$-admissible, then $\wC$ is an abelian category.
\end{rem}


In the following way, extriangulated categories obtained by ideal quotients can be also seen as a particular type of the localization.
\begin{rem}\label{Rem_IdealQuotLoc}
Let $\Ical\se\C$ be any full additive subcategory closed by isomorphisms and direct summands, whose objects are both projective and injective. Let $p\co \C\to \C/[\Ical]$ denote the ideal quotient. If we put $\Ssc=p\iv(\Iso(\C/[\Ical]))$, then it is easy to see that the ideal $[\Ncal_\Ssc]$ coincides with $[\Ical]$, hence we obtain $\ovl{\C}=\C/[\Ical]$.

As shown in \cite[Proposition 3.30]{NP}, it has a natural extriangulated structure $(\ovl{\C},\ovl{\E},\ovl{\sfr})$, given by $\ovl{\E}(C,A)=\E(C,A)$ for any $A,C\in\C$ and $\ovl{\sfr}(\del)=[A\ov{\ovl{x}}{\lra}B\ov{\ovl{y}}{\lra}C]$ for any $\del\in\E(C,A)$ with $\sfr(\del)=[A\ov{x}{\lra}B\ov{y}{\lra}C]$. 
It is obvious that $\phi=\{\phi_{C,A}=\id\co\E(C,A)\to\ovl{\E}(C,A)\}_{C,A\in\C}$ gives an exact functor $(p,\phi)\co\CEs\to(\ovl{\C},\ovl{\E},\ovl{\sfr})$. 

Since $\ovl{\Ssc}=\Iso(\ovl{\C})$ trivially satisfies {\rm (MR2)}, we see that Corollary~\ref{CorOfThm} can be applied to $(p,\phi)$.
We remark that $\ovl{\Ssc}$ also satisfies {\rm (MR4)}. Indeed, let $A\ov{f}{\lra}B\ov{s}{\lra}X\ov{x}{\lra}Y$ be a sequence of morphisms in $\C$ in which $f,x$ are $\sfr$-inflations and $s$ belongs to $\Ssc$. Then there exists $t\in\Ssc(X,B)$ such that $\ovl{t}\iv=\ovl{s}$, and by {\rm (M3)} we obtain a morphism of $\sfr$-triangles
\[
\xy
(-12,6)*+{X}="0";
(0,6)*+{Y}="2";
(12,6)*+{Z}="4";
(24,6)*+{}="6";
(-12,-6)*+{B}="10";
(0,-6)*+{Y\ppr}="12";
(12,-6)*+{Z}="14";
(24,-6)*+{}="16";
{\ar^{x} "0";"2"};
{\ar^{} "2";"4"};
{\ar@{-->}^{} "4";"6"};
{\ar_{t} "0";"10"};
{\ar_{u} "2";"12"};
{\ar@{=} "4";"14"};
{\ar_{x\ppr} "10";"12"};
{\ar_{} "12";"14"};
{\ar@{-->}_{} "14";"16"};
{\ar@{}|\circlearrowright "0";"12"};
{\ar@{}|\circlearrowright "2";"14"};
\endxy
\]
for some $u\in\Ssc(Y,Y\ppr)$. This gives $\ovl{x\ci s\ci f}=\ovl{u}\iv\ci\ovl{x\ppr\ci t\ci s \ci f}=\ovl{u}\iv\ci\ovl{x\ppr\ci f}$. Then $x\ppr\ci f$ is an $\sfr$-inflation by {\rm (C4)} for $\CEs$. Since $\Ssc$ is closed under compositions, this shows that $\ovl{\Mcal}_{\mathsf{inf}}\se\ovl{\Mcal}$ is closed by compositions. Dually for $\ovl{\Mcal}_{\mathsf{def}}$.

Thus $\wCEs$ obtained in Corollary~\ref{CorOfThm} is extriangulated.
As the proof of Theorem~\ref{ThmMultLoc} suggests, localization by $\Ssc$ does not change $(\ovl{\C},\ovl{\E},\ovl{\sfr})$ essentially, and there is an equivalence of extriangulated categories $(\ovl{\C},\ovl{\E},\ovl{\sfr})\ov{\simeq}{\lra}\wCEs$ in the sense of {\rm (2)} in Proposition~\ref{PropExEq}, in an obvious way.
\end{rem}


The following is a corollary of Theorem~\ref{ThmMultLoc}, which gives a condition for the resulting localization to correspond to an exact category. This will be used in Subsection~\ref{Subsection_Percolating}.
\begin{cor}\label{CorLocExact}
Let $\Ssc\se\Mcal$ and $\Ncal_{\Ssc}$ be as in Theorem~\ref{ThmMultLoc} {\rm (3)}. Assume moreover that it satisfies the following conditions {\rm (i),(ii)} and their duals.
Then the resulting localization $\wCEs$ corresponds to an exact category. 
Moreover, if any morphism in $\C$ is $\sfr$-admissible, then $\wC$ is an abelian category.
\begin{itemize}
\item[{\rm (i)}] $\Ker\big(\C(X,A)\ov{x\ci-}{\lra}\C(X,B)\big)\se[\Ncal_{\Ssc}](X,A)$ holds for any $X\in\C$ and any $\sfr$-inflation $x\in\C(A,B)$. 
\item[{\rm (ii)}] For any $N\in\Ncal_{\Ssc}$ and any morphism $f\in\C(A,N)$, there exists a morphism $s\in\C(A\ppr,A)$ such that $f\ci s=0$ and $Q(s)\in\Iso(\wC)$. %
\end{itemize}
\end{cor}
\begin{proof}
By duality, it is enough to show that any $\ws$-inflation is monic in $\wC$. By Lemma~\ref{LemComposeInf}, it suffices to show that $Q(x)$ is monic in $\wC$ for any $\sfr$-inflation $x\in\C(A,B)$.
Since any morphism $\al\in\wC(X,A)$ can be written as $\al=\ovl{Q}(\ovl{f})\ci\ovl{Q}(\ovl{u})\iv$ for some $\ovl{f}\in\ovl{\C}(X\ppr,A)$ and $\ovl{u}\in\ovl{\Ssc}(X\ppr,X)$, it suffices to show that $Q(x\ci f)=0$ implies $Q(f)=0$ for any $f\in\C(X\ppr,A)$.

By {\rm (MR2)}, we see that $Q(x\ci f)=0$ is equivalent to the existence of $\ovl{v}\in\ovl{\Ssc}(Y,X\ppr)$ such that $\ovl{x}\ci\ovl{f}\ci\ovl{v}=0$. Thus $j\ci i=x\ci f\ci v$ holds for some $N\in\Ncal_{\Ssc}$ and $i\in\C(Y,N)$, $j\in\C(N,B)$.
By {\rm (ii)}, 
there exists $s\in\C(Y\ppr,Y)$ which satisfies $i\ci s=0$ and $Q(s)\in\Iso(\wC)$. We have $x\ci f\ci v\ci s=j\ci i\ci s=0$, hence $\ovl{f}\ci \ovl{v}\ci \ovl{s}=0$ holds by {\rm (i)}. This implies $Q(f)\ci Q(v)\ci Q(s)=0$, hence $Q(f)=0$ since $Q(v),Q(s)\in\Iso(\wC)$.

The last assertion is immediate from Remark~\ref{RemAdm}.
\end{proof}

\begin{rem}\label{RemLocExact}
In Corollary~\ref{CorLocExact}, we see that the following holds.
\begin{enumerate}
\item Condition {\rm (i)} and its dual are trivially satisfied if $\CEs$ corresponds to an exact category.
\item Condition {\rm (ii)} and its dual are satisfied if $\C$ is abelian (or similarly for $\CEs$ in which any morphism is $\sfr$-admissible) and if $\Ssc$ is the multiplicative system associated to a Serre subcategory $\Ncal\se\C$. 
\end{enumerate}
\end{rem}

\section{Relation to known constructions} \label{Section_Examples}

\subsection{Thick subcategories}

Before listing the known constructions which we are going to deal with, let us introduce the notion of a thick subcategory in an extriangulated category, which serves to control corresponding $\Ssc$.
\begin{dfn}\label{DefThick}
A full additive subcategory $\Ncal\se\C$ is called a \emph{thick} subcategory if it satisfies the following conditions. 
\begin{itemize}
\item[{\rm (i)}] $\Ncal\se\C$ is closed by isomorphisms and direct summands.
\item[{\rm (ii)}] $\Ncal$ satisfies $2$-out-of-$3$ for $\sfr$-conflations. Namely, if any two of objects $A,B,C$ in an $\sfr$-conflation $A\ov{x}{\lra}B\ov{y}{\lra}C$ belong to $\Ncal$, then so does the third.
\end{itemize}
\end{dfn}

\begin{rem}
The following is obvious from the definition.
\begin{enumerate}
\item A thick subcategory $\Ncal\se\C$ is extension-closed, hence is an extriangulated category.
\item If $(F,\phi)\co\CEs\to(\D,\Fbb,\tfr)$ is an exact functor to a weakly extriangulated category $\D$, then $\Ker F\se\C$ is a thick subcategory.
\end{enumerate}
\end{rem}

\begin{dfn}\label{DefMorphClassesLR}
For a thick subcategory $\Ncal\se\C$, we associate the following classes of morphisms.
\begin{eqnarray*}
\Lcal&=&\{ f\in\Mcal\mid f\ \text{is an}\ \sfr\text{-inflation with}\ \Cone(f)\in\Ncal\}.\\
\Rcal&=&\{ f\in\Mcal\mid f\ \text{is an}\ \sfr\text{-deflation with}\ \CoCone(f)\in\Ncal\}.
\end{eqnarray*}

Define $\Ssc_{\Ncal}\se\Mcal$ to be the smallest subset closed by compositions containing both $\Lcal$ and $\Rcal$. It is obvious that $\Ssc_{\Ncal}$ satisfies condition {\rm (M0)} in Section~\ref{Section_Localization}.
\end{dfn}

\begin{rem}\label{RemSN}
$\Ssc_{\Ncal}$ coincides with the set of all finite compositions of morphisms in $\Lcal$ and $\Rcal$.
\end{rem}

\begin{lem}\label{LemNSN}
$\Ncal_{\Ssc_{\Ncal}}=\Ncal$ holds.
\end{lem}
\begin{proof}
Clearly, $\Ncal\se\Ncal_{\Ssc_{\Ncal}}$ holds. We show that the converse inclusion holds. Let $X$ be an arbitrary object in $\Ncal_{\Ssc_{\Ncal}}$. Then the morphism $0\to X$ belongs to $\Ssc_{\Ncal}$. By Remark~\ref{RemSN}, it suffices to show the following claim:
For any morphism $f\co N\to Y$ with $N\in\Ncal$, if $f$ is contained in either $\Lcal$ or $\Rcal$, then $Y$ is contained in $\Ncal$. This follows immediately from that $\Ncal$ is a thick subcategory. 
\end{proof}

\begin{lem}\label{LemM3}
$\Ssc_{\Ncal}$ satisfies {\rm (M3)} in Corollary~\ref{CorMultLoc}.
\end{lem}
\begin{proof}
Let $A\ov{x}{\lra}B\ov{y}{\lra}C\ov{\del}{\dra}$ and $A\ppr\ov{x\ppr}{\lra}B\ppr\ov{y\ppr}{\lra}C\ppr\ov{\del\ppr}{\dra}$ be any pair of $\sfr$-triangles, and suppose that $a\in\Ssc_{\Ncal}(A,A\ppr),c\in\Ssc_{\Ncal}(C,C\ppr)$ satisfy $a\sas\del=c\uas\del\ppr$. It suffices to show the existence of $b\in\Ssc_{\Ncal}$ which gives a morphism of $\sfr$-triangles
\[
\xy
(-12,6)*+{A}="0";
(0,6)*+{B}="2";
(12,6)*+{C}="4";
(24,6)*+{}="6";
(-12,-6)*+{A\ppr}="10";
(0,-6)*+{B\ppr}="12";
(12,-6)*+{C\ppr}="14";
(24,-6)*+{}="16";
{\ar^{x} "0";"2"};
{\ar^{y} "2";"4"};
{\ar@{-->}^{\del} "4";"6"};
{\ar_{a} "0";"10"};
{\ar_{b} "2";"12"};
{\ar^{c} "4";"14"};
{\ar_{x\ppr} "10";"12"};
{\ar_{y\ppr} "12";"14"};
{\ar@{-->}_{\del\ppr} "14";"16"};
{\ar@{}|\circlearrowright "0";"12"};
{\ar@{}|\circlearrowright "2";"14"};
\endxy.
\]
We may assume that either $c$ or $a$ equals to the identity map. Suppose $c=\id$. By Remark~\ref{RemSN}, we may assume that $a$ belongs to either $\Rcal$ or $\Lcal$. Then the assertion follows from {\rm (ET4)}, \cite[Proposition 3.15]{NP}, respectively.
\end{proof}

We prepare some lemmas which will be used in the proceeding subsections.
\begin{lem}\label{LemThickFirstProperties}
The following holds for any thick subcategory $\Ncal\se\C$.
\begin{enumerate}
\item $\Lcal,\Rcal\se\Mcal$ contain all isomorphisms, and are closed by compositions in $\Mcal$.
\item For any $l\in\Lcal$, morphism $\ovl{l}$ is epimorphic in $\ovl{\C}$. Dually, for any $r\in\Rcal$, morphism $\ovl{r}$ is monomorphic in $\ovl{\C}$.
\item For any $X\ov{l}{\lla}X\ppr\ov{g}{\lra}Y$ with $l\in\Lcal$, there exists a commutative square
\[
\xy
(-6,6)*+{X\ppr}="0";
(6,6)*+{Y}="2";
(-6,-6)*+{X}="4";
(6,-6)*+{Y\ppr}="6";
{\ar^{g} "0";"2"};
{\ar_{l} "0";"4"};
{\ar^{l\ppr} "2";"6"};
{\ar_{g\ppr} "4";"6"};
{\ar@{}|\circlearrowright "0";"6"};
\endxy
\]
in $\C$ such that $l\ppr\in\Lcal$. Moreover if $g\in\Rcal$, then it can be chosen to satisfy $l\ppr\in\Lcal$ and $g\ppr\in\Rcal$.

Dually, for any $X\ov{f}{\lra}Y\ov{r}{\lla}Y\ppr$ with $r\in\Rcal$, there exists a commutative square
\[
\xy
(-6,6)*+{X\ppr}="0";
(6,6)*+{Y\ppr}="2";
(-6,-6)*+{X}="4";
(6,-6)*+{Y}="6";
{\ar^{f\ppr} "0";"2"};
{\ar_{r\ppr} "0";"4"};
{\ar^{r} "2";"6"};
{\ar_{f} "4";"6"};
{\ar@{}|\circlearrowright "0";"6"};
\endxy
\]
in $\C$ such that $r\ppr\in\Rcal$. Moreover if $f\in\Lcal$, then it can be chosen to satisfy $r\ppr\in\Rcal$ and $f\ppr\in\Lcal$.
\end{enumerate}
\end{lem}
\begin{proof}
(1) Since $\Ncal\se\C$ is extension-closed, this is straightforward.

(2) Let $X\ppr\ov{l}{\lra}X$ be a morphism in $\Lcal$, and suppose that $\ovl{f}\ci\ovl{l}=0$ holds for $X\ov{f}{\lra}Y$.
The composed morphism $f\ci l$ can be factored as $X\ppr\ov{a}{\lra}N\ov{}{\lra}Y$ with $N\in\Ncal$.
Condition (C3) yields a morphism of $\sfr$-triangles
\[
\xy
(-12,6)*+{X\ppr}="0";
(0,6)*+{X}="2";
(12,6)*+{N\ppr}="3";
(24,6)*+{}="p";
(-12,-6)*+{N}="4";
(0,-6)*+{X\pprr}="6";
(12,-6)*+{N\ppr}="7";
(24,-6)*+{}="q";
{\ar^{l} "0";"2"};
{\ar "2";"3"};
{\ar_{a} "0";"4"};
{\ar^{a\ppr} "2";"6"};
{\ar_{i} "4";"6"};
{\ar "6";"7"};
{\ar@{}|\circlearrowright "0";"6"};
{\ar@{}|\circlearrowright "2";"7"};
{\ar@{=} "3";"7"};
{\ar@{-->} "3";"p"};
{\ar@{-->} "7";"q"};
\endxy
\]
which gives an $\sfr$-conflation $X\ppr\ov{\left[\bsm l\\ a\esm\right]}{\lra}X\oplus N\ov{[a\ppr\ -i]}{\lra}X\pprr$.
Note that $X\pprr\in\Ncal$ follows from $N,N\ppr\in\Ncal$. Since $f$ factors through $X\pprr$, this shows $\ovl{f}=0$. Thus the former assertion is shown, and dually for the latter.

(3) We only check the former assertion. Dually for the latter.
Due to {\rm (C3)}, a desired square exists.
If $g\in\Rcal$, by using {\rm (ET4)}, we have a choice to satisfy $l\ppr\in\Lcal$ and $g\ppr\in\Rcal$.
\end{proof}

\subsection{Typical cases}

Examples we have in mind are the following. 

\begin{ex}\label{ExVerdier}
{\rm (}\emph{Verdier quotient.}{\rm )} Suppose that $\C$ is a triangulated category. If we regard it as an extriangulated category, then Definition~\ref{DefThick} agrees with the usual definition of a thick subcategory.
In this case we have $\Ssc_{\Ncal}=\Lcal=\Rcal$ and the localization of $\C$ by $\Ssc_{\Ncal}$ naturally becomes triangulated. 
\end{ex}

\begin{ex}\label{ExSerreAbel}
{\rm (}\emph{Serre quotient.}{\rm )} Suppose that $\C$ is an abelian category, which we may regard as an extriangulated category. Then any Serre subcategory of $\C$ is thick in the sense of Definition~\ref{DefThick}.
In this case we have $\Ssc_{\Ncal}=\Lcal\ci\Rcal$, and the localization of $\C$ by $\Ssc_{\Ncal}$ becomes abelian.
\end{ex}

\begin{ex}\label{ExTwo-sidedExact}
Suppose that $\C$ is an exact category, and that $\Ncal\se\C$ is a full additive subcategory closed by isomorphisms. 
As in \cite[Definition~2.4]{HKR}, it is said to be a \emph{two-sided admissibly percolating subcategory} if it satisfies the following conditions.
\begin{itemize}
\item[{\rm (i)}] For any conflation $A\to B\to C$, we have $B\in\Ncal$ if and only if $A,C\in\Ncal$.
\item[{\rm (ii)}] For any morphism $f\co X\to A$ with $A\in\Ncal$, there is a factorization of $f$ as $X\ov{g}{\to}A\ppr\ov{h}{\to}A$ with a deflation $g$ and an inflation $h$.
\item[{\rm (iii)}] Dual of (ii).
\end{itemize}
In this case, we have $\Ssc_{\Ncal}=\Lcal\ci\Rcal$ and the localization of $\C$ by $\Ssc_{\Ncal}$ becomes an exact category such that the localization functor is an exact functor. 
\end{ex}

\begin{rem}
As for localizations for exact categories, there is a precedent result by C\'{a}rdenas-Escudero \cite{C-E}.
As stated in \cite[Theorem 8.1]{HR}, this can be dealt as a particular case of Example~\ref{ExTwo-sidedExact}.
\end{rem}

\begin{ex}\label{ExRump}
Suppose that $\C$ is an exact category, and that $\Ncal\se\C$ is a full additive subcategory closed by isomorphisms. 
As in \cite[Definition~7]{R}, it is said to be \emph{biresolving} if it satisfies the following conditions.
\begin{itemize}
\item[{\rm (i)}] For any $C\in\C$, there is an inflation $C\to N$ and a deflation $N\ppr\to C$ for some $N,N\ppr\in\Ncal$.
\item[{\rm (ii)}] $\Ncal$ satisfies $2$-out-of-$3$ for conflations.
\end{itemize}

In particular, any biresolving subcategory $\Ncal\se\C$ closed by direct summands is a thick subcategory in the sense of Definition~\ref{DefThick}.

In \cite[Proposition~12]{R}, it is shown that for any biresolving subcategory of an exact category, localization of $\ovl{\C}=\C/[\Ncal]$ by the morphisms which are both monomorphic and epimorphic, has a natural structure of a triangulated category.
\end{ex}

\begin{rem}\label{RemRump}
Since replacing $\Ncal$ by $\add\Ncal$ does not affect the resulting localization, in this article we only deal with the case where $\Ncal\se\C$ is thick. Here $\add\Ncal\se\C$ denotes the smallest additive full subcategory containing $\Ncal$, closed by direct summands.
\end{rem}

\begin{ex}\label{ExHTCP}
Let $\CEs$ be an extriangulated category. As defined in \cite[Definition~5.1]{NP}, a \emph{Hovey twin cotorsion pair} is a quartet $((\Scal,\Tcal),(\Ucal,\Vcal))$ of full subcategories closed by isomorphisms and direct summands satisfying
\begin{itemize}
\item[{\rm (i)}] $(\Scal,\Tcal)$ and $(\Ucal,\Vcal)$ are cotorsion pairs on $\C$.
\item[{\rm (ii)}] $\Ebb(\Scal,\Vcal)=0$.
\item[{\rm (iii)}] $\Cone(\Vcal,\Scal)=\CoCone(\Vcal,\Scal)$.
\end{itemize}

If $((\Scal,\Tcal),(\Ucal,\Vcal))$ is a Hovey twin cotorsion pair, then $\Ncal=\Cone(\Vcal,\Scal)$ satisfies $2$-out-of-$3$ for $\sfr$-conflations by \cite[Proposition~5.3]{NP}. 

It is also immediate that $\Ncal\se\C$ is closed by direct summands. Indeed if $X\oplus Y\in\Ncal$, then for $\sfr$-conflations $T\to S\to X$ and $T\ppr\to S\ppr\to Y$ satisfying $S,S\ppr\in\Scal$ and $T,T\ppr\in\Tcal$, we see that the $\sfr$-conflation $T\oplus T\ppr\to S\oplus S\ppr \to X\oplus Y$ should satisfy $T\oplus T\ppr\in\Ncal\cap\Tcal=\Vcal$ by the $2$-out-of-$3$ property. This implies $T,T\ppr\in\Vcal$, hence $X,Y\in\Ncal$.
Thus $\Ncal\se\C$ is a thick subcategory.

If moreover $\CEs$ satisfies condition {\rm (WIC)}, then the following holds by \cite[Corollaries~5.12, 5.22]{NP}.
\begin{itemize}
\item $\mathit{wCof}\se\Lcal$ and $\mathit{wFib}\se\Rcal$ hold in the notation of \cite[Definition~5.11]{NP},
\item $\Wbb=\mathit{wFib}\ci\mathit{wCof}\se\Mcal$ satisfies $\Lcal,\Rcal\se\Wbb$, and $2$-out-of-$3$ with respect to the composition in $\Mcal$,
\end{itemize}
hence in particular we have $\Ssc_{\Ncal}=\Wbb=\Rcal\ci\Lcal$. Localization of $\C$ by $\Ssc_{\Ncal}$ becomes naturally triangulated, as shown in \cite[Theorem~6.20]{NP}.
\end{ex}

The rest of this article is devoted to demonstrate how the examples listed above can be seen as particular cases of the localization given in the previous Section~\ref{Section_Localization}, by showing that the assumption of Theorem~\ref{ThmMultLoc} is indeed satisfied. More precisely, we roughly divide them into the following two cases. 
\begin{itemize}
\item[{\rm (A)}] Localizations obtained in Examples~\ref{ExVerdier}, \ref{ExRump}, \ref{ExHTCP}. 
\item[{\rm (B)}] Localizations obtained in Examples~\ref{ExVerdier}, \ref{ExSerreAbel}, \ref{ExTwo-sidedExact}. 
\end{itemize}

In fact, this division results from particular properties of the thick subcategories. Remark that Example~\ref{ExVerdier} belongs to both cases.
Subsection~\ref{Subsection_LocTri} deals with case {\rm (A)}. We will show that the assumption of Theorem~\ref{ThmMultLoc} is satisfied for any thick subcategory which is \emph{biresolving} (Definition~\ref{Def_BiResol}), and that the localization $\wCEs$ obtained by the theorem corresponds to a triangulated category. Subsection~\ref{Subsection_Percolating} deals with case {\rm (B)}. We will show that the assumption of Corollary~\ref{CorMultLoc} is satisfied for any thick subcategory which is \emph{percolating} (Definition~\ref{Def_Percolating}). With an additional assumption which fits well with percolating subcategories, the localization $\wCEs$ obtained by the corollary can be also made correspond to an exact category. 

\begin{rem}
For Examples~\ref{ExVerdier} and \ref{ExSerreAbel}, it is also easy to check directly that they satisfy the assumption of Corollary~\ref{CorMultLoc}.
On the other hand, this is not always satisfied in case {\rm (A)}. This is the reason why we need the generality of Theorem~\ref{ThmMultLoc}.
\end{rem}

\subsection{Case {\rm (A)}: Triangulated localization by biresolving subcategories}\label{Subsection_LocTri}

We observe that in each of Examples~\ref{ExRump}, \ref{ExHTCP} (and also \ref{ExVerdier}), the corresponding thick subcategory $\Ncal\se\C$ is of the following type. 
\begin{dfn}\label{Def_BiResol} 
A thick subcategory $\Ncal\se\C$ is called \emph{biresolving}, if for any $C\in\C$ there exist an $\sfr$-inflation $C\to N$ and an $\sfr$-deflation $N\ppr\to C$ for some $N,N\ppr\in\Ncal$.
\end{dfn}


This definition obviously covers the above-mentioned examples in mind. Let us summarize here for clarity, together with some other trivial cases.
\begin{ex}
Let $\CEs$ be an extriangulated category, as before.
\begin{enumerate}
\item $\C$ itself is always biresolving in $\CEs$.
\item The thick full subcategory of zero objects in $\C$ is biresolving if and only if $\CEs$ corresponds to a triangulated category.
\item If $\CEs$ corresponds to a triangulated category, then any thick subcategory $\Ncal\se\C$ is biresolving. 
\item Suppose that $\CEs$ corresponds to an exact category, and that $\Ncal\se\C$ is a thick subcategory. Then $\Ncal\se\C$ is biresolving in the sense of Definition~\ref{Def_BiResol} if and only if it is biresolving in the sense of \cite[Definition~7]{R}. (See also Remark~\ref{RemRump}.)
\item Suppose that $\CEs$ satisfies {\rm (WIC)}. If $((\Scal,\Tcal),(\Ucal,\Vcal))$ is a Hovey twin cotorsion pair, then $\Ncal=\Cone(\Vcal,\Scal)\se\C$ is a biresolving thick subcategory as seen in Example~\ref{ExHTCP}.
\item Suppose that $\CEs$ is a Frobenius extriangulated category in the sense of \cite[Definition~7.1]{NP}, and let $\Ncal\se\C$ be the full subcategory of projective-injective objects. Then $\Ncal\se\C$ is a biresolving thick subcategory. (See also \cite[Remark~7.3]{NP} for the relation with the above {\rm (5)}.)
\end{enumerate}
\end{ex}

\begin{rem}
As shown in \cite[Theorem~3.13]{ZZ} (cf. \cite[Corollary~7.4 and Remark~7.5]{NP}), the ideal quotient $(\ovl{\C},\ovl{\E},\ovl{\sfr})$ in {\rm (6)} corresponds to a triangulated category. By Remark~\ref{Rem_IdealQuotLoc}, so does $\wCEs$.
\end{rem}
%
%
%

In the rest of this subsection, let 
$\Ncal\se\C$ denote a biresolving thick subcategory. The aim of this subsection is to show Proposition~\ref{PropSatisfy} and Corollary~\ref{CorLocTri}.
Let us start with some lemmas.
\begin{lem}\label{LemAllInf}
Let $f\in\C(A,B)$ be any morphism.
\begin{enumerate}
\item By taking an $\sfr$-inflation $i\in\C(A,N)$ to $N\in\Ncal$, morphism $f$ can be written as a composition of
an $\sfr$-inflation $\left[\bsm f\\ i\esm\right]\in\C(A,B\oplus N)$ and a split epimorphism $r=[1\ 0]\in\C(B\oplus N,B)$. Remark that we have $r\in\Rcal$.
\item By taking an $\sfr$-deflation $j\in\C(N\ppr,B)$ from $N\ppr\in\Ncal$, morphism $f$ can be written as a composition of
an $\sfr$-deflation $[f\ j]\in\C(A\oplus N\ppr,B)$ and a split monomorphism $l=\left[\bsm1\\0\esm\right]\in\C(A,A\oplus N\ppr)$. Remark that we have $l\in\Lcal$.
\end{enumerate}
\end{lem}
\begin{proof}
This is obvious by \cite[Corollary 3.16]{NP}.
\end{proof}

\begin{lem}\label{LemRL}
$\Ssc_{\Ncal}=\Rcal\ci\Lcal$ holds.
\end{lem}
\begin{proof}
It suffices to show $\Lcal\ci\Rcal\se\Rcal\ci\Lcal$. 
For any pair of $\sfr$-conflations $N_1\ov{m}{\lra}X\ov{r}{\lra}Y$ and $Y\ov{l}{\lra}Z\ov{e}{\lra}N_2$ with $N_1,N_2\in\Ncal$, let us show that $f=l\ci r$ belongs to $\Rcal\ci\Lcal$. 
Take an $\sfr$-inflation $i\in\C(X,N)$ to some $N\in\Ncal$. As in Lemma~\ref{LemAllInf}, we obtain $\sfr$-conflations $X\ov{\left[\bsm f\\ i\esm\right]}{\lra}Z\oplus N\ov{g}{\lra}Z\ppr$ and $X\ov{\left[\bsm r\\ i\esm\right]}{\lra}Y\oplus N\ov{h}{\lra}Y\ppr$, and factorize $f$ as below.
\[
\xy
(-16,14)*+{N_1}="0";
(0,14)*+{N}="2";
%
(-16,0)*+{X}="10";
(0,0)*+{Z\oplus N}="12";
(16,0)*+{Z\ppr}="14";
(-16,-14)*+{Y}="20";
(0,-14)*+{Z}="22";
(16,-14)*+{N_2}="24";
%
%
{\ar_{m} "0";"10"};
{\ar^{\left[\bsm0\\1\esm\right]} "2";"12"};
%
{\ar^(0.4){\left[\bsm f\\ i\esm\right]} "10";"12"};
{\ar_(0.56){g} "12";"14"};
{\ar_{r} "10";"20"};
{\ar^{[1\ 0]} "12";"22"};
%
{\ar_{l} "20";"22"};
{\ar_{e} "22";"24"};
%
{\ar@{}|\circlearrowright "10";"22"};
\endxy
\]
Here, the rows and columns are $\sfr$-conflations.
By \cite[Lemma~3.14]{NP}, there exist morphisms $d,d\ppr$ which makes the following diagram commutative,
\[
\xy
(-16,7)*+{X}="0";
(0,7)*+{Y\oplus N}="2";
(16,7)*+{Y\ppr}="4";
(-16,-7)*+{X}="10";
(0,-7)*+{Z\oplus N}="12";
(16,-7)*+{Z\ppr}="14";
(0,-21)*+{N_2}="22";
(16,-21)*+{N_2}="24";
{\ar^(0.4){\left[\bsm r\\ i\esm\right]} "0";"2"};
{\ar^(0.54){h} "2";"4"};
{\ar@{=} "0";"10"};
{\ar_{l\oplus \id} "2";"12"};
{\ar^{d} "4";"14"};
{\ar_(0.4){\left[\bsm f\\ i\esm\right]} "10";"12"};
{\ar_(0.54){g} "12";"14"};
{\ar_{[e\ 0]} "12";"22"};
{\ar^{d\ppr} "14";"24"};
{\ar@{=} "22";"24"};
{\ar@{}|\circlearrowright "0";"12"};
{\ar@{}|\circlearrowright "2";"14"};
{\ar@{}|\circlearrowright "12";"24"};
\endxy
\]
such that $Y\ppr\ov{d}{\lra}Z\ppr\ov{d\ppr}{\lra}N_2$ is an $\sfr$-conflation.
By \cite[Proposition~3.17]{NP}, there exists a morphism $i\ppr$ which makes the following diagram commutative,
\[
\xy
(-16,7)*+{N_1}="2";
(0,7)*+{X}="4";
(16,7)*+{Y}="6";
(-16,-7)*+{N}="12";
(0,-7)*+{Y\oplus N}="14";
(16,-7)*+{Y}="16";
(-16,-21)*+{Y\ppr}="22";
(0,-21)*+{Y\ppr}="24";
{\ar^{m} "2";"4"};
{\ar^{r} "4";"6"};
{\ar_{i\ppr} "2";"12"};
{\ar^{\left[\bsm r\\ i\esm\right]} "4";"14"};
{\ar@{=} "6";"16"};
{\ar_(0.46){\left[\bsm 0\\ 1\esm\right]} "12";"14"};
{\ar_(0.54){[1\ 0]} "14";"16"};
{\ar_{h\ci\left[\bsm 0\\ 1\esm\right]} "12";"22"};
{\ar^{h} "14";"24"};
{\ar@{=} "22";"24"};
{\ar@{}|\circlearrowright "2";"14"};
{\ar@{}|\circlearrowright "4";"16"};
{\ar@{}|\circlearrowright "12";"24"};
\endxy
\]
such that $N_1\ov{i\ppr}{\lra}N\ov{h\ci\left[\bsm 0\\ 1\esm\right]}{\lra}Y\ppr$ is an $\sfr$-conflation.
We remark that $i\ppr=i\ci m$ also holds by the commutativity.
Since $\Ncal\se\C$ is thick, it follows $Y\ppr,Z\ppr\in\Ncal$. This shows $\left[\bsm f\\ i\esm\right]\in\Lcal$, and thus $f=[1\ 0]\ci \left[\bsm f\\ i\esm\right]\in\Rcal\ci\Lcal$ as desired.
\end{proof}

\begin{lem}\label{LemSInfDef}
Let $f\in\C(A,B)$ be any morphism.
\begin{enumerate}
\item $f\in\Lcal$ holds if and only if $f$ is an $\sfr$-inflation satisfying $f\in\Ssc_{\Ncal}$.
\item $f\in\Rcal$ holds if and only if $f$ is an $\sfr$-deflation satisfying $f\in\Ssc_{\Ncal}$.
\end{enumerate}
\end{lem}
\begin{proof}
By duality, it is enough to show {\rm (1)}. As the converse is trivial, it suffices to show that any $\sfr$-inflation $f\in\Ssc_{\Ncal}=\Rcal\ci\Lcal$ should belong to $\Lcal$. However, this immediately follows from the dual of \cite[Proposition~3.17]{NP} since $\Ncal\se\C$ is thick.
\end{proof}

\begin{lem}\label{LemAllInf2}
For any $f\in\C(A,B)$, the following are equivalent.
\begin{enumerate}
\item $f\in\Ssc_{\Ncal}$. 
\item $\left[\bsm f\\ i\esm\right]\in\Lcal$ holds for a decomposition as in Lemma~\ref{LemAllInf} {\rm (1)}.
\item $[f\ j]\in\Rcal$ holds for a decomposition as in Lemma~\ref{LemAllInf} {\rm (2)}.
\end{enumerate}
\end{lem}
\begin{proof}
By duality, it is enough to show {\rm (1)}\,$\EQ$\,{\rm (2)}. If $f\in\Ssc_{\Ncal}$, then $\left[\bsm f\\ i\esm\right]$ belongs to $\Ssc_{\Ncal}$ by Lemma~\ref{LemSplitN} {\rm (1)}, hence to $\Lcal$ by Lemma~\ref{LemSInfDef}. This shows {\rm (1)}\,$\tc$\,{\rm (2)}. Since $r\in\Rcal$ is always satisfied, the converse is trivial.
\end{proof}

\begin{lem}\label{LemAddedForBiresol}
$p(\Ssc_{\Ncal})=\ovl{\Ssc_{\Ncal}}$ holds.
\end{lem}
\begin{proof}
By Lemma~\ref{LemSplitN}, it suffices to show that condition {\rm (iv)} in Lemma~\ref{LemSplitN} {\rm (3)} is satisfied. Let $f\in\C(A,B)$ be any split monomorphism such that $\ovl{f}$ is an isomorphism in $\ovl{\C}$. As in Lemma~\ref{LemAllInf} {\rm (1)}, we may take an $\sfr$-conflation $A\ov{\left[\bsm f\\ i\esm\right]}{\lra}B\oplus N\to C$ with $N\in\Ncal$. Since $f$ is a split monomorphism, so is $\left[\bsm f\\ i\esm\right]$. This means that $\left[\bsm f\\ i\esm\right]$ induces an isomorphism $A\oplus C\cong B\oplus N$ in $\C$. Since $\ovl{f}$ is an isomorphism in $\ovl{\C}$, it follows $C\cong0$ in $\ovl{\C}$. Then there is a split monomorphism $s\colon C\to N\ppr$ with $N\in\Ncal$. As in Lemma~\ref{LemAllInf} {\rm (1)}, we may take an $\sfr$-conflation $C\ov{\left[\bsm s\\ i\ppr\esm\right]}{\lra}N\ppr\oplus N\pprr\to D$ with $N\pprr\in\Ncal$. In the same manner as before, this induces an isomorphism $C\oplus D\cong N\ppr\oplus N\pprr$ in $\C$. Then we have $C\in\Ncal$ since $\Ncal\se\C$ is closed by direct summands. Thus $\left[\bsm f\\ i\esm\right]\in\Lcal$ holds, hence $f\in\Ssc_{\Ncal}$ by Lemma~\ref{LemAllInf2}.
\end{proof}


\begin{lem}\label{LemForN2}
For any $l\in\Lcal$, morphism $\ovl{l}$ is monomorphic in $\ovl{\C}$.
Dually $\ovl{r}$ is epimorphic in $\ovl{\C}$, for any $r\in\Rcal$.
\end{lem}
\begin{proof}
Suppose that $f\in\C(X,Y)$ and $l\in\Lcal(Y,Y\ppr)$ satisfy $\ovl{l}\ci\ovl{f}=0$ in $\ovl{\C}$. Then there exists a commutative square
\[
\xy
(-6,6)*+{X}="0";
(6,6)*+{N}="2";
(-6,-6)*+{Y}="4";
(6,-6)*+{Y\ppr}="6";
{\ar^{i} "0";"2"};
{\ar_{f} "0";"4"};
{\ar^{j} "2";"6"};
{\ar_{l} "4";"6"};
{\ar@{}|\circlearrowright "0";"6"};
\endxy
\]
in $\C$ for some $N\in\Ncal$. Since $\Ncal\se\C$ is biresolving, there is an $\sfr$-deflation $N\ppr\to Y\ppr$ for some $N\ppr\in\Ncal$. Replacing $N$ by $N\oplus N\ppr$, we may assume that $j$ is an $\sfr$-deflation, from the beginning. Then there is an $\sfr$-conflation
\[ W\to Y\oplus N\ov{[l\ j]}{\lra}Y\ppr \]
for some $W$. By the commutativity of the above square, we see that $f$ factors through $W$. By the dual of Lemma~\ref{LemSplitN} {\rm (1)}, we have $[l\ j]\in\Ssc_{\Ncal}$. Thus Lemma~\ref{LemSInfDef} {\rm (2)} shows $W\in\Ncal$, 
hence $\ovl{f}=0$ follows. The latter part can be shown dually.
\end{proof}

\begin{lem}\label{LemForN3}
For any $X\ov{f}{\lra}Y\ov{l}{\lla}Y\ppr$ with $l\in\Lcal$, there exists a commutative square
\begin{equation}\label{Square_F6}
\xy
(-8,6)*+{X\ppr}="0";
(8,6)*+{Y\ppr}="2";
(-8,-6)*+{X\oplus N}="4";
(8,-6)*+{Y}="6";
{\ar^{f\ppr} "0";"2"};
{\ar_{l\ppr} "0";"4"};
{\ar^{l} "2";"6"};
{\ar_(0.6){[f\ j]} "4";"6"};
{\ar@{}|\circlearrowright "0";"6"};
\endxy
\end{equation}
in $\C$ such that $N\in\Ncal$ and $l\ppr\in\Lcal$. Thus we obtain a commutative diagram 
\[
\xy
(-6,6)*+{X\ppr}="0";
(6,6)*+{Y\ppr}="2";
(-6,-6)*+{X}="4";
(6,-6)*+{Y}="6";
{\ar^{\ovl{f}\ppr} "0";"2"};
{\ar_{\ovl{s}} "0";"4"};
{\ar^{\ovl{l}} "2";"6"};
{\ar_{\ovl{f}} "4";"6"};
{\ar@{}|\circlearrowright "0";"6"};
\endxy
\]
in $\ovl{\C}$, where $s=[1\ 0]\ci l\ppr\in\Ssc_{\Ncal}$.
%
%
\end{lem}
\begin{proof}
%
Suppose that $X\ov{f}{\lra}Y\ov{l}{\lla}Y\ppr$ satisfies $l\in\Lcal$. By the assumption, there is an $\sfr$-deflation $j\in\C(N,Y)$ from some $N\in\Ncal$. Then $[f\ j]\co X\oplus N\to Y$ is also an $\sfr$-deflation, and a commutative square $(\ref{Square_F6})$ can be obtained by using {\rm (ET4)$^\op$}.
\end{proof}

\begin{prop}\label{PropSatisfy}
Let $\Ncal\se\C$ be a biresolving thick subcategory. 
Then $\Ssc_{\Ncal}$ satisfies {\rm (MR1)},\,$\ldots$\,,\,{\rm (MR4)}.
Moreover, $\ovl{\Ssc_{\Ncal}}$ can be described as
\begin{equation}\label{Eq_MonoEpi}
\ovl{\Ssc_{\Ncal}}=\{\ovl{x}\in\Mor(\ovl{\C})\mid \ovl{x}\ \text{is both monomorphic and epimorphic} \}. 
\end{equation}
\end{prop}
\begin{proof}
By Lemma~\ref{LemAddedForBiresol}, we have $p(\Ssc_{\Ncal})=\ovl{\Ssc_{\Ncal}}$. Also, {\rm (M3)} is already shown in Lemma~\ref{LemM3}, and {\rm (MR4)} is immediate from Lemma~\ref{LemAllInf}. Thus by Claim~\ref{ClaimMultLoc}, it suffices to confirm conditions {\rm (M1),(MR2)} and show $(\ref{Eq_MonoEpi})$.

{\rm (M1)}
$\Ssc_{\Ncal}\se\Mcal$ is closed by composition, by its definition. Let
\begin{equation}\label{Diag_Comm_*}
\xy
(-7,6)*+{A}="0";
(-7,-6)*+{B}="2";
(2,2)*+{}="3";
(7,-6)*+{C}="4";
{\ar_{f} "0";"2"};
{\ar_{g} "2";"4"};
{\ar^{h} "0";"4"};
{\ar@{}|\circlearrowright "2";"3"};
\endxy
\end{equation}
be any commutative diagram in $\C$ satisfying $h\in\Ssc_{\Ncal}$. Let us show that $f\in\Ssc_{\Ncal}$ holds if and only if $g\in\Ssc_{\Ncal}$.

As the converse can be shown in a dual manner, it is enough to show that $g\in\Ssc_{\Ncal}$ implies $f\in\Ssc_{\Ncal}$. 
Since $g$ is a composition of morphisms in $\Lcal$ and $\Rcal$, it suffices to show in the case of $g\in\Lcal$ and of $g\in\Rcal$.

Take an $\sfr$-inflation $i\in\C(A,N)$ to some $N\in\Ncal$. Then $f\ppr=\left[\bsm f\\ i\esm\right],h\ppr=\left[\bsm h\\ i\esm\right]$ are $\sfr$-inflations, hence we obtain $\sfr$-triangles
\[
A\ov{\left[\bsm f\\ i\esm\right]}{\lra}B\oplus N\ov{y}{\lra}X\ov{\del}{\dra}\quad \text{and}\quad
A\ov{\left[\bsm h\\ i\esm\right]}{\lra}C\oplus N\ov{z}{\lra}N\ppr\ov{\rho}{\dra},
\]
in which $N\ppr\in\Ncal$ holds by Lemma~\ref{LemAllInf2}. It suffices to show $X\in\Ncal$, again by the same lemma.
Remark that $(\ref{Diag_Comm_*})$ yields another commutative diagram
\[
\xy
(-10,7)*+{A}="0";
(-10,-8)*+{B\oplus N}="2";
(2,4)*+{}="3";
(10,-8)*+{C\oplus N}="4";
{\ar_{f\ppr} "0";"2"};
{\ar_{g\ppr} "2";"4"};
{\ar^{h\ppr} "0";"4"};
{\ar@{}|\circlearrowright "2";"3"};
\endxy
\]
in $\C$, where we put $g\ppr=g\oplus\id_N$.
If $g\in\Lcal$, then $g\ppr\in\Lcal$ holds. Thus by {\rm (ET4)} we obtain a commutative diagram
\[
\xy
(-24,14)*+{A}="0";
(-8,14)*+{B\oplus N}="2";
(8,14)*+{X}="4";
(-24,0)*+{A}="10";
(-8,0)*+{C\oplus N}="12";
(8,0)*+{N\ppr}="14";
(-8,-14)*+{N\pprr}="22";
(8,-14)*+{N\pprr}="24";
{\ar^(0.44){f\ppr} "0";"2"};
{\ar^(0.56){y} "2";"4"};
{\ar@{=} "0";"10"};
{\ar^{g\ppr} "2";"12"};
{\ar^{} "4";"14"};
{\ar_(0.44){h\ppr} "10";"12"};
{\ar_(0.56){z} "12";"14"};
{\ar_{} "12";"22"};
{\ar^{} "14";"24"};
{\ar@{=} "22";"24"};
{\ar@{}|\circlearrowright "0";"12"};
{\ar@{}|\circlearrowright "2";"14"};
{\ar@{}|\circlearrowright "12";"24"};
\endxy
\]
with $N\pprr\in\Ncal$, whose rows and columns are $\sfr$-conflations. Since $\Ncal\se\C$ is thick, it follows $X\in\Ncal$.

On the other hand if $g\in\Rcal$, then $g\ppr\in\Rcal$ holds. Thus by the dual of \cite[Proposition~3.17]{NP} we obtain a diagram
\[
\xy
(-8,14)*+{N\pprr}="2";
(8,14)*+{N\pprr}="4";
(-24,0)*+{A}="10";
(-8,0)*+{B\oplus N}="12";
(8,0)*+{X}="14";
(-24,-14)*+{A}="20";
(-8,-14)*+{C\oplus N}="22";
(8,-14)*+{N\ppr}="24";
{\ar@{=} "2";"4"};
{\ar^{} "2";"12"};
{\ar^{} "4";"14"};
{\ar^(0.44){f\ppr} "10";"12"};
{\ar_(0.56){y} "12";"14"};
{\ar@{=} "10";"20"};
{\ar_{g\ppr} "12";"22"};
{\ar^{} "14";"24"};
{\ar_(0.44){h\ppr} "20";"22"};
{\ar_(0.56){z} "22";"24"};
{\ar@{}|\circlearrowright "2";"14"};
{\ar@{}|\circlearrowright "10";"22"};
{\ar@{}|\circlearrowright "12";"24"};
\endxy
\]
with $N\pprr\in\Ncal$, whose rows and columns are $\sfr$-conflations. Since $\Ncal\se\C$ is thick, it follows $X\in\Ncal$ also in this case.

{\rm (MR2)} 
By Lemmas~\ref{LemThickFirstProperties} and \ref{LemForN2}, any $\ovl{s}\in\ovl{\Ssc_{\Ncal}}$ is monomorphic and epimorphic in $\ovl{\C}$.
Remaining conditions also follow from Lemma~\ref{LemThickFirstProperties}, Lemma~\ref{LemForN3} and its dual.


As for $(\ref{Eq_MonoEpi})$, it remains to show that
\[ \ovl{\Ssc_{\Ncal}}\supseteq\{\ovl{x}\in\Mor(\ovl{\C})\mid \ovl{x}\ \text{is both monomorphic and epimorphic}\} \]
holds. Let $x\in\C(A,B)$ be any morphism such that $\ovl{x}$ is monomorphic and epimorphic in $\ovl{\C}$. By Lemma~\ref{LemAllInf}, we may assume that $x$ is an $\sfr$-inflation, from the beginning. Since $\Ncal\se\C$ is biresolving, there is an $\sfr$-deflation $N\ov{g}{\lra}B$ and an $\sfr$-inflation $B\ov{i}{\lra}N\ppr$ for some $N,N\ppr\in\Ncal$. There exist $\sfr$-triangles $A\ov{x}{\lra}B\ov{y}{\lra}C\ov{\del}{\dra}$ and $D\ov{f}{\lra}N\ov{g}{\lra}B\ov{\rho}{\dra}$. By {\rm (ET4)$^\op$}, we obtain a diagram made of $\sfr$-triangles as below,
\[
\xy
(-12,12)*+{D}="0";
(0,12)*+{D}="2";
(-12,0)*+{E}="10";
(0,0)*+{N}="12";
(12,0)*+{C}="14";
(24,0)*+{}="16";
(-12,-12)*+{A}="20";
(0,-12)*+{B}="22";
(12,-12)*+{C}="24";
(24,-12)*+{}="26";
(-12,-24)*+{}="30";
(0,-24)*+{}="32";
{\ar@{=} "0";"2"};
{\ar^{} "0";"10"};
{\ar^{f} "2";"12"};
{\ar_{} "10";"12"};
{\ar_{} "12";"14"};
{\ar@{-->}^{\thh} "14";"16"};
{\ar_{e} "10";"20"};
{\ar^{g} "12";"22"};
{\ar@{=} "14";"24"};
{\ar_{x} "20";"22"};
{\ar_{y} "22";"24"};
{\ar@{-->}_{\del} "24";"26"};
{\ar@{-->} "20";"30"};
{\ar@{-->}^{\rho} "22";"32"};
{\ar@{}|\circlearrowright "0";"12"};
{\ar@{}|\circlearrowright "10";"22"};
{\ar@{}|\circlearrowright "12";"24"};
\endxy
\]
for some $\Ebb$-extension $\thh$.
Since $\ovl{x}$ is monomorphic and epimorphic in $\ovl{\C}$, it follows $\ovl{e}=0$ and $\ovl{y}=0$. Thus we have $\ovl{\del}=\ovl{e\sas\thh}=0$, which means that there exists some $s\in\Ssc_{\Ncal}(A,A\ppr)$ such that $s\sas\del=0$. Since $\Ssc_{\Ncal}$ satisfies {\rm (M3)}, there exists $t\in\Ssc_{\Ncal}(B,A\ppr\oplus C)$ which gives the following morphism of $\sfr$-triangles.
\[
\xy
(-15,6)*+{A}="0";
(0,6)*+{B}="2";
(15,6)*+{C}="4";
(28,6)*+{}="6";
(-15,-6)*+{A\ppr}="10";
(0,-6)*+{A\ppr\oplus C}="12";
(15,-6)*+{C}="14";
(28,-6)*+{}="16";
{\ar^{x} "0";"2"};
{\ar^{y} "2";"4"};
{\ar@{-->}^{\del} "4";"6"};
{\ar_{s} "0";"10"};
{\ar_{t} "2";"12"};
{\ar@{=} "4";"14"};
{\ar_(0.4){\left[\bsm1\\0\esm\right]} "10";"12"};
{\ar_(0.6){[0\ 1]} "12";"14"};
{\ar@{-->}_{0} "14";"16"};
{\ar@{}|\circlearrowright "0";"12"};
{\ar@{}|\circlearrowright "2";"14"};
\endxy
\]
Then the commutativity of the right square means that $t$ is of the form $t=\left[\bsm b\\ y\esm\right]$ for some $b\in\C(B,A\ppr)$.
Since $\left[\bsm i\\ b \esm\right]\in\C(B,N\ppr\oplus A\ppr)$ is an $\sfr$-inflation, we have an $\sfr$-conflation $B\ov{\left[\bsm i\\ b\esm\right]}{\lra}N\ppr\oplus A\ppr\to Z$ associated to it. This in turn gives rise to an $\sfr$-conflation
\[ B\ov{\left[\bsm i\\ b\\0\esm\right]}{\lra}N\ppr\oplus A\ppr\oplus C\to Z\oplus C. \]
Since $\ovl{t}=\left[\bsm\ovl{b}\\0\esm\right]\in\ovl{\Ssc_{\Ncal}}$, we also have $\ovl{\left[\bsm i\\ b\\ 0\esm\right]}\in\ovl{\Ssc_{\Ncal}}$. By Lemma~\ref{LemAllInf2} this means  $\left[\bsm i\\ b\\ 0\esm\right]\in\Lcal$, namely $Z\oplus C\in\Ncal$. Since $\Ncal\se\C$ is closed by direct summands, we obtain $C\in\Ncal$. This shows $x\in\Lcal\se\Ssc_{\Ncal}$ as desired.
\end{proof}

\begin{cor}\label{CorLocTri}
Let $\Ncal\se\C$ be a biresolving thick subcategory. 
Then the localization of $\C$ by $\Ssc_{\Ncal}$ corresponds to a triangulated category.
\end{cor}
\begin{proof}
By Theorem~\ref{ThmMultLoc} and Proposition~\ref{PropSatisfy}, we obtain an extriangulated category $\wCEs$. Since any morphism in $\wC$ is both an $\ws$-inflation and an $\ws$-deflation by Lemmas~\ref{LemComposeInf} and \ref{LemAllInf}, this becomes triangulated.
\end{proof}

\subsection{Case {\rm (B)}: Localization by percolating subcategories}\label{Subsection_Percolating}
In this subsection, we show that the localizations in Examples~\ref{ExSerreAbel} and \ref{ExTwo-sidedExact} can be covered by our construction in Section~\ref{Section_Localization}. In fact, we may perform it in slightly a bit broader situation\footnote{The authors wish to thank Mikhail Gorsky, whose suggestive comment to the previous version of this manuscript led them to this improvement.}, so that it also contains Example~\ref{ExVerdier}. 

\begin{dfn}\label{Def_Percolating}
A thick subcategory $\Ncal\se\C$ is called a \emph{two-sided admissibly percolating} subcategory, or simply a \emph{percolating} subcategory in this article, if the following conditions are satisfied.
\begin{itemize}
\item[{\rm (P1)}] For any morphism $f\co X\to N$ in $\C$ with $N\in\Ncal$, there is a factorization of $f$ as $X\ov{g}{\to}N\ppr\ov{h}{\to}N$ with an $\sfr$-deflation $g$ and an $\sfr$-inflation $h$, and $N\ppr\in\Ncal$.
\item[{\rm (P1')}] Dually, for any morphism $f\co N\to Y$ in $\C$ with $N\in\Ncal$, there is a factorization of $f$ as $N\ov{g}{\to}N\ppr\ov{h}{\to}Y$ with an $\sfr$-deflation $g$ and an $\sfr$-inflation $h$, and $N\ppr\in\Ncal$.
\end{itemize}
\end{dfn}

\begin{lem}\label{LemPercdefinf}
Let $\Ncal\se\C$ be a thick subcategory.
The following are equivalent.
\begin{enumerate}
\item $\Ncal$ is percolating.
\item If a morphism $f\co X\to Y$ in $\C$ factors through some object in $\Ncal$, then there exists a factorization $X\ov{d}{\to}N\ov{i}{\to}Y$ of $f$ with an $\sfr$-deflation $d$, an $\sfr$-inflation $i$ and $N\in\Ncal$.
\end{enumerate}
\end{lem}
\begin{proof}
As the converse is trivial, let us show that {\rm (1)} implies {\rm (2)}.
Suppose that $f=h\ci g$ holds for morphisms $g\co X\to N_1, h\co N_1\to Y$ with $N_1\in\Ncal$. Since $\Ncal$ is percolating, there is a factorization $X\ov{d}{\to}N_2\ov{i}{\to}N_1$ of $g$. Here $d$ is an $\sfr$-deflation, $i$ is an $\sfr$-inflation and $N_2$ belongs to $\Ncal$. Then there is a factorization $N_2\ov{d\ppr}{\to}N_3\ov{i\ppr}{\to}Y$ of $h\ci i$ with an $\sfr$-deflation $d\ppr$, an $\sfr$-inflation $i\ppr$ and $N_3\in\Ncal$ because $\Ncal$ is percolating. Thus we obtain a factorization $X\ov{d\ppr\ci d}{\to}N_3\ov{i\ppr}{\to}Y$ of $f$ as desired. 
\end{proof}

In addition, we also consider the following condition so that Corollary~\ref{CorMultLoc} can be applied to $\Ssc_{\Ncal}$. 
\begin{cond}\label{ConditionPerc}
Let $\Ncal$ be a percolating thick subcategory in $\CEs$. Consider the following conditions.
\begin{enumerate}
\item[{\rm (P2)}] If $f\in\C(A,B)$ is a split monomorphism such that $\ovl{f}$ is an isomorphism in $\ovl{\C}$, then there exist $N\in\Ncal$ and $j\in\C(N,B)$ such that $[f\ j]\co A\oplus N\to B$ is an isomorphism in $\C$.
\item[{\rm (P3)}] $\Ker\big(\C(X,A)\ov{l\ci-}{\lra}\C(X,B)\big)\se[\Ncal](X,A)$ holds for any $X\in\C$ and any $l\in\Lcal(A,B)$. 
Dually, $\Ker\big(\C(C,X)\ov{-\ci r}{\lra}\C(B,X)\big)\se[\Ncal](C,X)$ holds for any $X\in\C$ and any $r\in\Rcal(B,C)$.
\end{enumerate}
\end{cond}

Thick subcategories in Examples~\ref{ExVerdier}, \ref{ExSerreAbel}, \ref{ExTwo-sidedExact} are percolating and moreover satisfy Condition~\ref{ConditionPerc}. More in detail, we have the following Remarks~\ref{RemP2} and \ref{RemP3}.

\begin{rem}\label{RemP2}
As for {\rm (P2)}, the following holds.
\begin{enumerate}
\item {\rm (P2)} is self-dual. Namely, {\rm (P2)} holds if and only if for any split epimorphism $e\in\C(B,A)$ such that $\ovl{e}$ is an isomorphism in $\ovl{\C}$, there exist $N\in\Ncal$ and $i\in\C(B,N)$ such that $\left[\bsm e\\ i\esm\right]\co B\to A\oplus N$ is an isomorphism in $\C$.

\item Suppose that $\CEs$ satisfies {\rm (WIC)} or more generally that any split monomorphism has a cokernel. Then {\rm (P2)} is always satisfied. In fact, cokernel $N$ of a split monomorphism $f\co A\to B$ belongs to $\Ncal$ whenever $\ovl{f}$ is an isomorphism, hence gives a decomposition into a direct sum $B\cong A\oplus N$.

We also remark that if any split monomorphism $f$ in $\C$ has a cokernel, then any extension-closed subcategory $\D\se\C$ closed by direct summands also possesses this property.

\item Suppose that $\CEs$ corresponds to an exact category. Then {\rm (P2)} is always satisfied by any percolating subcategory. 
Indeed if $f\in\C(A,B)$ is a split monomorphism such that $\ovl{f}$ is an isomorphism in $\ovl{\C}$, then there exist $e\in\C(B,A)$, $N\in\Ncal$, $i\in\C(B,N)$ and $j\in\C(N,B)$ which satisfy $e\ci f=\id_A$ and $f\ci e+j\ci i=\id_B$. We may assume that $i$ is an $\sfr$-deflation and $j$ is an $\sfr$-inflation by Lemma~\ref{LemPercdefinf}. Then $i\ci f=0$, $e\ci j=0$ and $i\ci j=\id$ follows since $i$ is epimorphic and $j$ is monomorphic, thus $[f\ j]\co A\oplus N\to B$ and $\left[\bsm e\\i\esm\right]\co B\to A\oplus N$ are inverse to each other.
\end{enumerate}
\end{rem}

\begin{rem}\label{RemP3}
As for {\rm (P3)}, the following holds.
\begin{enumerate}
\item Suppose that $\CEs$ corresponds to an exact category. Then {\rm (P3)} is trivially satisfied, since any $\sfr$-inflation is monomorphic and any $\sfr$-deflation is epimorphic.
\item Suppose that $\CEs$ corresponds to a triangulated category. Then {\rm (P3)} is always satisfied. This is because $N\to A\ov{f}{\lra} B$ is an $\sfr$-conflation if and only if $A\ov{f}{\lra} B\to N[1]$ is an $\sfr$-conflation.
\end{enumerate}
We also remark that in these particular cases {\rm (1)} and {\rm (2)}, the following stronger version {\rm (P3$^+$)} is satisfied.
\begin{itemize}
\item[{\rm (P3$^+$)}] For any $\sfr$-conflation $N\to A\ov{r}{\lra}B$ with $N\in\Ncal$, there exist $N\ppr\in\Ncal$ and 
$g\in\C(B,N\ppr)$ which gives a weak cokernel of $r$. Dually for $\sfr$-conflations $A\ov{l}{\lra}B\to N$ with $N\in\Ncal$.
\end{itemize}
\end{rem}

In summary, we have the following.
\begin{ex}\label{ExPerc}
Let $\CEs$ be an extriangulated category, as before.
\begin{enumerate}
\item $\C$ itself is percolating in $\CEs$ if and only if any morphism in $\C$ is $\sfr$-admissible. This always satisfies {\rm (P3)}. Moreover, if any split monomorphism in $\C$ has a cokernel, then it also satisfies {\rm (P2)}.
\item The thick full subcategory of zero objects in $\C$ is always percolating and satisfies Condition~\ref{ConditionPerc}.
\item If $\CEs$ corresponds to a triangulated category, then any thick subcategory $\Ncal\se\C$ is percolating and satisfies Condition~\ref{ConditionPerc}. 
\item Suppose that $\CEs$ corresponds to an exact category. Then any percolating thick subcategory $\Ncal\se\C$ becomes a Serre subcategory, in the sense that $B\in\Ncal$ holds if and only if $A,C\in\Ncal$, for each $\sfr$-conflation $A\to B\to C$. Also, it always satisfies Condition~\ref{ConditionPerc}. Thus Condition~\ref{ConditionPerc} does not pose any additional requirement on \cite[Definition 2.4]{HKR}. 
\item If $\C$ is abelian, then conversely any Serre subcategory $\Ncal\se\C$ is a percolating thick subcategory.
\end{enumerate}
\end{ex}

In the rest, let $\Ncal\se\C$ denote a percolating thick subcategory satisfying Condition~\ref{ConditionPerc}. 
Our aim is to show Proposition~\ref{PropPerc}, which asserts that the assumption of Corollary~\ref{CorMultLoc} is fulfilled by $\Ssc_{\Ncal}$ associated to it. 
\begin{lem}\label{LemExPerc_pS}
$p(\Ssc_{\Ncal})=\ovl{\Ssc_{\Ncal}}$ holds.
\end{lem}
\begin{proof}
By Lemma~\ref{LemSplitN} {\rm (3)}, this is immediate from {\rm (P2)}.
\end{proof}

\begin{lem}\label{Lem_ii}
The following holds.
\begin{enumerate}
\item Let $f\co N\to A$ be any morphism in $\C$ with $N\in\Ncal$. If there is an $\sfr$-inflation $x\co A\to B$ such that $x\ci f$ is an $\sfr$-inflation, then $f$ is an $\sfr$-inflation.
\item Let $f\co C\to N$ be any morphism in $\C$ with $N\in\Ncal$. If there exists an $\sfr$-deflation $y\co B\to C$ such that $f\ci y$ is an $\sfr$-deflation, then $f$ is an $\sfr$-deflation. 
\end{enumerate}
\end{lem}
\begin{proof}
Since {\rm (2)} can be shown dually, it is enough to show {\rm (1)}.
Let $f\co N\to A$ be a morphism with $N\in\Ncal$, and suppose that there is an $\sfr$-inflation $x\co A\to B$ such that $x\ci f$ is an $\sfr$-inflation.
Since $\Ncal$ is percolating, there exist an $\sfr$-conflation $N\pprr\ov{m}{\lra}N\ov{q}{\lra}N\ppr$ and an $\sfr$-inflation $i\in\C(N\ppr,A)$ such that $f=i\ci q$. It suffices to show that $q$ is an $\sfr$-inflation.

Since $(x\ci f)\ci m=0$ becomes an $\sfr$-inflation, we obtain an $\sfr$-conflation $N\pprr\ov{0}{\lra}B\ov{b}{\lra}B\ppr$ for some $B\ppr\in\C$. By {\rm (C1')} for $\CEs$, we see that there exists $e\in\C(B\ppr,B)$ such that $e\ci b=\id_B$, in particular $b$ is a split monomorphism. Then $\id_{B\ppr}-b\ci e\co B\ppr\to B\ppr$ satisfies $(\id_{B\ppr}-b\ci e)\ci b=0$, hence $\id_{B\ppr}=\ovl{b}\ci\ovl{e}$ holds in $\ovl{\C}$ by {\rm (P3)}. In particular $\ovl{b}$ is an isomorphism in $\ovl{\C}$. Thus by {\rm (P2)}, we have an isomorphism $[b\ j]\co B\oplus N_B\ov{\cong}{\lra}B\ppr$ for some $N_B\in\Ncal$. Thus we have a split $\sfr$-conflation $N_B\ov{j}{\lra}B\ppr\ov{e\ppr}{\lra}B$ such that $e\ppr\ci b=\id_B$. By {\rm (ET4)$^\op$}, we obtain a commutative diagram in $\C$
\[
\xy
(-18,12)*+{N\pprr}="0";
(-6,12)*+{0}="2";
(6,12)*+{N_B}="4";
(-18,0)*+{N\pprr}="10";
(-6,0)*+{B}="12";
(6,0)*+{B\ppr}="14";
(-6,-12)*+{B}="22";
(6,-12)*+{B}="24";
{\ar^{} "0";"2"};
{\ar^{} "2";"4"};
{\ar@{=} "0";"10"};
{\ar^{} "2";"12"};
{\ar^{j} "4";"14"};
{\ar_{0} "10";"12"};
{\ar_{b} "12";"14"};
{\ar_{\id_B} "12";"22"};
{\ar^{e\ppr} "14";"24"};
{\ar@{=} "22";"24"};
{\ar@{}|\circlearrowright "0";"12"};
{\ar@{}|\circlearrowright "2";"14"};
{\ar@{}|\circlearrowright "12";"24"};
\endxy
\]
in which $N\pprr\to0\to N_B$ becomes an $\sfr$-conflation. Then by \cite[Proposition~3.15]{NP}, we obtain a commutative diagram made of $\sfr$-conflations as below for some $M$.
\[
\xy
(-6,6)*+{N\pprr}="2";
(6,6)*+{N}="4";
(18,6)*+{N\ppr}="6";
(-6,-6)*+{0}="12";
(6,-6)*+{M}="14";
(18,-6)*+{N\ppr}="16";
(-6,-18)*+{N_B}="22";
(6,-18)*+{N_B}="24";
{\ar^{m} "2";"4"};
{\ar^{q} "4";"6"};
{\ar_{} "2";"12"};
{\ar^{} "4";"14"};
{\ar@{=} "6";"16"};
{\ar_{} "12";"14"};
{\ar_{r} "14";"16"};
{\ar_{} "12";"22"};
{\ar^{} "14";"24"};
{\ar@{=} "22";"24"};
{\ar@{}|\circlearrowright "2";"14"};
{\ar@{}|\circlearrowright "4";"16"};
{\ar@{}|\circlearrowright "12";"24"};
\endxy
\]
Then $r$ is an isomorphism, hence $q$ becomes an $\sfr$-inflation as desired.
\end{proof}

\begin{lem}\label{Lem9}
Let $l\co X\to Y$ be a morphism in $\Lcal$ and $r\co Y\to Z$ an $\sfr$-deflation. Then there exists the following commutative diagram
\begin{equation}\label{DiagramTA}
\xymatrix{
W\ar@{}[dr]|\circlearrowright\ar[r]^w\ar[d]_z&V\ar@{}[dr]|\circlearrowright\ar[d]^g\ar[r]^{f\ppr}&N_2\ar[d]^{g\ppr}\\
X\ar@{}[dr]|\circlearrowright\ar[d]_{z\ppr}\ar[r]^l&Y\ar@{}[dr]|\circlearrowright\ar[d]^r\ar[r]^f&N_1\ar^{r\ppr}[d]\\
Z\ppr\ar[r]_{l\ppr}&Z\ar[r]_n&N_3
}
\end{equation}
where all rows and all columns are $\sfr$-conflations and $N_{i}\in\Ncal$ for $i=1,2,3$. In particular, we have $\Rcal\ci\Lcal\se\Lcal\ci\Rcal$. 
\end{lem}
\begin{proof}
By definition, there exist $\sfr$-conflations $X\ov{l}{\lra}Y\ov{f}{\lra} N_1$ and $V\ov{g}{\lra} Y\ov{r}{\lra}Z$ with $N_1\in\Ncal$. Since $\Ncal$ is percolating, there exist $N_2\in\Ncal$, $\sfr$-deflation $f\ppr$ and $\sfr$-inflation $g\ppr$ such that $f\ci g=g\ppr\ci f\ppr$. Now we have the following diagram
\[
\xymatrix{
W\ar[r]^w&V\ar@{}[dr]|\circlearrowright\ar[d]^g\ar[r]^{f\ppr}&N_2\ar[d]^{g\ppr}\\
X\ar[r]^l&Y\ar[d]^r\ar[r]^f&N_1\ar[d]^{r\ppr}\\
\ &Z&N_3
}
\]
where all rows and all columns are $\sfr$-conflations and $N_3\in\Ncal$. 
By {\rm (ET4)}, we obtain a commutative diagram in $\C$ as below,
\[
\xy
(-18,12)*+{W}="0";
(-6,12)*+{V}="2";
(6,12)*+{N_2}="4";
(-18,0)*+{W}="10";
(-6,0)*+{Y}="12";
(6,0)*+{E}="14";
(-6,-12)*+{Z}="22";
(6,-12)*+{Z}="24";
{\ar^{w} "0";"2"};
{\ar^{f\ppr} "2";"4"};
{\ar@{=} "0";"10"};
{\ar^{g} "2";"12"};
{\ar^{d} "4";"14"};
{\ar_{g\ci w} "10";"12"};
{\ar_{c} "12";"14"};
{\ar_{r} "12";"22"};
{\ar^{e} "14";"24"};
{\ar@{=} "22";"24"};
{\ar@{}|\circlearrowright "0";"12"};
{\ar@{}|\circlearrowright "2";"14"};
{\ar@{}|\circlearrowright "12";"24"};
\endxy
\]
in which $W\ov{g\ci w}{\lra}Y\ov{c}{\lra}E$ and $N_2\ov{d}{\lra}E\ov{e}{\lra}Z$ are $\sfr$-conflations. By \cite[Lemma~3.13]{NP}, the upper right square is a weak pushout. Thus there exists a morphism $q\in\C(E,N_1)$ which makes the following diagram commutative.
\[
\xy
(-6,6)*+{V}="0";
(6,6)*+{N_2}="2";
(-6,-6)*+{Y}="4";
(6,-6)*+{E}="6";
(16,-16)*+{N_1}="8";
{\ar^{f\ppr} "0";"2"};
{\ar_{g} "0";"4"};
{\ar^{d} "2";"6"};
{\ar_{c} "4";"6"};
{\ar_{q} "6";"8"};
{\ar@/_0.80pc/_{f} "4";"8"};
{\ar@/^0.80pc/^{g\ppr} "2";"8"};
{\ar@{}|\circlearrowright "0";"6"};
{\ar@{}|\circlearrowright "6";(4,-16)};
{\ar@{}|\circlearrowright "6";(16,-4)};
\endxy
\]
Then $q$ is an $\sfr$-deflation by Lemma~\ref{Lem_ii} {\rm (2)}. Complete $q$ into an $\sfr$-conflation $Z\ppr\ov{q\ppr}{\lra}E\ov{q}{\lra}N_1$. By the dual of \cite[Lemma~3.14]{NP}, there exist morphisms $z,z\ppr$ which makes the following diagram commutative,
\[
\xy
(-18,12)*+{W}="0";
(-6,12)*+{X}="2";
(6,12)*+{Z\ppr}="4";
(-18,0)*+{W}="10";
(-6,0)*+{Y}="12";
(6,0)*+{E}="14";
(-6,-12)*+{N_1}="22";
(6,-12)*+{N_1}="24";
{\ar^{z} "0";"2"};
{\ar^{z\ppr} "2";"4"};
{\ar@{=} "0";"10"};
{\ar^{l} "2";"12"};
{\ar^{q\ppr} "4";"14"};
{\ar_{g\ci w} "10";"12"};
{\ar_{c} "12";"14"};
{\ar_{f} "12";"22"};
{\ar^{q} "14";"24"};
{\ar@{=} "22";"24"};
{\ar@{}|\circlearrowright "0";"12"};
{\ar@{}|\circlearrowright "2";"14"};
{\ar@{}|\circlearrowright "12";"24"};
\endxy
\]
such that $W\ov{z}{\lra}X\ov{z\ppr}{\lra}Z\ppr$ is an $\sfr$-conflation. Then by the dual of \cite[Proposition~3.17]{NP}, there exist morphisms $l\ppr,n$ which makes the following diagram commutative,
\[
\xy
(-6,12)*+{N_2}="2";
(6,12)*+{N_2}="4";
(-18,0)*+{Z\ppr}="10";
(-6,0)*+{E}="12";
(6,0)*+{N_1}="14";
(-18,-12)*+{Z\ppr}="20";
(-6,-12)*+{Z}="22";
(6,-12)*+{N_3}="24";
{\ar@{=} "2";"4"};
{\ar_{d} "2";"12"};
{\ar^{g\ppr} "4";"14"};
{\ar^{q\ppr} "10";"12"};
{\ar_{q} "12";"14"};
{\ar@{=} "10";"20"};
{\ar_{e} "12";"22"};
{\ar^{r\ppr} "14";"24"};
{\ar_{l\ppr} "20";"22"};
{\ar_{n} "22";"24"};
{\ar@{}|\circlearrowright "2";"14"};
{\ar@{}|\circlearrowright "10";"22"};
{\ar@{}|\circlearrowright "12";"24"};
\endxy
\]
such that $Z\ppr\ov{l\ppr}{\lra}Z\ov{n}{\lra}N_3$ is an $\sfr$-conflation.
Commutativity of $(\ref{DiagramTA})$ is immediate from the construction.

%

It remains to show $\Rcal\ci\Lcal\se\Lcal\ci\Rcal$. Suppose $r\in\Rcal$ in $(\ref{DiagramTA})$. Then $V$ is in $\Ncal$, and so is $W$ because $\Ncal$ is a thick subcategory. Thus $r\ci l=l\ppr\ci z\ppr$ belongs to $\Lcal\ci\Rcal$. 
\end{proof}

\begin{lem}\label{LemPercLR}
$\Ssc_{\Ncal}=\Lcal\circ\Rcal$ holds.
\end{lem}
\begin{proof}
This follows immediately from Lemma~\ref{Lem9}. 
\end{proof}

\begin{lem}\label{LemGeneralThick23_1}
Let
\[
\xy
(-7,6)*+{A}="0";
(-7,-6)*+{B}="2";
(2,2)*+{}="3";
(7,-6)*+{C}="4";
{\ar_{f} "0";"2"};
{\ar_{g} "2";"4"};
{\ar^{h} "0";"4"};
{\ar@{}|\circlearrowright "2";"3"};
\endxy
\]
be any commutative diagram in $\C$. The following holds.
\begin{enumerate}
\item If $f,h\in\Lcal$, then $g\in\Ssc_{\Ncal}$. 
\item If $g,h\in\Rcal$, then $f\in\Ssc_{\Ncal}$. 
\end{enumerate}
\end{lem}
\begin{proof}
By duality, it is enough to show {\rm (1)}. Let $A\ov{f}{\lra}B\to N\ov{\del}{\dra}$ and $A\ov{h}{\lra}C\to N\ppr\ov{\rho}{\dra}$ be $\sfr$-triangles with $N,N\ppr\in\Ncal$.
By the dual of \cite[Proposition~3.15]{NP}, we obtain a commutative diagram made of $\sfr$-triangles as below.
\[
\xy
(-7,7)*+{A}="2";
(7,7)*+{C}="4";
(21,7)*+{N\ppr}="6";
(35,7)*+{}="8";
(-7,-7)*+{B}="12";
(7,-7)*+{E}="14";
(21,-7)*+{N\ppr}="16";
(35,-7)*+{}="18";
(-7,-21)*+{N}="22";
(7,-21)*+{N}="24";
(-7,-34)*+{}="32";
(7,-34)*+{}="34";
{\ar^{h} "2";"4"};
{\ar^{} "4";"6"};
{\ar@{-->}^{\rho} "6";"8"};
{\ar_{f} "2";"12"};
{\ar^{} "4";"14"};
{\ar@{=} "6";"16"};
{\ar_{} "12";"14"};
{\ar_{} "14";"16"};
{\ar@{-->}_{f\sas\rho} "16";"18"};
{\ar_{} "12";"22"};
{\ar^{} "14";"24"};
{\ar@{=} "22";"24"};
{\ar@{-->}_{\del} "22";"32"};
{\ar@{-->}^{h\sas\del} "24";"34"};
{\ar@{}|\circlearrowright "2";"14"};
{\ar@{}|\circlearrowright "4";"16"};
{\ar@{}|\circlearrowright "12";"24"};
\endxy
\]
Since $h\sas\del=g\sas f\sas\del=0$, we may replace $[C\to E\to N]$ by $[C\ov{\left[\bsm 1\\0\esm\right]}{\lra}C\oplus N\ov{[0\ 1]}{\lra}N]$, to obtain
\[
\xy
(-7,7)*+{A}="2";
(7,7)*+{C}="4";
(21,7)*+{N\ppr}="6";
(35,7)*+{}="8";
(-7,-7)*+{B}="12";
(7,-7)*+{C\oplus N}="14";
(21,-7)*+{N\ppr}="16";
(35,-7)*+{}="18";
(-7,-21)*+{N}="22";
(7,-21)*+{N}="24";
(-7,-34)*+{}="32";
(7,-34)*+{}="34";
{\ar^{h} "2";"4"};
{\ar^{} "4";"6"};
{\ar@{-->}^{\rho} "6";"8"};
{\ar_{f} "2";"12"};
{\ar^{\left[\bsm1\\0\esm\right]} "4";"14"};
{\ar@{=} "6";"16"};
{\ar_(0.35){m} "12";"14"};
{\ar_{} "14";"16"};
{\ar@{-->}_{f\sas\rho} "16";"18"};
{\ar_{} "12";"22"};
{\ar^{[0\ 1]} "14";"24"};
{\ar@{=} "22";"24"};
{\ar@{-->}_{\del} "22";"32"};
{\ar@{-->}^{h\sas\del} "24";"34"};
{\ar@{}|\circlearrowright "2";"14"};
{\ar@{}|\circlearrowright "4";"16"};
{\ar@{}|\circlearrowright "12";"24"};
\endxy
\]
for some $m=\left[\bsm g\ppr\\ g\pprr\esm\right]\in\Lcal(B,C\oplus N)$. Since $[1\ 0]\in\Rcal(C\oplus N,C)$, we have $g\ppr=[1\ 0]\ci \left[\bsm g\ppr\\ g\pprr\esm\right]\in\Rcal\ci\Lcal\se\Ssc_{\Ncal}$. Since $(g-g\ppr)\ci f=0$, by {\rm (C1')} it follows that $g-g\ppr$ factors through $N$, hence $\ovl{g}=\ovl{g}\ppr$. Since $\Ssc_{\Ncal}$ satisfies {\rm (M0)}, we can apply Lemma~\ref{LemSplitN} to $\Ssc=\Ssc_{\Ncal}$ to conclude $g\in\Ssc_{\Ncal}$.
\end{proof}

\begin{lem}\label{LemGeneralThick23_2}
Let $A\ov{r}{\lra}D\ov{l}{\lra}C$ and $A\ov{f}{\lra}B\ov{g}{\lra}C$ be a pair of sequences in $\C$ satisfying $l\in\Lcal$ and $r\in\Rcal$.
Assume that
\[
\xy
(-6,6)*+{A}="0";
(6,6)*+{D}="2";
(-6,-6)*+{B}="4";
(6,-6)*+{C}="6";
{\ar^{\ovl{r}} "0";"2"};
{\ar_{\ovl{f}} "0";"4"};
{\ar^{\ovl{l}} "2";"6"};
{\ar_{\ovl{g}} "4";"6"};
{\ar@{}|\circlearrowright "0";"6"};
\endxy
\]
is commutative in $\ovl{\C}$.
Then the following holds.
\begin{enumerate}
\item If $f\in\Lcal$, then $g\in\Ssc_{\Ncal}$. 
\item If $g\in\Rcal$, then $f\in\Ssc_{\Ncal}$. 
\end{enumerate}
\end{lem}
\begin{proof}
By duality, it is enough to show {\rm (1)}.
By $\ovl{g}\ci\ovl{f}=\ovl{l}\ci\ovl{r}$, there exists $N_0\in\Ncal$ and $i\in\C(A,N_0),j\in\C(N_0,C)$ such that $l\ci r=g\ci f+j\ci i$.
If we put $B_0=B\oplus N_0$, $f_0=\left[\bsm f\\ i\esm\right]$ and $g_0=[g\ j]$, then this means the commutativity of
\begin{equation}\label{CommADB_0C}
\xy
(-6,6)*+{A}="0";
(6,6)*+{D}="2";
(-6,-6)*+{B_0}="4";
(6,-6)*+{C}="6";
{\ar^{r} "0";"2"};
{\ar_{f_0} "0";"4"};
{\ar^{l} "2";"6"};
{\ar_{g_0} "4";"6"};
{\ar@{}|\circlearrowright "0";"6"};
\endxy
\end{equation}
in $\C$. Remark that $f\in\Lcal$ implies $f_0\in\Lcal$. Thus there are $\sfr$-triangles
\[ A\ov{f_0}{\lra}B_0\ov{y}{\lra}N_1\ov{\del_1}{\dra},\quad N_2\ov{m}{\lra}A\ov{r}{\lra}D\ov{\del_2}{\dra} \]
for some $N_1,N_2\in\Ncal$.
By {\rm (ET4)}, we obtain a diagram made of $\sfr$-triangles as below.
\[
\xy
(-21,7)*+{N_2}="0";
(-7,7)*+{A}="2";
(7,7)*+{D}="4";
(-21,-7)*+{N_2}="10";
(-7,-7)*+{B_0}="12";
(7,-7)*+{E}="14";
(-7,-21)*+{N_1}="22";
(7,-21)*+{N_1}="24";
{\ar^{m} "0";"2"};
{\ar^{r} "2";"4"};
{\ar^{\del_2}@{-->} "4";(19,7)};
{\ar@{=} "0";"10"};
{\ar_{f_0} "2";"12"};
{\ar^{d} "4";"14"};
{\ar_{} "10";"12"};
{\ar_{b} "12";"14"};
{\ar@{-->}^{} "14";(19,-7)};
{\ar_{y} "12";"22"};
{\ar^{} "14";"24"};
{\ar@{=} "22";"24"};
{\ar@{-->}_{\del_1} "22";(-7,-34)};
{\ar@{-->}^{r\sas\del_1} "24";(7,-34)};
{\ar@{}|\circlearrowright "0";"12"};
{\ar@{}|\circlearrowright "2";"14"};
{\ar@{}|\circlearrowright "12";"24"};
\endxy
\]
By \cite[Lemma~3.13]{NP}, the upper right square
\[
\xy
(-6,6)*+{A}="0";
(6,6)*+{D}="2";
(-6,-6)*+{B_0}="4";
(6,-6)*+{E}="6";
{\ar^{r} "0";"2"};
{\ar_{f_0} "0";"4"};
{\ar^{d} "2";"6"};
{\ar_{b} "4";"6"};
{\ar@{}|\circlearrowright "0";"6"};
\endxy
\]
is a weak pushout. Thus by the commutativity of $(\ref{CommADB_0C})$, there exists $e\in\C(E,C)$ such that $e\ci b=g_0$ and $e\ci d=l$. Since $l,d\in\Lcal$, we obtain $e\in\Ssc_{\Ncal}$ by Lemma~\ref{LemGeneralThick23_1} {\rm (1)}. Thus $g=e\ci b\ci\left[\bsm1\\0\esm\right]\in\Ssc_{\Ncal}\ci\Rcal\ci\Lcal\se\Ssc_{\Ncal}$ follows.
\end{proof}

By using Lemma~\ref{LemGeneralThick23_2}, we can show the following.
\begin{lem}\label{LemPercPreMR1}
Let
\[
\xy
(-6,6)*+{A}="0";
(6,6)*+{D}="2";
(-6,-6)*+{B}="4";
(6,-6)*+{C}="6";
{\ar^{r} "0";"2"};
{\ar_{f} "0";"4"};
{\ar^{l} "2";"6"};
{\ar_{g} "4";"6"};
{\ar@{}|\circlearrowright "0";"6"};
\endxy
\]
be any commutative square in $\C$ with $l\in\Lcal$ and $r\in\Rcal$.
The following holds.
\begin{enumerate}
\item If $f\in\Rcal$, then $g\in\Ssc_{\Ncal}$. 
\item If $g\in\Lcal$, then $f\in\Ssc_{\Ncal}$. 
\end{enumerate}
\end{lem}
\begin{proof}
By duality, it is enough to show {\rm (1)}.
Take $\sfr$-conflations $N\ov{i}{\lra}A\ov{r}{\lra}D$ and $N\ppr\ov{i\ppr}{\lra}A\ov{f}{\lra}B$.
By the dual of Lemma~\ref{Lem9}, we obtain a commutative diagram
\[
\xy
(-12,12)*+{N_2}="0";
(0,12)*+{N\ppr}="2";
(12,12)*+{N_3}="4";
(-12,0)*+{N}="10";
(0,0)*+{A}="12";
(12,0)*+{D}="14";
(-12,-12)*+{N_1}="20";
(0,-12)*+{B}="22";
(12,-12)*+{Z}="24";
{\ar^{x} "0";"2"};
{\ar^{y} "2";"4"};
{\ar_{} "0";"10"};
{\ar_{i\ppr} "2";"12"};
{\ar^{x\ppr} "4";"14"};
{\ar^{i} "10";"12"};
{\ar^{r} "12";"14"};
{\ar_{} "10";"20"};
{\ar_{f} "12";"22"};
{\ar^{y\ppr} "14";"24"};
{\ar_{i\ppr} "20";"22"};
{\ar_{r\ppr} "22";"24"};
{\ar@{}|\circlearrowright "0";"12"};
{\ar@{}|\circlearrowright "2";"14"};
{\ar@{}|\circlearrowright "10";"22"};
{\ar@{}|\circlearrowright "12";"24"};
\endxy
\]
in $\C$ with $N_1,N_2,N_3\in\Ncal$, whose rows and columns are $\sfr$-conflations. In particular we have $y\ppr,r\ppr\in\Rcal$. 

Since $l\in\Lcal$, there is an $\sfr$-conflation $D\ov{l}{\lra}C\to N^{\prime\prime\prime}$. By {\rm (ET4)}, we obtain a commutative diagram
\[
\xy
(-12,12)*+{N_3}="0";
(0,12)*+{D}="2";
(12,12)*+{Z}="4";
(-12,0)*+{N_3}="10";
(0,0)*+{C}="12";
(12,0)*+{Z\ppr}="14";
(0,-12)*+{N^{\prime\prime\prime}}="22";
(12,-12)*+{N^{\prime\prime\prime}}="24";
{\ar^{x\ppr} "0";"2"};
{\ar^{y\ppr} "2";"4"};
{\ar@{=} "0";"10"};
{\ar^{l} "2";"12"};
{\ar^{z} "4";"14"};
{\ar_{} "10";"12"};
{\ar_{q} "12";"14"};
{\ar_{} "12";"22"};
{\ar^{} "14";"24"};
{\ar@{=} "22";"24"};
{\ar@{}|\circlearrowright "0";"12"};
{\ar@{}|\circlearrowright "2";"14"};
{\ar@{}|\circlearrowright "12";"24"};
\endxy
\]
whose rows and columns are $\sfr$-conflations.
Then we have
\[ q\ci g\ci f=q\ci l\ci r=z\ci y\ppr\ci r=z\ci r\ppr\ci f. \]
By {\rm (P3)}, 
this means that
\[
\xy
(-6,6)*+{B}="0";
(6,6)*+{Z}="2";
(-6,-6)*+{C}="4";
(6,-6)*+{Z\ppr}="6";
{\ar^{\ovl{r}\ppr} "0";"2"};
{\ar_{\ovl{g}} "0";"4"};
{\ar^{\ovl{z}} "2";"6"};
{\ar_{\ovl{q}} "4";"6"};
{\ar@{}|\circlearrowright "0";"6"};
\endxy
\]
is commutative in $\ovl{\C}$. Since $z\in\Lcal$ and $q,r\ppr\in\Rcal$ hold, we obtain $g\in\Ssc_{\Ncal}$ 
by Lemma~\ref{LemGeneralThick23_2} {\rm (2)}.
\end{proof}

\begin{prop}\label{PropPerc}
Let $\Ncal\se\C$ be a percolating thick subcategory which satisfies Condition~\ref{ConditionPerc}.
Then $\Ssc_{\Ncal}$ satisfies {\rm (M1)},\,$\ldots\,$,\,{\rm (M4)} and $\Ssc_{\Ncal}=p\iv(\ovl{\Ssc_{\Ncal}})$.
Thus the localization $\wCEs$ becomes extriangulated by Corollary~\ref{CorMultLoc}.
\end{prop}
\begin{proof}
By Lemmas~\ref{LemM3} and \ref{LemExPerc_pS}, it remains to show {\rm (M1),(M2),(M4)}. 

{\rm (M1)} Let
\[
\xy
(-7,6)*+{A}="0";
(-7,-6)*+{B}="2";
(2,2)*+{}="3";
(7,-6)*+{C}="4";
{\ar_{f} "0";"2"};
{\ar_{g} "2";"4"};
{\ar^{h} "0";"4"};
{\ar@{}|\circlearrowright "2";"3"};
\endxy
\]
be any commutative diagram in $\C$, with $h\in\Ssc_{\Ncal}$. Similarly as in the proof of Proposition~\ref{PropSatisfy}, by duality it is enough to show that $g\in\Ssc_{\Ncal}$ implies $f\in\Ssc_{\Ncal}$. Moreover, it suffices to show in the case where $g\in\Lcal$ or $g\in\Rcal$.
By Lemma~\ref{LemPercLR}, we may write $h=l\ci r$ by some $l\in\Lcal$ and $r\in\Rcal$. If $g\in\Rcal$, Lemma~\ref{LemGeneralThick23_2} {\rm (2)} shows $f\in\Ssc_{\Ncal}$.
Similarly if $g\in\Lcal$, Lemma~\ref{LemPercPreMR1} {\rm (2)} shows $f\in\Ssc_{\Ncal}$. 

{\rm (M2)} We firstly consider morphisms $X\ppr\ov{s}{\lla}X\ov{f}{\lra}Y$ in $\C$ and a factorization $s\co X\ov{r}{\lra}X\pprr\ov{l}{\lra}X\ppr$ with $l\in\Lcal, r\in\Rcal$.
We will complete the morphisms $X\ppr\ov{s}{\lla}X\ov{f}{\lra}Y$ to a commutative square.
Take an $\sfr$-conflation $N\ov{g}{\lra}X\ov{r}{\lra}X\pprr$.
By {\rm (P1')}, we may factorize 
$f\ci g$ as $N\ov{f\ppr}{\lra}N\ppr\ov{g\ppr}{\lra}Y$ with an $\sfr$-deflation $f\ppr$, an $\sfr$-inflation $g\ppr$ and $N\ppr\in\Ncal$. If we take an $\sfr$-conflation $N\ppr\ov{g\ppr}{\lra}Y\ov{r\ppr}{\lra}Y\ppr$, we obtain the following commutative solid squares
\[
\xymatrix@C=24pt@R=12pt{
N\ar[rr]^{f\ppr}\ar[dr]_g&&N\ppr\ar[dr]^{g\ppr}&&\\
&X\ar@{}[ur]|\circlearrowright\ar[rr]^f\ar[rd]^r\ar[dd]_s&&Y\ar[rd]^{r\ppr}&\\
&&X\pprr\ar@{}[l]|\circlearrowright\ar@{}[ur]|\circlearrowright\ar@{}[dr]|\circlearrowright\ar[rr]\ar[dl]^l&&Y\ppr\ar@{..>}[ld]^{l\ppr}\\
&X\ppr\ar@{..>}[rr]&&Y\pprr&
}
\]
By Lemma~\ref{LemThickFirstProperties}, we have the dotted arrows with $l\ppr\in\Lcal$ which make the whole diagram commutative.
Since $l\ppr\ci r\ppr\in\Ssc_\Ncal$, we have a desired commutative square.

Next, let $X\ppr\ov{s}{\lra}X\ov{f}{\lra}Y$ be a sequence with $s\in\Ssc_{\Ncal}$ and $f\ci s=0$.
By Lemma~\ref{LemPercLR}, there exist $r\in\Rcal(X\ppr,X\pprr)$ and $l\in\Lcal(X\pprr,X)$ such that $s=l\ci r$. By {\rm (P3)} 
we have $f\ci l=j\ci i$ for some $N\in\Ncal$, $i\in\C(X\pprr,N)$ and $j\in\C(N,Y)$. By Lemma~\ref{LemPercdefinf} we may assume that $j$ is an $\sfr$-inflation. By (C1') for $\CEs$, we obtain a commutative diagram with $N\ppr\in\Ncal$
\[
\xymatrix{
X\pprr\ar@{}[dr]|\circlearrowright\ar[r]^{l}\ar[d]_i&X\ar@{}[dr]|\circlearrowright\ar[r]\ar[d]^{f}&N\ppr\ar[d]\\
N\ar[r]_j&Y\ar[r]_{r\ppr}&Y\ppr
}
\]
in which two rows are $\sfr$-conflations.
Thus $r\ppr\ci f$ factors through $N\ppr\in\Ncal$. By Lemma~\ref{LemPercdefinf}, there exists some $N\pprr\in\Ncal$, $\sfr$-deflation $i\ppr\in\C(X,N\pprr)$ and $\sfr$-conflation $N\pprr\ov{j\ppr}{\lra}Y\ppr\ov{l\ppr}{\lra}Y\pprr$ such that $r\ppr\ci f=j\ppr\ci i\ppr$. Then $l\ppr\ci r\ppr\ci f=0$ holds for $l\ppr\ci r\ppr\in\Ssc_{\Ncal}$.
Other conditions can be checked dually.

{\rm (M4)} Let us show that $\Mcal_{\mathsf{def}}\se\Mcal$ is closed under compositions. Since $\Ssc_{\Ncal}=\Lcal\ci\Rcal$ is closed under compositions, it suffices to show that $y\ci s\ci y\ppr\in\Mcal_{\mathsf{def}}$ holds for any $\sfr$-deflations $y,y\ppr$ and any $s\in\Ssc_{\Ncal}$. By $\Ssc_{\Ncal}=\Lcal\ci\Rcal$, there are an $\sfr$-inflation $l$ and an $\sfr$-deflation $r$ satisfying $s=l\ci r$. By Lemma~\ref{Lem9}, there exist an $\sfr$-deflation $d$ and $l\ppr\in\Lcal$ satisfying $y\ci l=l\ppr\ci d$. Thus we have $y\ci s\ci y\ppr=l\ppr\ci d\ci r\ci y\ppr$ as desired. Dually, $\Mcal_{\mathsf{inf}}\se\Mcal$ is closed under compositions. 
\end{proof}

\begin{cor}\label{CorLast}
Let $\Ncal\se\C$ be a percolating thick subcategory satisfying {\rm (P2)}.
Suppose that $\Ncal$ also satisfies the following condition.
\begin{itemize}
\item $\Ker\big(\C(X,A)\ov{x\ci-}{\lra}\C(X,B)\big)\se[\Ncal](X,A)$ holds for any $X\in\C$ and any $\sfr$-inflation $x\in\C(A,B)$. 
Dually, $\Ker\big(\C(C,X)\ov{-\ci y}{\lra}\C(B,X)\big)\se[\Ncal](C,X)$ holds for any $X\in\C$ and any $\sfr$-deflation $y\in\C(B,C)$.
\end{itemize}
Then the localization $\wCEs$ obtained in Proposition~\ref{PropPerc} corresponds to an exact category.
If moreover any morphism in $\C$ is $\sfr$-admissible, then $\wC$ is an abelian category.
\end{cor}
\begin{proof}
Obviously $\Ncal$ satisfies {\rm (P3)}, hence Proposition~\ref{PropPerc} can be applied.
We remark that $\Ncal_{\Ssc_{\Ncal}}=\Ncal$ holds by Lemma~\ref{LemNSN}, hence the above condition is nothing but {\rm (i)} of Corollary~\ref{CorLocExact} with its dual.  Also, condition {\rm (ii)} of Corollary~\ref{CorLocExact} and its dual are fulfilled by {\rm (P1)} and {\rm (P1')}. Thus the resulting localization $\wCEs$ corresponds to an exact category by Corollary~\ref{CorLocExact}. The last assertion follows from Remark~\ref{RemAdm}.
\end{proof}

Let us conclude this subsection with the following construction which provides percolating subcategories satisfying Condition~\ref{ConditionPerc}. This is a generalization of Example~\ref{ExPerc} {\rm (3)}. Since the extriangulated category $\CEs$ obtained below is not triangulated nor exact in general, the resulting localization does not belong to either of Examples~\ref{ExVerdier}, \ref{ExSerreAbel}, \ref{ExTwo-sidedExact}.
A similar construction appears in \cite[Example~2.15]{GMT} for recollements of abelian categories.

\begin{prop}
Let $(T,\xi)\co\T\to\T_\Ncal$ be the Verdier quotient of a triangulated category $\T$ by a thick subcategory $\Ncal\se\T$. Here $\xi\co T\ci [1]\ov{\cong}{\ltc}[1]\ci T$ denotes a natural isomorphism, as in Remark~\ref{RemExFun}.
Let $\D\se\T_{\Ncal}$ be an extension-closed subcategory closed by direct summands, which we naturally regard as an extriangulated category $(\D,\Fbb,\tfr)$.

Let $\C\se\T$ be the extension-closed subcategory given by $\C=T\iv(\D)$, which we also regard as an extriangulated category $\CEs$. 
Let $(F=T|_{\C},\phi)\co\CEs\to(\D,\Fbb,\tfr)$ be the exact functor induced by restricting $T$, for which $\phi_{C,A}\co\E(C,A)\to\F(FC,FA)$ is given by $\phi_{C,A}(\del)=\xi_A\ci T(\del)$ for any $A,C\in\C$ and $\del\in\E(C,A)$.
Then the following holds.
\begin{enumerate}
\item $\Ncal\se\C$ is a percolating subcategory of $\CEs$ which satisfies Condition~\ref{ConditionPerc} and $\Ssc_{\Ncal}=F\iv(\Iso(\D))$.
\item The exact functor $(\wt{F},\wt{\phi})\co\wCEs\to(\D,\F,\tfr)$ obtained by Corollary~\ref{CorMultLoc} is an equivalence of extriangulated categories in the sense of {\rm (2)} in Proposition~\ref{PropExEq}.
\item If moreover $(\D,\Fbb,\tfr)$ corresponds to an exact category, then $\Ncal$ also satisfies the assumption of Corollary~\ref{CorLast}.
\end{enumerate}
\end{prop}
\begin{proof}
{\rm (1)} It is obvious that $\Ncal$ is a thick subcategory in $\CEs$. 
To show that $\Ncal\se\C$ is percolating, let $N\in\Ncal$ be any object, and $x\in\C(N,A)$ be any morphism. If we complete $x$ into a distinguished triangle $N\ov{x}{\lra}A\ov{y}{\lra}B\lra N[1]$ in $\T$, then it gives an $\sfr$-conflation $N\ov{x}{\lra}A\ov{y}{\lra}B$. This shows that $x$ is an $\sfr$-inflation. Dually, any morphism $y\in\C(A,N)$ to $N\in\Ncal$ is an $\sfr$-deflation. Thus in particular $\Ncal\se\C$ is percolating.

We know that $\Ncal$ satisfies {\rm (P2)} as seen in Remark~\ref{RemP2} {\rm (2)}. We can also easily check that $\Ssc_{\Ncal}=\Rcal=\Lcal$ holds in $\C$. Indeed, any morphism $f\in\C(A,B)$ belongs to $\Ssc_{\Ncal}$ if and only if it can be completed into a distinguished triangle $N\to A\ov{f}{\lra}B\to N[1]$ in $\T$ for some $N\in\Ncal$, if and only if it satisfies $F(f)\in\Iso(\D)$. Thus $\Ncal$ also satisfies {\rm (P3$^+$)}.

{\rm (2)} Remark that Proposition~\ref{PropPerc} shows that Corollary~\ref{CorMultLoc} can be applied to $\Ssc_{\Ncal}$. By Proposition~\ref{PropExEq}, it suffices to show that $\wt{F}$ is an equivalence and $\wt{\phi}$ is a natural isomorphism.
Since $T$ is essentially surjective, obviously so is $\wt{F}$. By construction, if $s\in\T(X,Y)$ is a morphism with $T(s)\in\Iso(\T_{\Ncal})$, then $X\in\C$ holds if and only if $Y\in\C$. This property is enough to conclude that $\wt{F}\co\wt{\C}\to\D$ is fully faithful. 

Let us show that $\wt{\phi}_{C,A}\co\wt{\Ebb}(C,A)\to\Fbb(FC,FA)$ is an isomorphism for any $A,C\in\C$. By definition, we have $\Fbb(FC,FA)=\T_{\Ncal}(FC,(FA)[1])$. Any element $\al\in\T_{\Ncal}(FC,(FA)[1])$ can be expressed as $\al=\xi_A\ci T(\sig)\ci F(t)\iv$ for some $C\ppr\in\C$ and $\sig\in\T(C\ppr,A[1])=\Ebb(C\ppr,A)$. It is straightforward to show that the map
\[ \lam\co \Fbb(FC,FA)\to\wE(C,A)\ ;\ \al\mapsto [\FR{\ovl{t}}{\ovl{\sig}}{\id}] \]
is well-defined and gives the inverse of $\wt{\phi}_{C,A}$.


{\rm (3)} Assume that $(\D,\Fbb,\tfr)$ corresponds to an exact category.
Suppose that $x\ci f=0$ holds for an $\sfr$-inflation $x\co A\to B$ and a morphism $f\co X\to A$ in $\C$. Since $F(x)$ is a monomorphism in $\D$ we get $F(f)=0$, which implies that $f$ factors through an object in $\Ncal$. Dually for $\sfr$-deflations.
\end{proof}

\end{document}